\title {A Counting Lemma 
for Somewhat Restricted $3$-APs}
\author{Amey Bhangale\thanks{Department of Computer Science and Engineering, University of California, Riverside. Supported by the Hellman Fellowship award and NSF CAREER award 2440882.}
	\and
	Subhash Khot\thanks{Department of Computer Science, Courant Institute of Mathematical Sciences, New York University. Supported by
		the NSF Award CCF-1422159, NSF CCF award 2130816, and the Simons Investigator Award.}
	\and
    Yang P. Liu\thanks{School of Mathematics, Institute for Advanced Study, Princeton, NJ. This material is based upon work supported by the National Science Foundation 
under Grant No. DMS-1926686}
    \and 
	Dor Minzer\thanks{Department of Mathematics, Massachusetts Institute of Technology. Supported by NSF CCF award 2227876 and NSF CAREER award 2239160.}}
\date{\vspace{-5ex}}

\documentclass[11pt]{article}
\usepackage{times}
\usepackage[T1]{fontenc}
\usepackage{amssymb}
\usepackage{amsmath}
\usepackage{amsthm}
\usepackage{bm}
\usepackage{xcolor}
\usepackage{fullpage}

\usepackage{hyperref}
\hypersetup{hypertexnames=false}

\usepackage{thmtools}
\usepackage{thm-restate}

   \newtheorem{thm}{Theorem}[section]
   \newtheorem{conj}[thm]{Conjecture}
   \newtheorem{lemma}[thm]{Lemma}
   
   \newtheorem{claim}[thm]{Claim}
   \newtheorem{fact}[thm]{Fact}
   \newtheorem{remark}[thm]{Remark}
   \newtheorem{definition}[thm]{Definition}
   
   \newtheorem{proposition}[thm]{Proposition}

\newcommand\E{\mathbb{E}}
\newcommand\card[1]{\left| {#1} \right|}

\newcommand\sett[2]{\left\{ \left. #1 \;\right\vert #2 \right\}}

\newcommand\set[1]{{\left\{ #1 \right\}}}
\newcommand\Prob[2]{{\Pr_{#1}\left[ {#2} \right]}}

\newcommand\cExpect[3]{{\mathop{\mathbb{E}}_{#1}\left[ \left. #3 \;\right\vert #2 \right]}}

\newcommand\norm[1]{\| #1 \|}

\newcommand\Expect[2]{{\mathop{\mathbb{E}}_{#1}\left[ {#2} \right]}}

\renewcommand{\lll}{\lesssim}
\newcommand\spn{{\sf span}_{\mathbb{N}}}
\newcommand\inner[2]{\langle{#1},{#2}\rangle}
\newcommand\eps{\varepsilon}

\renewcommand\geq{\geqslant}
\renewcommand\leq{\leqslant}

\newcommand{\rom}[1]{\uppercase\expandafter{\romannumeral #1\relax}}
\newcommand{\V}[1]{{#1}}

\begin{document}

\maketitle
\begin{abstract}
    For a prime $p\geq 3$, a somewhat restricted $3$-AP in $\mathbb{F}_p^n$ is a triplet $(x,x+a,x+2a)$, where $x\in\mathbb{F}_p^n$ and $a\in \{0,1,2\}^n$. We prove a counting lemma for somewhat restricted $3$-APs in dense sets in $\mathbb{F}_p^n$. More precisely, we prove that for all $\alpha>0$, there exists $\beta>0$, such that for sufficiently large $n$, if a set $A\subseteq \mathbb{F}_p^n$ has density at least $\alpha$, then it contains at least $\beta$ fraction of all somewhat restricted $3$-APs. 
    %This result is a counting analog of the existence result of [Bhangale, Khot, Minzer, 2024], who showed that any such $A$ must contain at least one non-trivial restricted $3$-AP. 
    %but in contrast to the case of standard $3$-APs, it is not known whether an existence result implies a counting lemma.

    Our proof builds on recently developed machinery from [Bhangale, Khot, Minzer, 2026]. Our main new ingredient is an arithmetic regularity lemma for patterns such as somewhat restricted 3-APs. This result is in the spirit of arithmetic regularity lemmas from the theory of Gowers uniformity norms [Green, Tao, 2010] and may be of independent interest.
\end{abstract}
\section{Introduction}

The emergence of arithmetic patterns in dense sets, first studied by Roth~\cite{Roth}, has proven to be a fundamental problem relating distinct areas in mathematics. Roth's theorem asserts that for all $\eps>0$, if $A\subseteq \{1,\ldots,N\}$ has density at least $\eps$ and $N$ is sufficiently large, then $A$ must contain a triplet of the form $x,x+a,x+2a$ with $a\neq 0$. Variants and extensions of Roth's theorem have received much attention over the years, and their study led to the development of important tools in combinatorics, analysis, ergodic theory, probability and more. Below are a few of these results:
\begin{enumerate}
    \item {\bf Quantitative bounds:} the precise density needed for the conclusion to hold has been the subject of intense study~\cite{Bourgain1999,Bourgain2008,Sanders2011,BloomSisask2019,BloomSisask2020,KelleyMeka2023}. By now, it is known that the answer is quasi-polynomial~\cite{KelleyMeka2023,BloomSisask2023}, matching the shape of the best known construction~\cite{Behrend1946,ElsholtzHunterProskeSauermann2024}.
    \item {\bf Longer progressions:} Szemer\'edi's~\cite{Szemeredi1975} extended Roth's theorem to progressions of arbitrarily large constant length. Further study of this problem led to the
    development of higher-order Fourier analysis~\cite{Gowers1998,Gowers2001} and of graph and hypergraph regularity~\cite{Szemeredi1978,RuzsaSzemeredi1978,Gowers2007}.
    \item {\bf More general patterns:} the study of arithmetic progressions has motivated the broader question of the occurrence of general patterns in sufficiently dense sets. A concrete class of examples are polynomial progressions~\cite{BergelsonLeibman1996}, which are patterns of the form $x+P_1(a),\ldots,x+P_k(a)$ where $P_1,\ldots,P_k$ are polynomials. Other examples include multi-dimensional patterns~\cite{FurstenbergKatznelson1978,Gowers2007} (such as $x+ia_1+ja_2$ for $i,j=0,1,2$), and combinatorial lines~\cite{FurstenbergKatznelson1991,Polymath2012,BKLMDHJ}.
\end{enumerate}

% \begin{enumerate}
%     \item {\bf Quantitative bounds:} the precise density needed for the conclusion to hold has been the subject of intense study~\cite{}. By now, it is known that the answer is quasi-polynomial~\cite{}, matching the shape of the best known example~\cite{}.
%     \item {\bf Longer progressions:} Szemer\'edi's Theorem~\cite{} is an extension of Roth's theorem for progressions of arbitrarily large constant length $k\in\mathbb{N}$. This is the starting point for the rich theory of higher-order Fourier analysis~\cite{}, and for graph and hypergraph regularity lemma~\cite{}.
%     \item {\bf More general patterns:} the study of arithmetic progressions has motivated the broader question of appearances of general patterns in sufficiently dense sets. Concrete examples include polynomial progressions, which are patterns of the form $x+P_1(a),\ldots,x+P_k(a)$ where $P_1,\ldots,P_k$ are polynomials, multi-dimensional patterns (such as $x+ia_1+ja_2$ for $i,j=0,1,2$), and combinatorial lines.
% \end{enumerate}
The current work fits in the third item above, and we study restricted forms of $3$-APs in the finite field setting, in which the common difference must come from a set much smaller than the domain.

\subsection{The Finite Field Setting}
In the finite field setting, one considers the problem of finding arithmetic patterns in a dense set $A\subseteq \mathbb{F}_p^n$. Here and throughout, $p$ is thought of as a prime of constant size and $n$ is thought of as tending to infinity.  The theory in this setting is very close  to the theory in the integer setting, and it avoids some of the technical difficulties (such as dealing with Bohr sets). In fact, the tools in the two settings are close analogs of one another, with the exception of the polynomial method~\cite{CrootLevPach2017,EllenbergGijswijt2017}, which only seems applicable for some problems in the finite field setting. 

Meshulam~\cite{Meshulam}  proved an analog of Roth's theorem in the finite field setting, following a similar density increment, Fourier-analytic strategy.
Our work is mostly concerned with studying more restricted patterns than the ones addressed in Meshulam's work, defined as follows:
\begin{definition}
    Fix a prime $p\geq 3$. 
    \begin{enumerate}
        \item A somewhat restricted $3$-AP in $\mathbb{F}_p^n$ is a triplet of the form $(x,x+a,x+2a)$ such that $x\in\mathbb{F}_p^n$ and 
    $a\in\{0,1,2\}^n\setminus\{\vec{0}\}$.
    \item A restricted $3$-AP in $\mathbb{F}_p^n$ is a triplet of the form $(x,x+a,x+2a)$ such that $x\in\mathbb{F}_p^n$ and 
    $a\in\{0,1\}^n\setminus\{\vec{0}\}$.
    \end{enumerate}
    
\end{definition}
One may consider two distinct problems with respect to restricted $3$-APs and somewhat restricted $3$-APs: 
\begin{enumerate}
    \item {\bf Existence:} what density guarantees that $A$ must contain at least one pattern?
    \item {\bf Counting:} if $A$ has density $\alpha>0$, does it have to contain $\Omega_{\alpha}(1)$ fraction the patterns (provided $n$ is sufficiently large)?
\end{enumerate}
The Fourier analytic machinery employed by Meshulam~\cite{Meshulam} and subsequent works does not seem to apply well to this setting, and these problems were highlighted in~\cite{Green,Mos}. For the first problem, the density Hales Jewett theorem implies that any constant density must contain at least one pattern. The quantitative bounds this argument gives were quite weak, and only recently effective bounds were given. For both problems, it is now known~\cite{BKM3AP,BKLMDHJ} that if  $A\subseteq\mathbb{F}_p^n$ contains no restricted $3$-AP, then it has density at most $1/(\log\log\log n)^{c_p}$ where $c_p>0$ is a constant. As far as we know it may be the case that $A$ must have exponentially small measure, and a recent result~\cite{conlon2026note} shows this is indeed the case for the richer class of progressions of the form $(x,x+a,x+2a)$ where $a\in \{0,1\ldots, (p+1)/2\}^n$.

The main objective of the current paper is to answer the counting problem above for somewhat restricted $3$-APs. Note that for standard $3$-APs, an existence result immediately yields a counting result, and we recall this argument next. Suppose that we know that for all $\alpha>0$ there is $n_0(p,\alpha)$ such that if $A\subseteq \mathbb{F}_p^n$ has density at least $\alpha$, then it has a $3$-AP. Now fix $A\subseteq \mathbb{F}_p^n$ of density at least $\alpha$, denote
\[
{\sf 3AP}(A,W) =
\card{\{x,a\in W~|~x,x+a,x+2a\in A\}},
\]
and take $n_0 = n_0(\alpha/2,p)$. Taking an affine subspace $W\subset\mathbb{F}_p^n$ of dimension $n_0$ randomly, by an averaging argument we have that $\card{A\cap W}\geq \frac{1}{2}\alpha \card{W}$ with probability at least $\alpha/2$, and then by assumption $A\cap W$ must contain a $3$-AP. Thus, 
${\sf 3AP}(A,W)\geq 1$ with probability at least $\alpha/2$. On the other hand, we have that $\Expect{W}{{\sf 3AP}(A,W)}= {\sf 3AP}(A,\mathbb{F}_p^n)\cdot p^{2(n_0-n)}$, hence we conclude that ${\sf 3AP}(A,\mathbb{F}_p^n)\geq p^{2n}\frac{\alpha}{2p^{2n_0}}=\Omega_{\alpha,p}(p^{2n})$.
Attempting to run such an argument in the setting of somewhat restricted $3$-APs or restricted $3$-APs fails because these patterns are basis dependent. This raises the question of whether a counting lemma for such patterns still holds. Our main result, stated in the next section, asserts that the answer is positive in the case of somewhat restricted $3$-APs.

\subsection{Main Results}
The main result of this paper is the following statement.
\begin{thm}\label{thm:main}
    For all prime $p\geq 3$ and $\alpha>0$, there exists $\beta>0$, such that for large enough $n$, if $A\subseteq \mathbb{F}_p^n$ has density at least $\alpha$, then
    \[
    \Prob{\substack{x\in \mathbb{F}_p^n\\ a\in \{0,1,2\}^n}}{x,x+a,x+2a\in A}\geq \beta.
    \]
\end{thm}
We do not give explicit quantitative bounds for $\beta$ in Theorem~\ref{thm:main}. 
At best, the argument given here yields a bound for $\beta$ that is inverse-tower-type in $1/\alpha$. 
This is because, as we discuss in Section~\ref{sec:techniques}, our argument uses an arithmetic regularity lemma in the spirit of~\cite{Green2005SzemerediRegularityAbelianGroups,GreenTao2010ArithmeticRegularity}, which often leads to weak quantitative bounds.
We believe that the statement of Theorem~\ref{thm:main} should be true for restricted $3$-APs, namely where $\{0,1,2\}^n$ is replaced with $\{0,1\}^n$, but our methods cannot handle this yet. We discuss one of the key obstructions for this in Section~\ref{sec:discussion}.

Theorem~\ref{thm:main} follows from our general result, which we now state. We need the following definition of embeddability of a distribution into an abelian group.
\begin{definition}
    Let $\Sigma,\Gamma,\Phi$ be finite alphabets, let $\mu$ be a distribution over $\Sigma\times\Gamma\times\Phi$, and let $(G,+)$ be an Abelian group. We say that $\mu$ has a $(G,+)$-embedding if there are maps $\sigma\colon \Sigma\to G$, $\gamma\colon \Gamma\to G$, $\phi\colon \Phi\to G$ not all constant such that for all $(x,y,z)\in{\sf supp}(\mu)$ we have
    \[
    \sigma(x)+\gamma(y)+\phi(z) = 0_G.
    \]
\end{definition}
With this notion in mind, we state the general theorem below.  Theorem~\ref{thm:main} follows easily from this theorem (it is a routine exercise to check that somewhat restricted $3$-APs do not admit $(\mathbb{Z},+)$-embeddings, see~\cite{BKM3AP}, for instance).

\begin{thm}\label{thm:main_2}
    Let $\Sigma$ be an alphabet of size at most $m$, and let $\mu$ be a distribution over $\Sigma^3$ and $\nu$ be a distribution over $\Sigma$ such that 
    \begin{enumerate}
        \item The marginal of $\mu$ on each coordinate is equal to $\nu$, and $(x,x,x)\in {\sf supp}(\mu)$.
        \item The probability of each atom in $\mu$ is at least $s>0$.
        \item The distribution $\mu$ has no $(\mathbb{Z},+)$-embedding.
    \end{enumerate}
    Then for all $\alpha>0$, there exists $\beta = \beta(\alpha,m,s)>0$ such that for sufficiently large $n$, if $A\subseteq \Sigma^n$ has 
    $\nu^{\otimes n}(A)\geq \alpha$, then
    \[
    \Prob{(x,y,z)\sim\mu^{\otimes n}}{x,y,z\in A}\geq \beta.
    \]
\end{thm}
We remark that the first two conditions in Theorem~\ref{thm:main_2} are necessary for the result to hold, but we believe that the third condition can be relaxed.

\subsection{Techniques}\label{sec:techniques}
In this section we give an overview of the proof of Theorem~\ref{thm:main}. We first outline a direct argument for the counting lemma for standard $3$-APs, and then explain how to generalize it to our setting.

\subsubsection{A Counting Lemma for Standard \texorpdfstring{$3$}{3}-APs}
Let $A\subseteq \mathbb{F}_p^n$ have measure $\alpha>0$ which we think of as a small constant, and denote for ease of notation $f(x) = 1_A(x)$. 
A standard Fourier analytic computation shows that if all non-trivial Fourier coefficients of $f$ are at most $\eps$ in absolute value, then 
$\Expect{x,a\in\mathbb{F}_p^{n}}{f(x)f(x+a)f(x+2a)}\geq \alpha^3 - \eps\alpha$. 
Therefore, if all Fourier coefficients of $f$ are at most $\alpha^2/2$ in absolute value, the conclusion immediately follows. Otherwise, one could pick a heavy Fourier character $0\neq \gamma\in\mathbb{F}_p^n$ and show that the density of $f$ increases on one of the hyperplanes defined by $\gamma$, and iterate the argument there. This works in the case of standard $3$-APs (because a co-dimension $1$ subspace contains a constant fraction of $3$-APs), but fails in the case of somewhat restricted $3$-APs because the density increment argument we know~\cite{BKM3AP} passes down to a subspace of super-constant co-dimension. Thus, we wish to come up with an argument that avoids this type of iterative approach.

To do so, upon finding $\gamma$, one can define the averaging operator 
$\mathrm{T}_{\gamma}\colon L_2(\mathbb{F}_p^n)\to L_2(\mathbb{F}_p^n)$ by
\[
\mathrm{T}_{\gamma}f(x)
=\cExpect{y\in\mathbb{F}_p^n}
{\chi_{\gamma}(y) = \chi_{\gamma}(x)}{f(y)}.
\]
In words, $\mathrm{T}_{\gamma}$ averages over the cosets of the linear hyperplane defined by $\gamma$. For a single $\gamma$, one can observe that $\mathrm{T}_{\gamma}f(x) = \sum\limits_{t\in\mathbb{F}_p}\widehat{f}(t\gamma)\chi_{t\gamma}(x)$, so $\widehat{f}(\gamma)$ being large in absolute value means that the $2$-norm of $(I-\mathrm{T}_{\gamma})f$ is noticeably smaller than that of $f$. Iterating, one finds a collection of characters $\mathcal{C} = \{\gamma_1,\ldots,\gamma_r\}$, such that defining
\[
\mathrm{T}_{\mathcal{C}}f(x)
=\cExpect{y\in\mathbb{F}_p^n}
{\chi_{\gamma}(y) = \chi_{\gamma}(x)~\forall\gamma\in\mathcal{C}}{f(y)},
\]
all Fourier coefficients of $(\mathrm{I}-\mathrm{T}_{\mathcal{C}})f$ are at most $O(1/\sqrt{r})$ in absolute value. Thus, one can write
\begin{equation}\label{eq:intro}
\Expect{x,a\in\mathbb{F}_p^n}{f(x)f(x+a)f(x+2a)}
\geq
\Expect{x,a\in\mathbb{F}_p^n}{\mathrm{T}_{\mathcal{C}}f(x)\mathrm{T}_{\mathcal{C}}f(x+a)\mathrm{T}_{\mathcal{C}}f(x+2a)}
-O\left(\frac{1}{\sqrt{r}}\right),
\end{equation}
so the task now is to analyze the right-hand side and show it is noticeable. Because the function 
$\mathrm{T}_{\mathcal{C}}f$ is very structured and has only $O_{r}(1)$ many non-zero Fourier coefficients, this expectation seems more feasible to analyze. Indeed, taking $W$ to be the linear subspace of codimension at most $r$ defined by the vectors in $\mathcal{C}$, it is easy to observe that:
\begin{itemize}
    \item Taking a random shift $v+W$ of $W$, with probability at least $\alpha/2$, the average of $f$ inside $v+W$ is at least $\alpha/2$. Noting that $\mathrm{T}_{\mathcal{C}} f$ is constant on each such affine shift, we get that with probability at least $\alpha/2$ over the choice of $v$, the value of $\mathrm{T}_{\mathcal{C}} f$ on all points of $v+W$ is at least $\alpha/2$.
    \item Sampling $x,a\in\mathbb{F}_p^n$, the probability that $x,x+a,x+2a$ are all in $v+W$ is $p^{-2{\sf codim}(W)}$, which is at least $p^{-2r}$.
\end{itemize}
With these two items in mind, one can get that
\[
\Expect{x,a\in\mathbb{F}_p^n}{\mathrm{T}_{\mathcal{C}}f(x)\mathrm{T}_{\mathcal{C}}f(x+a)\mathrm{T}_{\mathcal{C}}f(x+2a)}\geq p^{-2r}\left(\frac{\alpha}{2}\right)^{4},
\]
which initially seems good, but upon further inspection is too weak.\footnote{A more careful argument gives $p^{-r}\alpha^{3}$, but this does not make a difference.} Indeed, this bound is quantitatively weaker than the error term on the right hand side of~\eqref{eq:intro}, so it does not give any lower bound on $\Expect{x,a\in\mathbb{F}_p^n}{f(x)f(x+a)f(x+2a)}$.

At a high-level, the attempt above can be thought of as decomposing $f$ into a structure function $\mathrm{T}_{\mathcal{C}} f$, and a Fourier-quasi random function 
$(\mathrm{I}-\mathrm{T}_{\mathcal{C}}) f$, and ultimately the issue is that there is an unfavorable tradeoff between the quasi-randomness parameter of $(\mathrm{I}-\mathrm{T}_{\mathcal{C}}) f$ (which governs the error term) and the size of $\mathcal{C}$ (which governs the codimension of $W$). It turns out that one can break this dependency using stronger decompositions known as arithmetic regularity lemmas, appearing implicitly in~\cite{bourgain1986szemeredi} for $U^2$-norm and explicitly in~\cite{Green2005SzemerediRegularityAbelianGroups,GreenTao2010ArithmeticRegularity}. The idea is that by iterating the above process several times, fixing an arbitrary decay function $w\colon (0,1)^2\to (0,1)$, one can find two collections $\mathcal{C}$ and $\mathcal{C}'$ such that:
\begin{itemize}
    \item All Fourier coefficients of 
    $(\mathrm{I}-\mathrm{T}_{\mathcal{C}}) f$ are at most $\eps$ in absolute value.
    \item All Fourier coefficients of $(\mathrm{I}-\mathrm{T}_{\mathcal{C}'}) f$ are at most $\eps'=w(\eps,\card{\mathcal{C}}^{-1})$ in absolute value.
    \item The functions 
    $(\mathrm{I}-\mathrm{T}_{\mathcal{C}}) f$ and 
    $(\mathrm{I}-\mathrm{T}_{\mathcal{C}'}) f$ are $\eps$-close in $L_2$-norm.
\end{itemize}
With this in mind, one can replace~\eqref{eq:intro} with 
\begin{equation}\label{eq:intro_2}
\Expect{x,a\in\mathbb{F}_p^n}{f(x)f(x+a)f(x+2a)}
\geq
\Expect{x,a\in\mathbb{F}_p^n}{\mathrm{T}_{\mathcal{C}'}f(x)\mathrm{T}_{\mathcal{C}'}f(x+a)\mathrm{T}_{\mathcal{C}'}f(x+2a)}
-O(\eps').
\end{equation}
The point now is that one can still take the linear subspace $W$ as above with respect to $\mathcal{C}$ (and not to $\mathcal{C}'$). While it is no longer the case that $\mathrm{T}_{\mathcal{C}'}f$ is constant on cosets of $W$, we know it is close to being constant on a typical coset, thanks to the $L_2$-closeness to  $\mathrm{T}_{\mathcal{C}}f$. This is sufficient to repair the argument above.

\subsubsection{A Counting Lemma for Somewhat Restricted \texorpdfstring{$3$}{3}-APs}
We now explain the way our argument for somewhat restricted $3$-APs fits the above template. 

The first complication is that the direct Fourier-analytic computation showing that small Fourier coefficients imply counting no longer works.
This issue was already studied in~\cite{BKM3AP,BKM4}, and morally speaking, it is known that if $f$ has correlation at most $\eta$ with all functions of the form $\chi_{\gamma}(x)\cdot L(x)$ 
where $\gamma\in\mathbb{F}_p^n$ and $L\colon \mathbb{F}_p^n\to\mathbb{C}$ has $\norm{L}_2=1$ and Fourier analytic degree at most $d$, then
\[
\Expect{\substack{x\in\mathbb{F}_p^n\\a\in\{0,1,2\}^n}}{f(x)f(x+a)f(x+2a)}
\geq \alpha^3 - \delta(\eta,d),
\]
where $\delta(\eta,d)$ goes to $0$ as $\eta\rightarrow 0$ and $d\rightarrow \infty$; see Theorem~\ref{thm:csp4} for a more precise statement.

The second complication that arises is that one has to come up with an analog of the averaging operator $\mathrm{T}_{\mathcal{C}}$ above. This was done in~\cite{BKM5}, and they introduce noise operators $\mathrm{T}_{\mathcal{C},1-\eps}$. Morally speaking, for $x\in\mathbb{F}_p^n$, the value of $\mathrm{T}_{\mathcal{C},1-\eps}f(x)$ can be thought of as the average of $f(y)$ over points $y$ such that $\chi_{\gamma}(y) = \chi_{\gamma}(x)$ for all $\gamma\in\mathcal{C}$, and additionally $x$ and $y$ agree on roughly $(1-\eps)$-fraction of the coordinates. A key feature of this  operator is that it acts like the identity operator on functions of the form $\chi_{\gamma}'(x) L'(x)$ with $\gamma\in \spn(\mathcal{C})$ and ${\sf deg}(L')\ll 1/\eps$, which suggests it is the correct analog of $\mathrm{T}_{\mathcal{C}}$ above. Indeed,~\cite{BKM5} prove a regularity lemma with respect to this operator, showing that for all $\eta>0$, given $f$, one can find 
$\mathcal{C}$ of size $O_{\eta}(1)$ and $\eps = \Omega_{\eta}(1)$, such that 
$(\mathrm{I}-\mathrm{T}_{\mathcal{C},1-\eps})f$ has correlation at most $\eta$ with all functions of the form $\chi_{\gamma}(x) L(x)$. We remark that technically, the operator $\mathrm{T}_{\mathcal{C},1-\eps}$ is significantly harder to work with compared to $\mathrm{T}_{\mathcal{C}}$, and this is partly because it does not behave as nicely as $\mathrm{T}_{\mathcal{C}}$ with respect to the Fourier decomposition.

With this in mind, attempting to reproduce the initial failed attempt, one could start with the inequality
\[
\Expect{\substack{x\in\mathbb{F}_p^n\\a\in\{0,1,2\}^n}}{f(x)f(x+a)f(x+2a)}
\geq 
\Expect{\substack{x\in\mathbb{F}_p^n\\a\in\{0,1,2\}^n}}{\mathrm{T}_{\mathcal{C},1-\eps}f(x)\mathrm{T}_{\mathcal{C},1-\eps}f(x+a)\mathrm{T}_{\mathcal{C},1-\eps}f(x+2a)}
-\delta(\eta,1/\eps),
\]
where $\lim_{\eps,\eta\rightarrow 0}\delta(\eta,1/\eps)=0$, and then analyze the expectation on the right hand side. Towards this end, one may use the fact from~\cite{BKM5} that the function $\mathrm{T}_{\mathcal{C},1-\eps}f$ can be $\xi$-approximated in $L_2$-distance by a function of the form
\[
\sum\limits_{\gamma\in{\sf span}(\mathcal{C})}\chi_{\gamma}(x) L_{\gamma}(x),
\]
where $L_{\gamma}$ are all of degree and $2$-norm at most $O_{\xi,\eta}(1)$. Intuitively, this says that after we restrict $\mathrm{T}_{\mathcal{C},1-\eps}f$ down to a coset of $W$, we will get a low-degree function (instead of a constant function). This is of course more complicated to analyze, but luckily (after appropriate preprocessing), one could use the invariance principle of~\cite{MOO} to relate the resulting expectation to a similar looking expectation over Gaussian space, which can be lower bounded via results of~\cite{ChenDafnisPaouris2015HolderGaussian}. We remark that to actually make such an argument go through, one needs to ensure that the collection $\mathcal{C}$ has a large rank (roughly asserting that no non-trivial combination of vectors from $\mathcal{C}$ has small Hamming weight). This leads to significant complications which ignore in this overview. We also remark that the ability to solve the Gaussian problem we get in the end is the primary reason we are only able to handle somewhat restricted $3$-APs at the moment (as opposed to restricted $3$-APs), and we discuss this further in Section~\ref{sec:discussion}.

With the initial failed attempt recovered, the main contribution of this paper is to show an appropriate analog of the arithmetic regularity lemma for operators of the form $\mathrm{T}_{\mathcal{C},1-\eps}$, and this is achieved in Lemma~\ref{lem:arithmetic_reg}. We then boost it to Lemma~\ref{lem:arithmetic_reg_high_rank}, which makes the same assertion with additional quasirandomness properties of the collection $\mathcal{C}$ in the spirit of rank discussed above (see Definitions~\ref{def:rank} and~\ref{def:softrank}). Finally, in Section~\ref{sec:proof_of_main} we show how to use our arithmetic regularity lemma in conjunction with the invariance principle to prove Theorem~\ref{thm:main_2}.

\paragraph{Statement of AI use:} we used ChatGPT 5.5 during the polishing of this write-up, and to write the proofs from Section~\ref{sec:GRH_complex}.
All other mathematical content in this paper is due to the authors.
\section{Preliminaries}
\paragraph{Notations.} 
    For a vector $x\in \Sigma^n$ and a subset $I\subseteq [n]$ of coordinates, we denote by
 $x_I$ the vector in $\Sigma^{I}$ which results by dropping from $x$ all coordinates outside $I$. 
 %We denote by $x_{\bar{I}}$ the vector in $\Sigma^{n-\card{I}}$ resulting from dropping from $x$ all coordinates from $I$.
% For $i\in [n]$ we denote by $x_{-i}$ the vector in $\Sigma^{n-1}$ resulting from dropping the $i$th coordinate of $x$. For $I\subseteq[n]$, $a\in \Sigma^{I}$
% and $b\in \Sigma^{n-\card{I}}$ we denote by $(x_I = a, x_{\bar{I}} = b)$ the point in $\Sigma^{n}$ whose $I$-coordinates
% are filled according to $a$, and whose $\overline{I}$-coordinates are filled according to $b$. For two strings $x,y\in \Sigma^n$
% we denote by $\Delta(x,y)$ the Hamming distance between $x$ and $y$, that is, the number of coordinates $i\in [n]$ such that $x_i\neq y_i$. 
We denote by ${\bf i}$ the complex root of $-1$, by $\overline{a}$ the complex conjugate of the number $a\in\mathbb{C}$, and by $\mathbb{D}\subseteq\mathbb{C}$ the set of complex numbers of absolute value $1$.
We denote by $I\subseteq_{\rho}[n]$ a 
random subset of $[n]$ in which we include each $i\in[n]$ independently with probability $\rho$. We will use standard big-$O$ notations, and write $A = O(B)$ if $A\leq C\cdot B$ for an absolute constant $C$, and $A = \Omega(B)$ if $A\geq C\cdot B$ for an absolute constant $C>0$. If
there is a dependency of the hidden constant on some auxiliary parameter, say $m$, we denote $A = O_m(B)$
and $A = \Omega_m(B)$.

Our arguments will use sufficiently rapid decay functions $w\colon (0,1)\to(0,1)$, which by default will always be monotonically increasing and satisfying $w(\eps)\leq \eps$. 
We say that a tuple of scales 
$\rho_0,\ldots,\rho_{\ell}$ $w$-separated, and denote
\[
\rho_\ell\ll_w\rho_{\ell-1}\ll_w\ldots\ll_w\rho_0,
\]
if $\rho_{j+1}\leq w(\rho_j)$ for all $0\leq j<\ell$.

\subsection{Product Functions}
\begin{definition}
    Let $\Sigma$ be a finite alphabet and $n\in\mathbb{N}$. A function $P\colon \Sigma^n\to\mathbb{D}$ is called a product function if there are functions $P_1,\ldots,P_n\colon \Sigma\to\mathbb{D}$ such that $P(x) = \prod\limits_{i=1}^{n}P_i(x_i)$.
\end{definition}
Given a product function $P\colon \Sigma^{n}\to\mathbb{D}$, the decomposition as $P_1(x_1)\cdots P_n(x_n)$ is only unique up to multiplying the $P_i$'s by constants of absolute value $1$, and this will be inconvenient for us at times. To remedy that we fix a distinguished element $\sigma^{\star}\in \Sigma$, and by default normalize so that $P(\sigma^{\star},\ldots,\sigma^{\star}) = 1$ and $P_1(\sigma^{\star})=\ldots=P_{n}(\sigma^{\star})=1$. Unless specified otherwise, throughout this paper we will consider normalized product functions $P$ and use their normalized decompositions.

We next define the projections of a product function.
\begin{definition}
    Given a product function $P(x) = \prod\limits_{i=1}^{n}P_i(x_i)$ and $I\subseteq [n]$, we define the function $P|_{I}\colon \Sigma^n\to\mathbb{D}$ as 
    $P|_{I}(x) = \prod\limits_{i\in I}P_i(x_i)$. For a collection of product functions $\mathcal{P}\subseteq \{P\colon \Sigma^n\to\mathbb{D}\}$ and $I\subseteq [n]$, we denote by 
    $\mathcal{P}|_{I}$ the collection $\{P|_{I}~|~P\in\mathcal{P}\}$.
\end{definition}

Next, we define the order of a product function and the notion of cyclic and discrete product functions.
\begin{definition}
    Let $P\colon \Sigma^{n}\to\mathbb{D}$. We define the order of $P$, denoted by ${\sf ord}(P)$, as the smallest integer $k$ such that $P(x)^{k}= 1$ for all $x\in\Sigma^n$. If no such $k$ exists we denote ${\sf ord}(P) = \infty$.
\end{definition}
% \begin{definition}
%     Let $\Sigma$ be a finite alphabet, let $G$ be an Abelian group, and let $\sigma\colon \Sigma\to G$. We denote by $\mathcal{P}(\Sigma,G,\sigma)$ the collection of product functions $P\colon \Sigma^{n}\to\mathbb{C}$ that take the form $P(x) = \chi(\sigma(x_1),\ldots,\sigma(x_n))$ where $\chi\in\hat{G}^{\otimes n}$.
% \end{definition}
% It is clear from definitions that if $P\in\mathcal{P}(\Sigma,G,\sigma)$, then ${\sf ord}(P)$ divides $G$. It will be most convenient for us to work with product functions that have orders that are prime powers.
\begin{definition}
    Let $P\colon \Sigma^n\to\mathbb{D}$ be a product function.
    \begin{enumerate}
        \item We say that $P$ is a cyclic product function if there is a prime $p$ and an integer $r$ such that ${\sf ord}(P) = p^r$. 
        \item We say that $P$ is $M$-discrete if ${\sf ord}(P)\leq M$.
    \end{enumerate}
\end{definition}
We have the following basic fact.
\begin{fact}\label{fact:triv_disc_cyclic}
    Let $P\colon \Sigma^n\to\mathbb{D}$ be a product function.
    \begin{enumerate}
        \item If ${\sf ord}(P)=M<\infty$, we may find $\sigma_1,\ldots,\sigma_n\colon \Sigma\to \mathbb{Z}_{M}$ such that $P(x) = e^{\frac{1}{M}2\pi{\bf i}\sum\limits_{i=1}^{n}\sigma_i(x_i)}$.
        \item If $P$ is $M$-discrete, then for all $x,x'\in\Sigma^n$, we either have $P(x)=P(x')$, or $\card{P(x)-P(x')}\geq \frac{1}{M}$.
    \end{enumerate}
\end{fact}
\begin{proof}
    For the first item, writing $P(x) = \prod\limits_{i=1}^{n}P_i(x_i)$ we get that $P_i^{M}$ is constant for all $i$, and by normalization it must be the constant $1$ function. Thus, it follows that there is $\sigma_i\colon \Sigma\to\mathbb{Z}_{M}$ such that 
    $P_i(x_i) = e^{\frac{1}{M}2\pi{\bf i}\sigma_i(x_i)}$ for all $x_i\in \Sigma$, completing the proof of the first item.

    For the second item, note that if $P(x)\neq P(x')$, then as they are both roots of unity of order ${\sf ord}(P)$ we get that
    \[
    \card{P(x)-P(x')}
    = 
    \card{e^{\frac{1}{{\sf ord}(P)}2\pi{\bf i}r}-1}
    \]
    for some $0<r<{\sf ord}(P)$.
    We assume $r\leq {\sf ord}(P)/2$, otherwise we can replace $r$ with ${\sf ord}(P)-r$ and the above equality is still valid. 
    If $r\geq {\sf ord}(P)/4$, then the real part of $e^{\frac{1}{{\sf ord}(P)}2\pi{\bf i}r}$ is 
    ${\sf cos}\left(\frac{r}{{\sf ord}(P)}2\pi)\right)\leq 0$, we get $\card{P(x)-P(x')}\geq 1$. Otherwise, we get
    \[
    \card{e^{\frac{1}{{\sf ord}(P)}2\pi{\bf i}r}-1}
    \geq {\sf sin}\left(\frac{r}{{\sf ord}(P)}2\pi\right)
    \geq \frac{2}{\pi} \frac{r}{{\sf ord}(P)}2\pi
    \geq \frac{1}{{\sf ord}(P)},
    \]
    where we used the inequality ${\sf sin}(\theta)\geq \frac{2}{\pi}\theta$ for $0\leq \theta\leq \frac{\pi}{2}$. This proves the second item.
\end{proof}
The next claim shows that discrete product functions far from constants have small expectation and are very high degree. To state it, we define the standard noise operator over product spaces.
\begin{definition}\label{def:noise_op}
    Let $\Sigma$ be a finite alphabet and let $\nu$ be a distribution over $\Sigma$. For $\rho\in [0,1]$ and $x\in \Sigma$, we define the distribution $y\sim\mathrm{T}_{\rho,\nu}x$ as outputting $x$ with probability $\rho$, and else outputting $y\sim \nu$. For $n\geq 2$ and $x\in\Sigma^n$, we define the distribution $\mathrm{T}_{\rho,\nu}^{\otimes n}x$ by outputting $y$ sampled by taking $y_i\sim \mathrm{T}_{\rho,\nu} x_i$ independently for each $i$.

    We also think of $\mathrm{T}_{\rho,\nu}^{\otimes n}$ as a linear operator on $L_2(\Sigma^n, \nu^{\otimes n})$, defined as
    \[
    \mathrm{T}_{\rho,\nu}^{\otimes n} f(x)
    =\Expect{y\sim \mathrm{T}_{\rho,\nu}^{\otimes n}x}{f(y)}
    \]
    for each $f\colon \Sigma^n\to\mathbb{C}$ and $x\in\Sigma^n$. When $n$ and $\nu$ are clear from context, we often omit them from notation.
\end{definition}

The next fact allows us to decompose a discrete product function as a product of cyclic product functions.
\begin{fact}\label{fact:break_discrete_to_cyclic}
    For all $M\in\mathbb{N}$ there exists $T\in\mathbb{N}$ such that the following holds.
    Suppose $P\colon \Sigma^{n}\to\mathbb{D}$ is an $M$-discrete product function. Then we may find cyclic product functions $Q_1,\ldots,Q_T$ with ${\sf ord}(Q_i)\leq M$ such that $P(x) = Q_1(x)\cdots Q_T(x)$.
\end{fact}
\begin{proof}
    Working coordinate by coordinate, it suffices to prove the statement for $n=1$. By assumption we know that $P(x) = e^{\frac{1}{M'}2\pi{\bf i} k(x)}$ for some $M'\leq M$ and $k\colon \Sigma\to \mathbb{Z}_{M'}$. Factorize $M' = p_1^{r_1}\cdots p_T^{r_T}$, and define $M_{-i}' = \prod\limits_{j\neq i} p_j^{r_j}$. As ${\sf gcd}(M_{-1}',\ldots,M_{-T}') = 1$, we get by Bezout's identity that there are integers $a_1,\ldots,a_T$ such that
    \[
     1 = \sum\limits_{i=1}^{T} a_i M_{-i}'.
    \]
    Plugging this in, we get that
    \[
    P(x) = e^{\frac{1}{M'}2\pi{\bf i} k(x)}
    =
    \prod\limits_{i=1}^{T}
    e^{\frac{a_i M_{-i}'}{M'}2\pi{\bf i} k(x)}
    =\prod\limits_{i=1}^{T}
    e^{\frac{a_i}{p_i^{r_i}}2\pi{\bf i} k(x)}.
    \]
    Defining $Q_i(x) = e^{\frac{a_i}{p_i^{r_i}}2\pi{\bf i} k(x)}$ gives the result.
\end{proof}

We will often consider combinations of functions from a collection $\mathcal{P}$, and we make the following definition.
\begin{definition}
    Let $\mathcal{P}$ be a collection of product functions over $\Sigma^n$. We define
    \[
    \spn(\mathcal{P})
    =\left\{\prod\limits_{P\in\mathcal{P}}P(x)^{\alpha_{P}}~\big|~\alpha_P\in\mathbb{N}~\forall P\in\mathcal{P}\right\}.
    \]
\end{definition}
We note that if each function in $\mathcal{P}$ has order at most $M<\infty$, then $\spn(\mathcal{P})$ is finite and has size at most $M^{\card{\mathcal{P}}}$.

\subsection{Two Rank Notions}
In this section we define the two notions of rank used in our argument. Both use the symbolic distance between product functions, defined as follows.
\begin{definition}
    Let $P,P'\colon \Sigma^n\to\mathbb{D}$ be product functions. We define $\Delta_{{\sf symbolic}}(P,P')$ to be the smallest integer $k$ such that there are $u_1,u_1'\ldots,u_n,u_n'\colon \Sigma\to\mathbb{D}$ such that 
    \[
    P(x) = \prod\limits_{i=1}^{n}u_i(x_i),
    \qquad
    P'(x) = \prod\limits_{i=1}^{n}u_i'(x_i)
    \]
    and $u_i=u_i'$ for all but $k$ of $i\in [n]$.
\end{definition}
\begin{definition}
    For $P\colon\Sigma^{n}\to\mathbb{D}$ and a collection of product functions $\mathcal{P}$ over $\Sigma^n$, we define 
    \[
    \Delta_{{\sf symbolic}}(P,\spn(\mathcal{P}))
    =\min_{Q\in\spn(\mathcal{P})}\Delta_{{\sf symbolic}}(P,Q).
    \]
\end{definition}
The following simple fact will be used throughout the paper.
\begin{fact}\label{fact:symb_dist_prop}
    Let $P,Q,R\colon \Sigma^{n}\to\mathbb{D}$ be product functions. Then:
    \begin{enumerate}
        \item $\Delta_{{\sf symbolic}}(P,R)\leq \Delta_{{\sf symbolic}}(P,Q)+\Delta_{{\sf symbolic}}(Q,R)$.
        \item For all integers $k$, 
        $\Delta_{{\sf symbolic}}(P^k,Q^k)\leq \Delta_{{\sf symbolic}}(P,Q)$.
    \end{enumerate}
\end{fact}
\begin{proof}
    The proof of both items is immediate by definition.
\end{proof}
\begin{fact} \label{fact:highrankdecay}
    Let $P\colon \Sigma^n\to\mathbb{D}$ be an $M$-discrete product function, and let $\nu$ be a distribution over $\Sigma$ in which the probability of each atom is at least $\alpha$. Suppose that $\Delta_{{\sf symbolic}}(P,1)\geq D$. Then:
    \begin{enumerate}
        \item 
        $\card{\Expect{x\sim\nu^{\otimes n}}{P(x)}}\leq 2^{-\Omega_{\alpha, M}(D)}$.
        \item For all $\tau>0$, $\norm{\mathrm{T}_{1-\tau}^{\otimes n}
        P}_2\leq 2^{-\Omega_{\alpha,\tau, M}(D)}$.
    \end{enumerate}
\end{fact}
\begin{proof}
    Write $P(x) = \prod\limits_{i=1}^{n}P_i(x_i)$. 
    By independence, it follows that
    \begin{equation}\label{eq:high_rk_0_exp}
    \card{\Expect{x\sim\nu^{\otimes n}}{P(x)}}
    =\prod\limits_{i=1}^{n} \card{\Expect{x_i\sim\nu}{P_i(x_i)}}.
    \end{equation}
    By assumption at least $D$ of the $P_i$ are not constant, and without loss of generality we assume these are $P_1,\ldots,P_D$. Fix $1\leq i\leq D$. Then
    \[
    \card{\Expect{x_i\sim\nu}{P_i(x_i)}}^2
    =\Expect{x_i,x_i'\sim\nu}{P_i(x_i)\overline{P_i(x_i')}}
    =
    1-\frac{1}{2}
    \Expect{x_i,x_i'\sim\nu}{\card{P_i(x_i)-P_i(x_i')}^2}
    \]
    Since $P_i$ is not constant there are $a,b\in\Sigma$ such that $P_i(a)\neq P_i(b)$, so by Fact~\ref{fact:triv_disc_cyclic} $\card{P_i(a)-P_i(b)}\geq \frac{1}{M}$. It follows that
    $\frac{1}{2}
    \Expect{x_i,x_i'\sim\nu}{\card{P_i(x_i)-P_i(x_i')}^2}\geq \alpha^2/M$, and so $\card{\Expect{x_i\sim\nu}{P_i(x_i)}}^2\leq 1-\frac{\alpha^2}{2M}$. Plugging into~\eqref{eq:high_rk_0_exp} finishes the proof of the first item.

    For the second item, by independence again $\norm{\mathrm{T}_{1-\tau}^{\otimes n}P}_2=\prod\limits_{i=1}^{n}\norm{\mathrm{T}_{1-\tau}P_i}_2$. Fixing $1\leq i\leq D$, we get that
    \[
    \norm{\mathrm{T}_{1-\tau}P_i}_2^2
    =\Expect{\substack{x_i\sim\nu\\x_i'\sim\mathrm{T}_{1-\tau}x_i}}{P_i(x_i)\overline{P_i(x_i')}}
    =
    1-\frac{1}{2}\Expect{\substack{x_i\sim\nu\\x_i'\sim\mathrm{T}_{1-\tau}x_i}}{\card{P_i(x_i)-P_i(x_i')}^2}.
    \]
    By the same argument as before, it follows that 
    $\Expect{\substack{x_i\sim\nu\\x_i'\sim\mathrm{T}_{1-\tau}x_i}}{\card{P_i(x_i)-P_i(x_i')}^2}\geq \frac{\alpha^2\tau}{M}$, and the second item follows.
\end{proof}

The following definition is the stronger notion of rank we consider, and roughly speaking it asserts that non-trivial combinations of functions from $\mathcal{P}$ are always high degree.
\begin{definition}\label{def:rank}
    Let $\mathcal{P}$ be a collection of product functions. We say that it has rank at least $D$, and denote ${\sf rk}(\mathcal{P})\geq D$, if for all integers $\{\alpha_P\}_{P\in\mathcal{P}}$ such that $0\leq \alpha_P<{\sf ord}(P)$ and at least one is non zero, the function $S(x) = \prod\limits_{P}P(x)^{\alpha_P}$ satisfies that
    $\Delta_{{\sf symbolic}}(S,1)\geq D$. 
\end{definition}

The following definition is a weaker notion of rank, and roughly speaking it asserts that any combination of functions from $\mathcal{P}$ either gives a constant function, or else a function far from constant in symbolic distance.
\begin{definition}\label{def:softrank}
    Let $\mathcal{P}$ be a collection of product functions. We say that it has soft-rank at least $D$, and denote ${\sf softrk}(\mathcal{P})\geq D$, if for all integers $\{\alpha_P\}_{P\in\mathcal{P}}$, defining $S(x) = \prod\limits_{P}P(x)^{\alpha_P}$, we either have that $S$ is constant or that
    $\Delta_{{\sf symbolic}}(S,1)\geq D$.
\end{definition}

The contexts in which rank and softrank are used is somewhat different. The notion of rank is used in scenarios where we are only dealing with a single function defined over a probability space. In that scenario, rank governs certain pseudorandomness properties, such as the probability the function attains a certain value. The weaker softrank is used in scenarios we have multiple functions defined over correlated probability spaces. Therein, it governs some form of weak independence between them. To formally explain this multi-function scenario we require the following definition.
\begin{definition}\label{def:tensor_of_collecitions}
    Suppose that $\Sigma,\Gamma,\Phi$ are finite alphabets, $\mu$ is a distribution over $\Sigma\times\Gamma\times \Phi$, and $\mathcal{P},\mathcal{Q},\mathcal{R}$ are collections of product functions over $\Sigma^n$, $\Gamma^n$ and $\Phi^n$ respectively. The collection of product functions $\mathcal{D}_{\mu}$ over ${\sf supp}(\mu)^n$ is defined by including in it, for each $P\in \mathcal{P}$, 
    (and similarly for each $Q\in \mathcal{Q}$ and $R\in\mathcal{R}$)  
    the function $P'(x,y,z) = P(x)$. 
\end{definition}
In words, the collection $\mathcal{D}_{\mu}$ can be thought of as a joint lifting of the collections $\mathcal{P}$, $\mathcal{Q}$ and $\mathcal{R}$ to a common domain. Thus, the rank and softrank of $\mathcal{D}_{\mu}$ can be thought of as pseudorandomness notions measuring the correlations/independence of $\mathcal{P}$, $\mathcal{Q}$, $\mathcal{R}$ over their correlated probability spaces. Though much of what we say holds for general $\mathcal{P},\mathcal{Q},\mathcal{R}$, the case where they are all the same will be sufficient for us, and we make the following definition.
\begin{definition}
    Let $\mathcal{P}$ be a collection of product functions, and let $\mu$ be a distribution over $\Sigma^3$. We define ${\sf softrk}_{\mu}(\mathcal{P}) = {\sf softrk}(\mathcal{D}_{\mu}(\mathcal{P},\mathcal{P},\mathcal{P}))$
\end{definition}
We now elaborate on the motivation behind the notions of rank and softrank, and on how one goes about achieving large rank/softrank.
\begin{enumerate}
    \item {\bf The single function setting:} suppose that we have a single collection of product functions $\mathcal{P}$, which we want to ``process'' into a high-rank collection. Intuitively, a violation to the high-rank condition means that there is a non-trivial dependency between functions in $\mathcal{P}$. If this dependency took a very simple form, such as being able to express $P_1\in\mathcal{P}$ as a product of other $P\in\mathcal{P}$, then one could move to a smaller collection $\mathcal{P}'\subseteq \mathcal{P}$ in a way such that $\mathrm{T}_{1-\eps,\mathcal{P}}\approx \mathrm{T}_{1-\eps,\mathcal{P}'}$.
    \footnote{We remark that, in general the dependencies one has to deal with need not take this form, and then a different modification of $\mathcal{P}$ is necessary.} Indeed, we show that there is a way to modify $\mathcal{P}$ to achieve high rank, so that the noise operator does not change by too much; in fact, the rank we are able to guarantee is arbitrarily large compared to $\eps$ and the size of the collection.\footnote{Strictly speaking, this involves flexibility in choosing the noise rate, which we are able to achieve.} 

    Having high rank is very helpful in analyzing events such as $\{x~|~P(x) = c_P~\forall P\in\mathcal{P}\}$, for fixed values $C_P$. In particular, it allows one to give strong estimates of their probabilities, show strong mixing results on random walks starting in them, and so on.
    
    \item {\bf The multi-function setting:} suppose we have multiple collections of product functions $\mathcal{P}$,
    $\mathcal{Q}$,
    $\mathcal{R}$ over $\Sigma^n$, $\Gamma^n$, $\Phi^n$ respectively, and we have a distribution $\mu$ over $\Sigma\times\Gamma\times \Phi$. Then  $\mathcal{D} = \mathcal{D}_{\mu}(\mathcal{P},\mathcal{Q},\mathcal{R})$ is a lifting of $\mathcal{P},\mathcal{Q},\mathcal{R}$ to the common domain ${\sf supp}(\mu)^{n}$. Thus, if ${\sf rk}(\mathcal{D})$ is large, then as discussed above, we would be able to analyze correlations between the values of $\mathcal{P}$, $\mathcal{Q}$ and $\mathcal{R}$ when the input is sampled from $\mu^{\otimes n}$ very well; in fact, their values would be nearly independently distributed. How can we ensure that $\mathcal{D}$ has a high rank, though?

    Applying a modification procedure as in the last time is possible, but it results in modifications to $\mathcal{D}$, instead of to  $\mathcal{P},\mathcal{Q},\mathcal{R}$. This presents issues in our analysis, the most serious of them is that we are unable to guarantee that the noise operators corresponding to each one of them do not change. %\footnote{More precisely, our argument requires that such modifications procedures will not change noise operators corresponding to $\mathcal{P},\mathcal{Q},\mathcal{R}$, by much. We do not now how to implement modifications to $\mathcal{D}$ in this way.} 
    To circumvent this issue, we instead use the weaker notion of softrank in such contexts, which turns out to be useful in our setting. 
    
    A procedure to gain high softrank is as follows. If $\mathcal{D}$ has small softrank, then we can combine in $\mathcal{D}$ to get a function that is close to constant in symbolic distance, which in our case is possible only when it depends on a small subset  $I\subseteq [n]$ of the coordinates. By projecting each function in $\mathcal{P},\mathcal{Q},\mathcal{R}$ on $\bar{I}$, the aforementioned combination becomes completely constant. Iterating this procedure, we get modifications of $\mathcal{P}$, $\mathcal{Q}$, $\mathcal{R}$ whose joint lifting has a large softrank.
    \footnote{Crucially, by taking the noise rate of the averaging operators to be sufficiently small, we can ensure that this modification does not change by much the noise operators of $\mathcal{P},\mathcal{Q},\mathcal{R}$.}
\end{enumerate}

\subsection{The \texorpdfstring{$\mu$}{mu}-norm and the CSP Stability Result}
We need the following notions and results from~\cite{BKM5}. The first of which is a definition of a certain semi-norm.
	\begin{definition}
		For a distribution $\mu$ over $\Sigma\times \Gamma\times \Phi$
		and a function $f\colon (\Sigma^n,\mu_x^{\otimes n})\to\mathbb{C}$,
		we define the $\mu$ semi-norm of $f$ as
		\[
		\norm{f}_{\mu}
		=\sup_{\substack{g\colon \Gamma^n\to\mathbb{C}\\ h\colon \Phi^n\to\mathbb{C}\\ \text{$1$-bounded}}}\card{\Expect{(x,y,z)\sim \mu^{\otimes n}}{f(x)g(y)h(z)}}.
		\]
	\end{definition}
    For a distribution $\nu$ over $\Sigma$, we define a collection of distributions over $\Sigma^3$ in which one of the marginals is $\nu$.
	\begin{definition}\label{def:set_of_dists}
		For an alphabet $\Sigma$, a distribution $\nu$ over $\Sigma$ and a parameter $\alpha>0$, define the collections
		\[
		M_{\nu} = \left\{\mu~~\Bigg|~~\substack{\text{\normalsize $\mu$ is a pairwise connected distribution over $\Sigma^3$ with no $(\mathbb{Z}, +)$-embedding}\\\\
        \text{\normalsize$\mu_x = \nu$\text{ or }$\mu_y=\nu$\text{ or }$\mu_z = \nu$}}\right\},
		\]
		\[
		M_{\alpha} = \sett{\mu}{\mu(x,y,z)\geq \alpha~\forall(x,y,z)\in {\sf supp}(\mu)},
		\]
		and $M_{\nu,\alpha} = M_{\nu}\cap M_{\alpha}$.
	\end{definition}

    \begin{definition}\label{def:semi_norm_nu}
		Let $\nu$ be a distribution over $\Sigma$, let $\alpha>0$, and let $f\colon (\Sigma^n,\nu^{\otimes n})\to\mathbb{C}$ be a function. We define
		\[
		\norm{f}_{\nu,\alpha}
		=\sup_{\mu\in M_{\nu,\alpha}}\norm{f}_{\mu}.
		\]
        When $\alpha$ is clear from context, we often omit it from the notation and write $\norm{f}_{\nu}$ instead of $\norm{f}_{\nu,\alpha}$.
	\end{definition}
    The last result in this section is the inverse theorem from~\cite{BKM4}.
    \begin{thm}\label{thm:csp4}
		For all $m\in\mathbb{N}$ there is a constant $S_m$ and
        a finite Abelian group $G$ of size at most $S_m$ such that for all
		$\eps,\alpha>0$ there exists $\delta>0$ and $d\in\mathbb{N}$ such that the following holds.
		Suppose that $\Sigma$, $\Gamma$, $\Phi$ are alphabets of size at most $m$,
		and $\mu$ is a distribution over $\Sigma\times\Gamma\times \Phi$ with no $\mathbb{Z}$ embeddings
		in which the probability of each atom is at least $\alpha$.
        Then there exists an embedding $(\sigma,\gamma,\phi)$ of ${\sf supp}(\mu)$ into $G$, such that
		if $f\colon \Sigma^n\to\mathbb{C}$ is a $1$-bounded function with $\norm{f}_{\mu}\geq \eps$, then
		there is $\chi\in \hat{G}^{\otimes n}$ and $L\colon \Sigma^n\to\mathbb{C}$ of degree at most
		$d$ and $\norm{L}_2\leq 1$ such that
		\[
		\card{\inner{f}{L\cdot \chi\circ\sigma^{\otimes n}}}\geq \delta.
		\]
	\end{thm}

    \subsection{Noise Operators}
    We will use the following noise operator from~\cite{BKM5}.
    
   \begin{definition}\label{def:noise_op_P}
		Let $\Sigma$ be a finite alphabet, let $\nu$ be a distribution over $\Sigma$,
		let $\mathcal{P} = \set{P_1,\ldots,P_r\colon \Sigma^n\to\mathbb{D}}$ be a collection of functions
        and let $I\subseteq [n]$.
		For each $x\in \Sigma^n$ we define the distribution $\mathrm{T}_{\nu, \mathcal{P}, I} x$
		as:
		\begin{enumerate}
			\item Sample $y\sim \nu^{\otimes n}$ conditioned on $y_{\overline{I}} = x_{\overline{I}}$ and $P_i(y) = P_i(x)$ for all $i$.
			\item Output $y$.
		\end{enumerate}
        For $\eps>0$, we define the distribution of $\mathrm{T}_{\nu,\mathcal{P},1-\eps} x$ by sampling $I\subseteq_{\eps}[n]$, and then outputting $y\sim \mathrm{T}_{\nu,\mathcal{P},I} x$.
	\end{definition}
	As usual, we will associate with $\mathrm{T}_{\nu, \mathcal{P}, I},\mathrm{T}_{\nu, \mathcal{P}, 1-\eps}$ averaging operators over functions,
	which by abuse of notation we denote as $\mathrm{T}_{\nu, \mathcal{P}, I}, \mathrm{T}_{\nu, \mathcal{P}, 1-\eps} \colon L_2(\Sigma^n,\nu^{\otimes n})\to L_2(\Sigma^n,\nu^{\otimes n})$. We define
	\[
    \mathrm{T}_{\nu, \mathcal{P}, I}f(x)=\Expect{y\sim \mathrm{T}_{\nu, \mathcal{P}, I}x}{f(y)},
    \qquad
	\mathrm{T}_{\nu, \mathcal{P}, 1-\eps} f(x) = 
    \Expect{I\subseteq_{\eps} [n]}{\Expect{y\sim \mathrm{T}_{\nu, \mathcal{P}, I} x}{f(y)}}
    =\Expect{I\subseteq_{\eps} [n]}{\mathrm{T}_{\nu, \mathcal{P}, I}f(x)}.
	\]
	When the distribution $\nu$ is clear from context, we will often drop $\nu$ from notation and simply write  $\mathrm{T}_{\mathcal{P},I}$ 
    and $\mathrm{T}_{\mathcal{P},1-\eps}$ in place of $\mathrm{T}_{\nu,\mathcal{P},I}$ and $\mathrm{T}_{\nu,\mathcal{P},1-\eps}$. 

 \subsubsection{Properties of the Noise Operator}
    We now list a few properties of the noise operator established in~\cite{BKM5}. The first one is~\cite[Fact 4.2]{BKM5}.
    \begin{fact}\label{fact:noise_op_basic_prop1}
        For all $I\subseteq[n]$ and $\mathcal{P}$, $\nu^{\otimes n}$ 
        is a stationary distribution of $\mathrm{T}_{\nu,\mathcal{P},I}$. Consequently $\nu^{\otimes n}$ is a stationary distribution of $\mathrm{T}_{\nu, \mathcal{P}, 1-\eps}$.	\end{fact}

        We will need the following variant of~\cite[Fact 4.3]{BKM5}.
 %        \begin{fact}\label{fact:noise_op_basic_prop2}
	% 	Let $\mathcal{P}$ be a collection of product functions, let $\eps>0$, let $P'\colon \Sigma^n\to\mathbb{C}$ be a product function and suppose that
	% 	$k = \min_{P\in \spn(\mathcal{P})}\Delta_{{\sf symbolic}}(P,P')$. Then
	% 	\begin{enumerate}
	% 		\item There is a coupling of $(x, y, y')$ such that $(x,y)$ is distributed according to $(x, \mathrm{T}_{\nu, \mathcal{P}, 1-\eps} x)$,
	% 		$(x,y')$ is distributed according to $(x, \mathrm{T}_{\nu, \mathcal{P}\cup\{P'\}, 1-\eps}x)$ and $\Prob{}{y\neq y'}\leq k\eps$.
	% 		\item For any $1$-bounded function $f\colon \Sigma^n\to\mathbb{C}$,
	% 		$\norm{\mathrm{T}_{\nu, \mathcal{P}\cup \{P'\}, 1-\eps} f - \mathrm{T}_{\nu, \mathcal{P}, 1-\eps} f}_2\leq 2\sqrt{k\eps}$.
	% 	\end{enumerate}
	% \end{fact}
%We will need the following easy corollary of Fact~\ref{fact:noise_op_basic_prop2}
\begin{fact}\label{fact:noise_op_basic_prop2}
    Suppose that $\mathcal{P},\mathcal{P}'$ are collections of product functions such that:
    \begin{enumerate}
        \item For all $P\in \mathcal{P}$, $\Delta_{{\sf symbolic}}(P,\spn(\mathcal{P}'))\leq k$.
        \item For all $P'\in \mathcal{P}'$, $\Delta_{{\sf symbolic}}(P',\spn(\mathcal{P}))\leq k$.
    \end{enumerate}
           Then for all $\eps>0$:
        \begin{enumerate}
			\item There is a coupling of $(x, y, y')$ such that $(x,y)$ is distributed according to $(x, \mathrm{T}_{\nu, \mathcal{P}, 1-\eps} x)$,
			$(x,y')$ is distributed according to $(x, \mathrm{T}_{\nu, \mathcal{P}', 1-\eps}x)$ and $\Prob{}{y\neq y'}\leq k(\card{\mathcal{P}}+|\mathcal{P}'|)\eps$.
			\item For any $1$-bounded function $f\colon \Sigma^n\to\mathbb{C}$,
			$\norm{\mathrm{T}_{\nu, \mathcal{P}', 1-\eps} f - \mathrm{T}_{\nu, \mathcal{P}, 1-\eps} f}_2\leq 2\sqrt{k(\card{\mathcal{P}}+|\mathcal{P}'|)\eps}$.
		\end{enumerate}
\end{fact}
\begin{proof}
    For each $P\in \spn(\mathcal{P})$, let $\tilde{P}\in\spn(\mathcal{P}')$ be the minimizer of $\Delta(P,\tilde{P})$, and let $K(P) = \{i~|~P_i\neq \tilde{P}_i\}$, so that $\card{K(P)}\leq k$. Similarly, define $K'(P')$ for $P'\in\mathcal{P}'$, and set $K = \bigcup_{P\in\mathcal{P}}K(P)\cup \bigcup_{P'\in\mathcal{P}'}K'(P')$. Then we get that $\card{K}\leq (\card{\mathcal{P}}+|\mathcal{P}'|)k$. Note that if $I\subseteq [n]\setminus K$, then $\mathrm{T}_{\nu,\mathcal{P},I}$ and $\mathrm{T}_{\nu,\mathcal{P}',I}$ are identical as they are both equal to $\mathrm{T}_{\nu,\mathcal{P}|_{I},I}$. We have $\Prob{I\subseteq_{\eps}[n]}{I\subseteq [n]\setminus K}\geq 1-\card{K}\eps$, so it follows that the statistical distance between $(x, \mathrm{T}_{\nu, \mathcal{P}', 1-\eps}x)$ and 
    $(x, \mathrm{T}_{\nu, \mathcal{P}, 1-\eps}x)$ is at most $\card{K}\eps$, and the first item follows.

    For the second item, fix a coupling $(x,y,y')$ as in the first item. Then
    \[
    \norm{\mathrm{T}_{\nu, \mathcal{P}', 1-\eps} f - \mathrm{T}_{\nu, \mathcal{P}, 1-\eps} f}_2^2
    =\Expect{x}{
    \card{\mathrm{T}_{\nu, \mathcal{P}, 1-\eps} f(x)
    -\mathrm{T}_{\nu, \mathcal{P}', 1-\eps} f(x)}^2}
    \leq \Expect{x,y,y'}{\card{f(y)-f(y')}^2},
    \]
    which is at most $4\Expect{x,y,y'}{1_{y\neq y'}}$ by the $1$-boundedness of $f$, and the second item follows.
\end{proof}

The following result is essentially~\cite[Claim 4.5]{BKM5}, except that we replaced the assumption that $\mathcal{P}$ arises from some group ($\mathcal{P}\subseteq\mathcal{P}(\Sigma,G,\sigma)$ in the notations therein) with the assumption that each function in $\mathcal{P}$ is discrete; this is in fact the property used in the proof therein. In words, the result says that for sufficiently small noise rate, the operator $\mathrm{T}_{\nu,\mathcal{P},1-\eps}$ acts like the identity on low-degree functions.
\begin{lemma}\label{lem:act_on_ld}
		Let $d,r,m\in\mathbb{N}$ and $\alpha,\eps>0$.
		Suppose that $\Sigma$ has size at most $m$, $\nu$ is a distribution over $\Sigma$
		in which the probability of each atom is at least $\alpha$, and
		$\mathcal{P}$
		is a collection of $O_m(1)$-discrete product functions of size at most $r$. Then for any $L\colon \Sigma^n\to\mathbb{C}$ of degree at most $d$
		we have that
		\[
		\norm{(I-\mathrm{T}_{\nu,\mathcal{P}, 1-\eps})L}_2^2\lll_{d,m,r,\alpha} \eps^{1/3}\norm{L}_2^2.
		\]
	\end{lemma}
The following result is essentially~\cite[Lemma 4.9]{BKM5}, where again we replace the assumption about $\mathcal{P}$ as before. The proof is the same as therein. In words, the result says that if $f$ is $1$-bounded, then $\mathrm{T}_{\nu,\mathcal{P}, 1-\eps}f$ can be approximated arbitrarily well in $L_2$-distance by a function of the form 
$\sum\limits_{P\in \spn(\mathcal{P})} P(x)L_P(x)$, where each $L_P$ is a low-degree function with bounded $2$-norm.
\begin{lemma}\label{lem:approx_formula}
		For all $m,r\in\mathbb{N}$, $\alpha, \eps>0$ and $\xi>0$ there exist $C, D\in\mathbb{N}$ such that the following holds.
		Suppose that $\Sigma$ has size at most $m$, $\nu$ is a distribution over $\Sigma$
		in which the probability of each atom is at least $\alpha$, and
		$\mathcal{P}$ is a collection of $O_m(1)$-discrete product functions of size at most $r$.
		Then for any $1$-bounded function $f\colon \Sigma^n\to\mathbb{C}$, there exists a function
		$f'\colon\Sigma^n\to\mathbb{C}$ such that:
		\begin{enumerate}
			\item The function $f'$ approximates $\mathrm{T}_{\nu,\mathcal{P}, 1-\eps}f$: $\norm{\mathrm{T}_{\nu,\mathcal{P}, 1-\eps}f - f'}_2\leq \xi$.
			\item The function $f'$ can be written as
			\[
			f'(x) = \sum\limits_{P\in \spn(\mathcal{P})}P(x)\cdot L_{P}(x)
			\]
			where for all $P$, ${\sf deg}(L_P)\leq D$ and $\norm{L_P}_2\leq C$.
		\end{enumerate}
	\end{lemma}

    \section{High and Low Rank Collections}
    In this section we give a few results concerning the notions of rank in Definitions~\ref{def:rank} and~\ref{def:softrank}. First in Sections~\ref{sec:high_rank},~\ref{sec:high_softrank} and~\ref{sec:high_bothrank}   we show processes that turn a given low-rank/low-softrank collection $\mathcal{P}$ into a high-rank/high-softrank collection. Then, in Section~\ref{sec:quasirandom_props} we prove some quasirandomness properties of high-rank collections.
    \subsection{Achieving a High Rank}\label{sec:high_rank}
    In this section we show how to process a given collection of product functions $\mathcal{P}$ into a high-rank collection $\mathcal{P}'$. The following claim captures the main sub-procedure in that process. 
    \begin{claim}\label{claim:high_rank_process}
        Let $\mathcal{P}\subseteq \{P\colon\Sigma^n\to\mathbb{D}\}$ be a collection of $M$-discrete cyclic product functions. If ${\sf rk}(\mathcal{P})\leq D$, then there exists $P\in \mathcal{P}$ such that one of the following holds:
        \begin{enumerate}
            \item Setting $\mathcal{P}'=\mathcal{P}\setminus\{P\}$, we have $\Delta_{{\sf symbolic}}(P,\spn(\mathcal{P}'))\leq D$.
            \item There exists an $M$-discrete cyclic product function $P'\colon \Sigma^n\to\mathbb{C}$ such that
            ${\sf ord}(P')<{\sf ord}(P)$, $\Delta_{{\sf symbolic}}(P',\spn(\mathcal{P}))\leq D$, and setting $\mathcal{P}'=(\mathcal{P}\setminus\{P\})\cup\{P'\}$, we have 
            \[
            \Delta_{{\sf symbolic}}(P,\spn(\mathcal{P}'))\leq D.
            \]
        \end{enumerate}
    \end{claim}
    \begin{proof}
        By definition, since ${\sf rk}(\mathcal{P})\leq D$ we may find not-all-zero integers $\{\alpha_P\}_{P\in\mathcal{P}}$ such that $0\leq \alpha_{P}\leq {\sf ord}(P)-1$ and the function $R(x) = \prod\limits_{P\in\mathcal{P}}P(x)^{\alpha_P}$ satisfies that $\Delta_{{\sf symbolic}}(R,1)\leq D$. For each prime $p$, let $\mathcal{P}_{p}\subseteq \mathcal{P}$ be the subset of product functions whose order is a power of $p$, and fix $p$ such that for some $P\in \mathcal{P}_{p}$ we have $\alpha_P\neq 0$. 
        For each $p'$, let $\beta_{p'} = \max_{P'\in\mathcal{P}_{p'}}{\sf ord}(P')$, and denote 
        \[
        s = \prod\limits_{p'\neq p: \mathcal{P}_{p'}\neq \emptyset} \beta_{p'}.
        \]
        Then by Fact~\ref{fact:symb_dist_prop} we have that
        \[
        \Delta_{{\sf symbolic}}(R^s,1)\leq 
        \Delta_{{\sf symbolic}}(R,1)
        \leq D.
        \]
        For $p'\neq p$ and $P'\in \mathcal{P}_{p'}$, as the order of $P'$ divides $s$, we get that $P'(x)^{s}\equiv 1$. For  $P\in \mathcal{P}_p$, as $s$ is coprime to $p$ we may find integers $i$ and $j$ such that $is + j{\sf ord}(P) = 1$, so
        $is\alpha_P \equiv \alpha_P\pmod{{\sf ord}(P)}$, and therefore $\alpha_P':=s\alpha_P\neq 0\pmod{{\sf ord}(P)}$ for some $P\in \mathcal{P}_{p}$.
        Thus,
        \[
        R(x)^s = 
        \prod\limits_{P\in\mathcal{P}_p}P(x)^{\alpha_P'},
        \]
        where $0\leq \alpha_{P}'\leq {\sf ord}(P)-1$ for all $P\in\mathcal{P}_p$ and at least one of them is non-zero.
        There are two cases now. 
        \vspace{-2ex}
        \paragraph{Case 1: } If some $\alpha_P'$ is coprime to $p$, then we fix such $P$ and find $\beta$ such that $\beta\cdot  \alpha_P' = 1\pmod{{\sf ord}(P)}$. Using Fact~\ref{fact:symb_dist_prop} we get that 
        $\Delta(R^{s\beta},1)\leq D$, and a direct computation gives 
        \[
        R(x)^{s\beta}
        = P(x)\prod\limits_{P'\in\mathcal{P}\setminus\{P\}} P'(x)^{\alpha_{P'}' \beta}.
        \]
        Thus, $\Delta(P,\prod\limits_{P'\in\mathcal{P}\setminus\{P\}} P'^{-\alpha_{P'}'\beta})\leq D$, and the first item holds. 
        \vspace{-2ex}
        \paragraph{Case 2:} Else, we may find an integer $k\geq 1$ and integers  $\beta_P$, at least one of which, say $\beta_{P^{\star}}$, is coprime to $p$, such that $\alpha_P' = p^k \beta_P$. Thus, we may write
        \[
        R(x)^{s} 
        = \left(\prod\limits_{P\in \mathcal{P}_p}P(x)^{\beta_P}\right)^{p^k} = 
        R'(x)^{p^k},
        \]
        where we define $R'(x)=\prod\limits_{P\in \mathcal{P}_p}P(x)^{\beta_P}$. Since $\Delta_{{\sf symbolic}}(R,1)\leq D$, we may find $J\subseteq [n]$ of size at most $D$, such that $R|_{\bar{J}}\equiv 1$. We get that for $P' = R'|_{\bar{J}}$, we have $(P')^{p^{k}}\equiv 1$. We now note that:
        \begin{enumerate}
        \item $P'$ is a cyclic product function with ${\sf ord}(P') \leq p^{k}\leq \alpha_{P^{\star}}' < {\sf ord}(P^{\star})$.
        \item We have that
        \[
            \Delta\left((P^{\star})^{-\beta_{P^{\star}}},
            \overline{P'}\prod\limits_{P\in\mathcal{P}_{p}\setminus\{P^{\star}\}} P^{\beta_P}\right)
            =
            \Delta\left(P', (P^{\star})^{\beta_{P^{\star}}}
            \prod\limits_{P\in\mathcal{P}_{p}\setminus\{P^{\star}\}} P^{\beta_P}\right)
            \leq \card{J}\leq D.
        \]
        As $\beta_{P^{\star}}$ is coprime to $p$, we may find $s'$ such that $s'\cdot \beta_{P^{\star}}=1\pmod{{\sf ord}(P^{\star})}$, and get from the above and Fact~\ref{fact:symb_dist_prop} that
        \[
        \Delta
        \left(P^{\star},
            (P')^{s'}\prod\limits_{P\in\mathcal{P}_{p}\setminus\{P^{\star}\}} P^{-s'\cdot\beta_P}\right)
            =
            \Delta\left((P^{\star})^{s'\beta_{P^{\star}}},(P')^{s'}\prod\limits_{P\in\mathcal{P}_{p}\setminus\{P^{\star}\}} P^{-s'\beta_P}\right)
            \leq D.
        \]
        \end{enumerate}
        Together, we get that the second item holds for $\mathcal{P}'=
        (\mathcal{P}\setminus\{P^{\star}\})\cup\{P'\}$.
    \end{proof}
    Iterating Claim~\ref{claim:high_rank_process}, we get the following lemma, which will be used in subsequent sections.
    \begin{lemma}\label{lem:rank_gain}
        Let $\mathcal{P}$ be a collection of $M$-discrete cyclic product functions, let $k = M \card{\mathcal{P}}+1$ and let $D_1\leq D_2\leq\ldots\leq D_k$ be integers. Then we may find a collection $\mathcal{P}'$ of $M$-discrete cyclic product functions with $|\mathcal{P}'|\leq \card{\mathcal{P}}$ such that for some $1\leq i\leq k$, the  following properties hold:
        \begin{enumerate}
            \item ${\sf rk}(\mathcal{P}')\geq D_i$.
            \item For each $P\in\mathcal{P}$, we have that 
            $\Delta_{{\sf symbolic}}(P,\spn(\mathcal{P}'))\leq D_1+\ldots+D_{i-1}$.
            \item For each $P'\in\mathcal{P}'$, we have that 
            $\Delta_{{\sf symbolic}}(P',\spn(\mathcal{P}))\leq D_1+\ldots+D_{i-1}$.
        \end{enumerate}
    \end{lemma}
    \begin{proof}
        %If ${\sf rk}(\mathcal{P})\geq D_1$ we are done, so assume otherwise. 
        %Applying Claim~\ref{claim:high_rank_process} we find a collection $\mathcal{P}_1'$ as therein. 
        %It follows that for each $P\in \spn(\mathcal{P})$, $\Delta_{{\sf symbolic}}(P,\spn(\mathcal{P}_1'))\leq D_1$, and that for each $P'\in\spn(\mathcal{P}_1')$, $\Delta_{{\sf symbolic}}(P',\spn(\mathcal{P}))\leq D_1$. 
        Starting with $\mathcal{P}_0'=\mathcal{P}$ and $i=1$, we iterate the following process:
        \begin{itemize}
            \item If ${\sf rk}(\mathcal{P}'_{i-1})\geq D_i$, halt.
            \item Else, apply Claim~\ref{claim:high_rank_process} on $\mathcal{P}_{i-1}'$ to find $\mathcal{P}'_i$ as therein. Thus, we get that for each $P\in \spn(\mathcal{P}_{i-1}')$ and $P'\in \spn(\mathcal{P}_{i}')$ we have \[
            \Delta_{{\sf symbolic}}(P,\spn(\mathcal{P}_i'))\leq D_i,
            \qquad\qquad
            \Delta_{{\sf symbolic}}(P',\spn(\mathcal{P}_{i-1}'))\leq D_i.
            \]
            Increment $i$ and iterate.
        \end{itemize}
        Note that as $\sum\limits_{P\in\mathcal{P}_i'}{\sf ord}(P)
        <\sum\limits_{P\in\mathcal{P}_{i-1}'}{\sf ord}(P)$, the process terminates at some $i\leq k$, and we take $\mathcal{P}' = \mathcal{P}_{i-1}'$. The first item automatically holds. For the second item fix $P\in\spn(\mathcal{P})$. Applying the second bullet of the process we get functions $P_j\in \spn(\mathcal{P}_{j}')$ such that 
        \[
        \Delta_{\sf symbolic}(P, P_1')\leq D_1,
        \qquad\qquad
        \Delta_{\sf symbolic}(P_{j}', P_{j+1}')\leq D_j~\forall j\leq i-2.
        \]
        Combining via Fact~\ref{fact:symb_dist_prop} gives that $\Delta_{\sf symbolic}(P,P_{i-1}')\leq D_1+\ldots+D_{i-1}$, so the second item holds. The proof of the third item is identical.
    \end{proof}
    \subsection{Achieving a High Softrank}\label{sec:high_softrank}
    In this section we show how to process a given collection of product functions $\mathcal{P}$ into a high softrank collection. This is captured by the following lemma.
    \begin{lemma}\label{lem:softrank_gain}
        Let $\Sigma$ be a finite alphabet, let $\mu$ be a distribution over $\Sigma^3$, let $\mathcal{P}$ be a collection of $M$-discrete cyclic product functions and set $k = M^{3\card{\mathcal{P}}}$. Then for all integers $D_1\leq D_2\leq\ldots\leq D_k$ we may find $1\leq i\leq k$ and a subset $J\subseteq [n]$ such that the following holds:
        \begin{enumerate}
            \item $\card{J}\leq D_1+\ldots+D_{i-1}$.
            \item ${\sf softrk}_{\mu}(\mathcal{P}|_{\bar{J}})\geq D_{i}$.
        \end{enumerate}
    \end{lemma}
    \begin{proof}
        Start with $i=1$ and $\mathcal{P}_1 = \mathcal{P}$. As long as ${\sf softrk}_{\mu}(\mathcal{P}_i)<D_i$, by definition we may find integers $\{\alpha_P,\beta_{P},\gamma_{P}\}_{P\in\mathcal{P}}$ such that the function 
        $R\colon {\sf supp}(\mu)\to\mathbb{D}$ defined as
        \[
        R(x,y,z) = \prod\limits_{P\in\mathcal{P}}P(x)^{\alpha_P}P(y)^{\beta_P}P(z)^{\gamma_P}
        \] 
        satisfies that $R\not\equiv 1$ but $\Delta_{{\sf symbolic}}(R,1)\leq D_i$. We may find $J_i\subseteq [n]$ of size at most $D_i$ such that $R|_{\bar{J}_i}\equiv 1$, define $\mathcal{P}_{i+1} = \mathcal{P}|_{\bar{J}_i}$, increment $i$ by $1$ and iterate. 

        We note that if for integers $\{\alpha_P',\beta_{P}',\gamma_{P}'\}_{P\in\mathcal{P}}$ we have that
        \[
        \prod\limits_{P\in\mathcal{P}}P(x)^{\alpha_P'}P(y)^{\beta_P'}P(z)^{\gamma_P'}\equiv 1,
        \]
        in $\mathcal{P}_i$, then this remains true for the corresponding functions in $\mathcal{P}_{i+1}$. Thus, each step of the iterative process increases the number of $\{\alpha_P',\beta_{P}',\gamma_{P}'\}_{P\in\mathcal{P}}$ as above by at least $1$, so the process terminates within $k$ steps. Let $i$ be the index we stop with, and set $J = J_1\cup\ldots \cup J_{i-1}$. Then $\card{J}\leq D_1+\ldots+D_{i-1}$ so the first item holds, and $\mathcal{P}_i = \mathcal{P}|_{\bar{J}}$ so the second item holds.
    \end{proof}

    \subsection{Achieving High Rank and High Soft Rank}\label{sec:high_bothrank}
    We now combine Lemmas~\ref{lem:rank_gain} and~\ref{lem:softrank_gain} to get the following result:
    \begin{lemma}\label{lem:bothrank_gain}
        Let $\Sigma$ be a finite alphabet, let $\mu$ be a distribution over $\Sigma^3$, let $\mathcal{P}$ be a collection of $M$-discrete cyclic product functions. Then there exists $k = k(\card{\mathcal{P}},M)$
        such that for all integers $D_1\leq D_2\leq\ldots\leq D_k$ we may find $1\leq i\leq k$ and a collection $\mathcal{P}'$ of $M$-discrete cyclic product functions with $|\mathcal{P}'|\leq \card{\mathcal{P}}$ such that the following holds:
        \begin{enumerate}
            \item ${\sf softrk}_{\mu}(\mathcal{P}')$ and ${\sf rk}(\mathcal{P}')$ are both at least $D_i -
        (D_1+\ldots+D_{i-1})$.
            \item For each $P\in\mathcal{P}$, we have that 
            $\Delta_{{\sf symbolic}}(P,\spn(\mathcal{P}'))\leq 2(D_1+\ldots+D_{i-1})$.
            \item For each $P'\in\mathcal{P}'$, we have that 
            $\Delta_{{\sf symbolic}}(P',\spn(\mathcal{P}))\leq 2(D_1+\ldots+D_{i-1})$.
        \end{enumerate}
    \end{lemma}
    \begin{proof}
        Take $k_1$ and $k_2$ from Lemmas~\ref{lem:rank_gain} and~\ref{lem:softrank_gain} and set $k=k_1\cdot k_2$. Consider the sequence $E_1\leq E_2\leq\ldots\leq E_{k_1}$ given as $E_j = D_{j k_2}$. Then by Lemma~\ref{lem:rank_gain} we may find a collection $\widetilde{\mathcal{P}}$ and $j$ as therein. In particular, 
        ${\sf rk}(\widetilde{\mathcal{P}})\geq E_j = D_{j\cdot k_2}$ and 
        the second and third item there hold with the right hand side being $E_1+\ldots+E_{j-1}\leq D_1+\ldots+D_{(j-1)\cdot k_2}$. Consider the sequence $E_1'\leq E_2'\leq\ldots\leq E_{k_2}'$ given by $E_\ell' = D_{(j-1)k_2+\ell}$. Then applying Lemma~\ref{lem:softrank_gain} we may find an $\ell$ and a subset $J$ of size at most $E_1'+\ldots + E_{\ell-1}'$ such that ${\sf softrk}_{\mu}(\widetilde{\mathcal{P}}|_{\bar{J}})\geq E_{\ell}'$. 
        
        We take $\mathcal{P}' = \widetilde{\mathcal{P}}|_{\bar{J}}$ and $i=(j-1)k_2+\ell$. Then ${\sf softrk}_{\mu}(\mathcal{P}') = {\sf softrk}_{\mu}(\widetilde{P}|_{\bar{J}})\geq E_{\ell}' = D_{i}$ and 
        \[
        {\sf rk}(\mathcal{P}')
        \geq 
        {\sf rk}(\widetilde{\mathcal{P}})-\card{J}
        \geq D_{j k_2}
        -
        (D_1+\ldots+D_{i-1})
        \geq D_i-
        (D_1+\ldots+D_{i-1}),
        \]
        so the first item of the lemma holds. For the second item, for $P\in\mathcal{P}$ we have
        \begin{align*}
        \Delta_{{\sf symbolic}}(P,\spn(\mathcal{P}'))
        \leq 
        \Delta_{{\sf symbolic}}(P,\spn(\widetilde{\mathcal{P}}))+\card{J}
        &\leq E_1+\ldots+E_{j-1}+E_1'+\ldots+E_{\ell-1}'\\
        &\leq 
        2(D_1+\ldots+D_{i-1}).
        \end{align*}
        The proof of the third item is identical.
    \end{proof}
    
    \subsection{Quasirandomness of High-rank Collections}\label{sec:quasirandom_props}
In this section, we prove a few lemmas that show the utility of having a collection of product functions $\mathcal{P}$ with high rank.

The first lemma shows that if $\mathcal{P}$ has a high rank, then the distribution of $\{ P(x)\}_{P\in\mathcal{P}}$ where $x\sim 
\nu^{\otimes n}$ is close to the uniform distribution on $\prod\limits_{P\in\mathcal{P}}{\sf Image}(P)$.
      \begin{lemma}
    \label{lemma:quasirandom_onefunc}
		Let $\Sigma$ be an alphabet of size $m$ and let $\nu$ be a distribution over $\Sigma$ in which the probability of each atom is at least $\alpha$.
		Suppose that $\mathcal{P}$ is a collection of $M$-discrete, $n$-variate product functions with $|\mathcal{P}|=r$, and suppose that ${\sf rk}(\mathcal{P})\geq D$. Then for any $\vec{b}\in\prod\limits_{P\in\mathcal{P}}{\sf Image}(P)$ we have that
		\[
		\card{\Prob{x \sim\nu^{\otimes n}}{\forall P\in \mathcal{P}, P(x)= b_P}
		- \prod\limits_{P\in\mathcal{P}}\frac{1}{{\sf ord}(P)}}\leq 2^{-\Omega_{M,r,\alpha}(D)}.
		\]
	\end{lemma}

    \begin{proof}
 Denote $\mathcal{P} = \{ P_1, P_2, \ldots, P_r\}$. Then we may expand:
      \begin{align*}
          \Prob{x\sim\nu^{\otimes  n}}{\forall P\in \mathcal{P}, P(x)= b_P} &= \Expect{{x\sim\nu^{\otimes  n}}}{\prod_{P\in \mathcal{P}}\left( \frac{1}{{\sf ord}(P)} \sum_{j=0}^{{\sf ord}(P)-1}\left(\overline{b_P}P(x)\right)^j\right)}\\
          &\quad\quad= \sum_{\vec{j} = (j_1, \ldots, j_r)}  \left(\prod_{i=1}^r \frac{\overline{b_P}^{j_i}}{{\sf ord}(P_i)}\right) \Expect{x\sim\nu^{\otimes  n}}{\prod_{i=1}^r P_i(x)^{j_i}}\\
      \end{align*}
      Denote $S_{\vec{j}}(x):= \prod_{i=1}^r P_i(x)^{j_i}$. 
     We now analyze the expectation for the following two different scenarios. 
     \begin{enumerate}
         \item[(a)] $\vec{j}  = \vec{0}$: In this case, we have that the expectation is $1$.
         
         \item[(b)] $\vec{j}  \neq \vec{0}$: In this case, $\Delta_{{\sf symbolic}}(S_{\vec{j}},1)\geq D$ since ${\sf rk}(\tilde{\mathcal{P}})\geq D$. Using Fact~\ref{fact:highrankdecay} (1), we have 
         \[
         \card{\Expect{x\sim\nu^{\otimes  n}}{\prod_{i=1}^r P_i(x)^{j_i}}} \leq 2^{-\Omega_{\alpha, M}(D)}.
         \]
     \end{enumerate}
        Using the above, summing over all $\vec{j}$, we get that
\[
		\card{\Prob{x \sim\nu^{\otimes n}}{\forall P\in \mathcal{P}, P(x)= b_P}
		- \prod\limits_{P\in\mathcal{P}}\frac{1}{{\sf ord}(P)}}\leq 2^{-\Omega_{M,r,\alpha}(D)}.
        \qedhere
		\]

    \end{proof}

    The next lemma shows that if the collection of product functions $\mathcal{P}$ has high rank and softrank with respect to a distribution $\mu$ over $\Sigma^3$, then 
    %the tuple $(P(x), P(y), P(y))$ has a quasirandom behavior, where $(x, y, z)\sim \mu^{\otimes n}$. Note that this behavior is weaker than the one in Lemma~\ref{lemma:quasirandom_onefunc}. More specifically, 
    for any $\vec{b}\in\prod\limits_{P\in\mathcal{P}}{\sf Image}(P)$, the probability that the product functions evaluate to $\vec{b}$ on each one of $x, y, z$, where $(x,y,z)\sim \mu^{\otimes n}$, is non-negligible.
    \begin{lemma}
        Let $\Sigma$ be an alphabet of size $m$ and let $\mu$ be a distribution over $\Sigma^{3}$ such that $(x,x,x)\in {\sf supp}(\mu)$ for all $x\in\Sigma$, and in which the probability of each atom is at least $\alpha$.
        Suppose that $\mathcal{P}$ is a collection of $M$-discrete, $n$-variate product functions, and suppose that ${\sf rk}(\mathcal{P}),{\sf softrk}_{\mu}(\mathcal{P})\geq D$. 
        For any $\vec{b}\in\prod\limits_{P\in\mathcal{P}}{\sf Image}(P)$ we have that
        \[
        \Prob{(x,y,z)\sim\mu^{\otimes n}}{P(x)=P(y)=P(z) = b_P}
        \geq \prod\limits_{P\in\mathcal{P}}\frac{1}{{\sf ord}(P)^3}-2^{-\Omega_{M,m,\alpha}(D)}.
        \]
    \end{lemma}

    We prove a more general version of the above lemma in which the conclusion holds even after conditioning on a small subset of coordinates taking values from the support of $\mu$.

    \begin{lemma}
    \label{lemma:quasirandom_prop_2}
		Let $\Sigma$ be an alphabet of size $m$ and let $\mu$ be a distribution over $\Sigma^{3}$ such that $(x,x,x)\in {\sf supp}(\mu)$ for all $x\in\Sigma$, and in which the probability of each atom is at least $\alpha$.
		Suppose that $\mathcal{P}$ is a collection of $M$-discrete, $n$-variate product functions with $|\mathcal{P}|=r$, and suppose that ${\sf rk}(\mathcal{P}),{\sf softrk}_{\mu}(\mathcal{P})\geq D$. Suppose $J\subseteq [n]$ of size $|J|\leq \frac{D}{2}$. Then for every $(w, w', w'') \in {\sf supp}(\mu|_{J})$ and 
		for any $\vec{b}\in\prod\limits_{P\in\mathcal{P}}{\sf Image}(P)$ we have that
		\[
		\Prob{(x,y,z)\sim\mu^{\bar{J}}}{\forall P\in \mathcal{P}, P(w,x)=P(w', y)=P(w'', z) = b_P}
		\geq \prod\limits_{P\in\mathcal{P}}\frac{1}{{\sf ord}(P)^3}-2^{-\Omega_{M,r,\alpha}(D)}.
		\]
	\end{lemma}

    \begin{proof}
      Fix any $(w, w', w'') \in {\sf supp}(\mu|_{J})$ and $\vec{b}\in\prod\limits_{P\in\mathcal{P}}{\sf Image}(P)$. We begin by expressing the probability as follows. For the sake of ease of notation, in what follows, we let $o_P:={\sf ord}(P)$ and $\mathcal{P} = \{ P_1, P_2, \ldots, P_r\}$.
      \begin{align*}
          &\Prob{(x,y,z)\sim\mu^{\bar{J}}}{\forall P\in \mathcal{P}, P(w,x)=P(w',y)=P(w'',z) = b_P} \\
          &\quad\quad= \Expect{(x,y,z)\sim\mu^{\bar{J}}}{\prod_{P\in \mathcal{P}}\left( \frac{1}{o_P} \sum_{j=0}^{o_P-1}\left(\overline{b_P}P(w,x)\right)^j\right)\cdot \left( \frac{1}{o_P} \sum_{j=0}^{o_P-1}\left(\overline{b_P}P(w',y)\right)^j\right)\cdot \left( \frac{1}{o_P} \sum_{j=0}^{o_P-1}\left(\overline{b_P}P(w'',z)\right)^j\right)}\\
          &\quad\quad= \Expect{(x,y,z)\sim\mu^{ \bar{J}}}{\prod_{P\in \mathcal{P}} \frac{1}{o_P^3}\left(  \sum_{j, k, \ell=0}^{o_P-1}\left(\overline{b_P}^{(j+k+\ell)}P(w,x)^j P(w',y)^k P(w'',z)^\ell\right)\right)}\\
         &\quad\quad= \sum_{\substack{\vec{j} = (j_1, \ldots, j_r) \\ \vec{k} = (k_1, \ldots, k_r) \\ \vec{\ell} = (\ell_1, \ldots, \ell_r)}}\Expect{(x,y,z)\sim\mu^{\bar{J}}}{\left(  \prod_{i=1}^r \frac{1}{o_{P_i}^3}\left(\overline{b_{P_i}}^{(j_i +k_i+\ell_i)} P_i(w,x)^{j_i} P_i(w',y)^{k_i} P_i(w'',z)^{\ell_i}\right)\right)}
         \end{align*}
We let $\alpha_i = \prod_{i'\in J} P_{i,i'}(w_{i'})$ be the evaluation of $P_i$ at $w$ restricted to the coordinates in $J$. Similarly, let  $\beta_i = \prod_{i'\in J} P_{i,i'}(w'_{i'})$ and  $\gamma_i = \prod_{i'\in J} P_{i,i'}(w''_{i'})$. Also, let $\tilde{P} = P|_{\bar{J}}$ and $\tilde{\mathcal{P}} = \{ \tilde{P}_1, \tilde{P}_2, \ldots, \tilde{P}_r\}$. With these notations, we can rewrite the above expression as:

         \begin{align*}
          &\Prob{(x,y,z)\sim\mu^{\bar{J}}}{\forall P\in \mathcal{P}, P(w,x)=P(w',y)=P(w'',z) = b_P} \\
          &\quad\quad= \sum_{\substack{\vec{j} = (j_1, \ldots, j_r) \\ \vec{k} = (k_1, \ldots, k_r) \\ \vec{\ell} = (\ell_1, \ldots, \ell_r)}}  \left(\prod_{i=1}^r \frac{1}{o_{P_i}^3}\right)\underbrace{\left(\prod_{i=1}^r \left(\overline{b_{P_i}}^{(j_i +k_i+\ell_i)}\right) \alpha_i^{j_i} \beta_i^{k_i} \gamma_i^{\ell_i}\right) \Expect{(x,y,z)\sim\mu^{\bar{J}}}{\prod_{i=1}^r \tilde{P}_i(x)^{j_i} \tilde{P}_i(y)^{k_i} \tilde{P}_i(z)^{\ell_i}}}_{=:\Theta(\vec{j}, \vec{k}, \vec{\ell})}\\
      \end{align*}
     Let us denote the product function within the expectation by $S_{(\vec{j}, \vec{k}, \vec{\ell})}$, i.e., 
     $$ S_{(\vec{j}, \vec{k}, \vec{\ell})}(x, y, z):= \prod_{i=1}^r \tilde{P}_i(x)^{j_i} \tilde{P}_i(y)^{k_i} \tilde{P}_i(z)^{\ell_i}.$$
     As  ${\sf rk}(\mathcal{P}),{\sf softrk}_{\mu}(\mathcal{P})\geq D$ and $|J|\leq D/2$, we have ${\sf rk}(\tilde{\mathcal{P}}),{\sf softrk}_{\mu}(\tilde{\mathcal{P}})\geq D/2$. We now analyze the quantity $\Theta(\vec{j}, \vec{k}, \vec{\ell})$ for the following three different scenarios. 
     \begin{enumerate}
         \item[(a)] $\vec{j} = \vec{k} = \vec{\ell} = \vec{0}$: In this case, we have $\Theta(\vec{j}, \vec{k}, \vec{\ell})= 1$.
         
         \item[(b)] $\Delta_{{\sf symbolic}}(S_{(\vec{j}, \vec{k}, \vec{\ell})},1)\geq D/2$: In this case, using Fact~\ref{fact:highrankdecay} (1), we have $\card{\Theta(\vec{j}, \vec{k}, \vec{\ell})} \leq 2^{-\Omega_{\alpha, M}(D)}$.
         
        \item[(c)]  $S_{(\vec{j}, \vec{k}, \vec{\ell})}$ is a constant function: As ${\sf softrk}_{\mu}(\tilde{\mathcal{P}})\geq D/2$, this is the only remaining case. We now use the fact that ${\sf rk}(\tilde{\mathcal{P}})$ is also large. Since $(x, x, x)$ is in the support of $\mu$ and $S_{(\vec{j}, \vec{k}, \vec{\ell})}$ is a constant function, $S_{(\vec{j}, \vec{k}, \vec{\ell})}(x, x, x)$ is also a constant function. As ${\sf rk}(\tilde{\mathcal{P}})\geq D/2$, this can only happen if $j_i+k_i+\ell_i = 0 \mod o_{P_{i}}$ for all $1\leq i\leq r$. Thus, $S_{(\vec{j}, \vec{k}, \vec{\ell})}(x, x, x)$ is the constant $1$ function and hence $S_{(\vec{j}, \vec{k}, \vec{\ell})}$ is also the constant $1$ function. Furthermore, as $b_{P_i}\in {\sf Image}(P_i)$, we get $\overline{b_{P_i}}^{(j_i +k_i+\ell_i)}\equiv 1$ for all $1\leq i\leq r$. Thus, 
         $$\Theta(\vec{j}, \vec{k}, \vec{\ell}) =  \left(\prod_{i=1}^r\alpha_i^{j_i} \beta_i^{k_i} \gamma_i^{\ell_i}\right).$$
         As we are in the case that $\Delta_{{\sf symbolic}}(S_{(\vec{j}, \vec{k}, \vec{\ell})},1)< D/2$, it must be the case that the following function is a constant function for every $(w, w', w'') \in {\sf supp}(\mu|_{J})$ and $(x, y, z)\in {\sf supp}(\mu^{\bar{J}})$.
         \begin{align*}
         \prod_{i=1}^r \left( P_i(w,x)^{j_i} P_i(w',y)^{k_i} P_i(w'',z)^{\ell_i}\right) &=\prod_{i=1}^r  \alpha_i^{j_i} \beta_i^{k_i} \gamma_i^{\ell_i} \prod_{i=1}^r \tilde{P}_i(x)^{j_i} \tilde{P}_i(y)^{k_i} \tilde{P}_i(z)^{\ell_i}\\
         & = \prod_{i=1}^r  \alpha_i^{j_i} \beta_i^{k_i} \gamma_i^{\ell_i}.
         \end{align*}
         This is because in this case, the function on the left hand side will have symbolic distance from $1$ strictly less than $D$ and we know that ${\sf softrk}_{\mu}(\mathcal{P})\geq D$. Note that in the last step above, we used the fact that $S_{(\vec{j}, \vec{k}, \vec{\ell})}$ is the constant $1$ function. As $(x, x, x)\in {\sf supp}(\mu)$ for all $x\in \Sigma$, we must that that above function, when evaluated at $(w, w, w)$  for $w\in \Sigma^{J}$ is equal to 
          $$\prod_{i=1}^r  \alpha_i^{j_i} \alpha_i^{k_i} \alpha_i^{\ell_i} = \prod_{i=1}^r  \alpha_i^{{j_i}+{k_i} +{\ell_i}} = 1,$$
          where $\alpha_i = \prod_{i'\in J} P_{i,i'}(w_{i'})$ and again we used the fact that $\alpha_i\in {\sf Image}(P_i)$ and ${j_i}+{k_i} +{\ell_i} = 0 \mod o_{P_{i}}$. Thus, we have 
          $$\Theta(\vec{j}, \vec{k}, \vec{\ell}) = 1$$
     \end{enumerate}
     If we denote the number of tuples $(\vec{j}, \vec{k}, \vec{\ell})$ that fall into cases (a) and (c) above by $C_0$, then we have $C_0\geq 1$. As the total number of tuples is upper bounded by $\left(\prod\limits_{P\in\mathcal{P}} {\sf ord (P)}\right)^{3r} = O_{M, r}(1)$, we get 
       \begin{align*}
          \Prob{(x,y,z)\sim\mu^{\bar{J}}}{\forall P\in \mathcal{P}, P(w,x)=P(w', y)=P(w'', z) = b_P}\geq C_0\prod\limits_{P\in\mathcal{P}}\frac{1}{{\sf ord}(P)^3}-2^{-\Omega_{M,r,\alpha}(D)},
      \end{align*}
      as required.
    \end{proof}

Consider the noise operator $\mathrm{T}_{1-\kappa,\mu}^{\otimes n}$ on $(\Sigma^n)^3$ as per Definition~\ref{def:noise_op}: on a tuple $(x, y, z)$ it outputs $(x', y', z')$ where for each coordinate $i\in [n]$ we have $(x'_i,y'_i,z'_i) = (x_i, y_i, z_i)$ with probability $1-\kappa$ and we take $(x'_i,y'_i,z'_i)\sim \mu$  with probability $\kappa$ independently. As usual, when clear from context we omit the $n$ and $\mu$ from notation. The following lemma shows that for a typical tuple $(x, y, z)$ the evaluations of the product function on $(x',y',z')\sim \mathrm{T}_{1-\kappa,\mu}^{\otimes n}(x, y, z)$ is essentially independent of $(x, y, z)$ if ${\sf softrk}_{\mu}(\mathcal{P})$ is large.
\begin{lemma}
\label{lemma:noise_condition_uniform}
    Let $\Sigma$ be an alphabet of size $m$ and let $\mu$ be a distribution over $\Sigma^{3}$ such that $(x,x,x)\in {\sf supp}(\mu)$ for all $x\in\Sigma$, and in which the probability of each atom is at least $\alpha$.
		Suppose that $\mathcal{P}$ is a collection of $M$-discrete, $n$-variate product functions with $|\mathcal{P}|=r$, and suppose that ${\sf softrk}_{\mu}(\mathcal{P})\geq D$. 
        
        Suppose $J\subseteq [n]$ has size $|J|\leq \frac{D}{2}$. Then for every $(w, w', w'') \in {\sf supp}(\mu|_{J})$, $\vec{a},\vec{b}, \vec{c} \in\prod\limits_{P\in\mathcal{P}}{\sf Image}(P)$ with 
       \begin{align*}
           S_{\vec{a}} &= \{ x\in \Sigma^{\overline{J}} \mid P(w, x) = a_P~~\forall P\in \mathcal{P}\},\\
         S_{\vec{b}} &= \{ y\in \Sigma^{\overline{J}} \mid P(w', y) = b_P~~\forall P\in \mathcal{P}\}, \\
           S_{\vec{c}} &= \{ z\in \Sigma^{\overline{J}} \mid P(w'', z) = c_P~~\forall P\in \mathcal{P}\}.
       \end{align*}
       and
        \[
        S = \{ (x, y, z) \in (\Sigma^{\overline{J}})^3 \mid x\in S_{\vec{a}}, y\in S_{\vec{b}}, z\in S_{\vec{c}}\},
        \]
        we have that
		\[
        \norm{\mathrm{T}_{1-\kappa}1_S - \mu(S)}_{L_2(\mu)} \leq 2^{-\Omega_{\alpha, \kappa, r, M}(D)}.
        \]
        %Here, the noise operator $T_{1-\kappa}$ on a tuple $(x, y, z)$ outputs $(x', y', z')$ where for each coordinate $i\in \bar{J}$ we have
%$(x'_i,y'_i,z'_i) = (x_i, y_i, z_i)$ with probability $1-\kappa$ and we take $(x'_i,y'_i,z'_i)\sim \mu$  with probability $\kappa$ independently.
	\end{lemma}
\begin{proof}
    The proof is mostly identical to the proof of Lemma~\ref{lemma:quasirandom_prop_2}. Let $o_P:={\sf ord}(P)$ and $\mathcal{P} = \{ P_1, P_2, \ldots, P_r\}$. We start by writing $1_S$ as follows.
\begin{align*}
         1_S(x, y, z) &= {\prod_{P\in \mathcal{P}}\left( \frac{1}{o_P} \sum_{j=0}^{o_P-1}\left(\overline{a_P}P(w,x)\right)^j\right)\cdot \left( \frac{1}{o_P} \sum_{j=0}^{o_P-1}\left(\overline{b_P}P(w',y)\right)^j\right)\cdot \left( \frac{1}{o_P} \sum_{j=0}^{o_P-1}\left(\overline{c_P}P(w'',z)\right)^j\right)}\\
          &= {\prod_{P\in \mathcal{P}} \frac{1}{o_P^3}\left(  \sum_{j, k, \ell=0}^{o_P-1}\overline{a_P}^j \overline{b_P}^k \overline{c_P}^\ell P(w,x)^j P(w',y)^k P(w'',z)^\ell\right)}\\
         &= \sum_{\substack{\vec{j} = (j_1, \ldots, j_r) \\ \vec{k} = (k_1, \ldots, k_r) \\ \vec{\ell} = (\ell_1, \ldots, \ell_r)}}{\left(  \prod_{i=1}^r \frac{1}{o_{P_i}^3}\overline{a_{P_i}}^{j_i} \overline{b_{P_i}}^{k_i} \overline{c_{P_i}}^{\ell_i} P_i(w,x)^{j_i} P_i(w',y)^{k_i} P_i(w'',z)^{\ell_i}\right)}
         \end{align*}
         Similar to the above proof, we let $\alpha_i = \prod_{i'\in J} P_{i,i'}(w_{i'})$ be the evaluation of $P_i$ at $w$ restricted to the coordinates in $J$. Similarly, let  $\beta_i = \prod_{i'\in J} P_{i,i'}(w'_{i'})$ and  $\gamma_i = \prod_{i'\in J} P_{i,i'}(w''_{i'})$. Also, let $\tilde{P} = P|_{\bar{J}}$ and $\tilde{\mathcal{P}} = \{ \tilde{P}_1, \tilde{P}_2, \ldots, \tilde{P}_r\}$. With these notations, we can rewrite the above expression as:
\begin{align*}
         1_S(x, y, z) &= \sum_{\substack{\vec{j} = (j_1, \ldots, j_r) \\ \vec{k} = (k_1, \ldots, k_r) \\ \vec{\ell} = (\ell_1, \ldots, \ell_r)}}{\left(  \prod_{i=1}^r \frac{1}{o_{P_i}^3}\overline{a_{P_i}}^{j_i} \overline{b_{P_i}}^{k_i} \overline{c_{P_i}}^{\ell_i} P_i(w,x)^{j_i} P_i(w',y)^{k_i} P_i(w'',z)^{\ell_i}\right)}\\
        &= \sum_{\substack{\vec{j} = (j_1, \ldots, j_r) \\ \vec{k} = (k_1, \ldots, k_r) \\ \vec{\ell} = (\ell_1, \ldots, \ell_r)}}{\left(  \prod_{i=1}^r \frac{1}{o_{P_i}^3}\overline{a_{P_i}}^{j_i} \overline{b_{P_i}}^{k_i} \overline{c_{P_i}}^{\ell_i} \alpha_i^{j_i}\beta_i^{k_i}\gamma_i^{\ell_i}\tilde{P}_i(x)^{j_i} \tilde{P}_i(y)^{k_i} \tilde{P}_i(z)^{\ell_i}\right)}.
         \end{align*}
         Let us denote the product function within the summand by $S_{(\vec{j}, \vec{k}, \vec{\ell})}$, i.e., 
     $$ S_{(\vec{j}, \vec{k}, \vec{\ell})}(x, y, z):= \prod_{i=1}^r \tilde{P}_i(x)^{j_i} \tilde{P}_i(y)^{k_i} \tilde{P}_i(z)^{\ell_i}.$$ 
We have
\begin{align*}
   1_S = \sum_{\substack{\vec{j} = (j_1, \ldots, j_r) \\ \vec{k} = (k_1, \ldots, k_r) \\ \vec{\ell} = (\ell_1, \ldots, \ell_r)}}
    \left(  \prod_{i=1}^r \frac{1}{o_{P_i}^3}\overline{a_{P_i}}^{j_i} \overline{b_{P_i}}^{k_i} \overline{c_{P_i}}^{\ell_i} \alpha_i^{j_i}\beta_i^{k_i}\gamma_i^{\ell_i} \right)       S_{(\vec{j}, \vec{k}, \vec{\ell})}
\end{align*}
          As  ${\sf softrk}_{\mu}(\mathcal{P})\geq D$ and $|J|\leq D/2$, we have ${\sf softrk}_{\mu}(\tilde{\mathcal{P}})\geq D/2$.  Let $\eta$ be the sum of all the terms in the summand when the function  $S_{(\vec{j}, \vec{k}, \vec{\ell})}$ is constant on the support of $\mu^{\otimes n'}$. We claim that $\mu(S)$ is close to $\eta$. To see this, as ${\sf softrk}_{\mu}(\mathcal{P})\geq D/2$, the only other terms in the summation are those with $\Delta_{{\sf symbolic}}(S_{(\vec{j}, \vec{k}, \vec{\ell})},1)\geq D/2$. Thus,
          \begin{align*}
               &\card{ \Expect{(x, y, z)\sim \mu^{\bar{J}}}{1_S(x, y, z)} -\eta}\\
              & \quad\quad\leq  \hspace{-20pt} \sum_{\substack{\vec{j}, \vec{k}, \vec{\ell} \\\Delta_{{\sf symbolic}}(S_{(\vec{j}, \vec{k}, \vec{\ell})},1)\geq D/2}}\hspace{-15pt}
    \left(  \prod_{i=1}^r \frac{1}{o_{P_i}^3}\overline{a_{P_i}}^{j_i} \overline{b_{P_i}}^{k_i} \overline{c_{P_i}}^{\ell_i} \alpha_i^{j_i}\beta_i^{k_i}\gamma_i^{\ell_i} \right)      \card{\Expect{(x, y, z)\sim \mu^{\bar{J}}}{ S_{(\vec{j}, \vec{k}, \vec{\ell})}(x, y, z)}}
          \end{align*}
          Using Fact~\ref{fact:highrankdecay} (1), for $S_{(\vec{j}, \vec{k}, \vec{\ell})}$ with $\Delta_{{\sf symbolic}}(S_{(\vec{j}, \vec{k}, \vec{\ell})},1)\geq D/2$, we have 
          $$\card{\Expect{(x, y, z)\sim \mu^{\bar{J}}}{ S_{(\vec{j}, \vec{k}, \vec{\ell})}(x, y, z)}}\leq  2^{-\Omega_{\alpha, M}(D)}.$$
          Thus,
          $$\card{ \mu(S) -\eta}\leq 2^{-\Omega_{\alpha,r, M}(D)}.$$
Applying the noise operator $\mathrm{T}_{1-\kappa}$ to the function $1_S$, we have
  \begin{align*}
               \mathrm{T}_{1-\kappa}1_S-\eta =  \hspace{-20pt} \sum_{\substack{\vec{j}, \vec{k}, \vec{\ell} \\\Delta_{{\sf symbolic}}(S_{(\vec{j}, \vec{k}, \vec{\ell})},1)\geq D/2}}\hspace{-15pt}
    \left(  \prod_{i=1}^r \frac{1}{o_{P_i}^3}\overline{a_{P_i}}^{j_i} \overline{b_{P_i}}^{k_i} \overline{c_{P_i}}^{\ell_i} \alpha_i^{j_i}\beta_i^{k_i}\gamma_i^{\ell_i} \right)      \mathrm{T}_{1-\kappa} S_{(\vec{j}, \vec{k}, \vec{\ell})}.
          \end{align*}
Observing the fact that the constant in front of $\mathrm{T}_{1-\kappa} S_{(\vec{j}, \vec{k}, \vec{\ell})}$ in the above summation is at most $1$ in absolute value, by the triangle inequality,
\begin{align*}
               \norm{\mathrm{T}_{1-\kappa}1_S-\eta}_{L^2(\mu)} \leq  \hspace{-20pt} \sum_{\substack{\vec{j}, \vec{k}, \vec{\ell} \\\Delta_{{\sf symbolic}}(S_{(\vec{j}, \vec{k}, \vec{\ell})},1)\geq D/2}} \norm{ \mathrm{T}_{1-\kappa} S_{(\vec{j}, \vec{k}, \vec{\ell})}}_{L^2(\mu)}.
          \end{align*}
 Using Fact~\ref{fact:highrankdecay} (2), we have $\norm{ \mathrm{T}_{1-\kappa} S_{(\vec{j}, \vec{k}, \vec{\ell})}}_{L^2(\mu)} \leq 2^{-\Omega_{\alpha,\kappa, M}(D)}$. Using this and the above established fact that $\card{ \mu(S) -\eta}\leq 2^{-\Omega_{\alpha,r, M}(D)}$, we have
  $$\norm{\mathrm{T}_{1-\kappa}1_S - \mu(S)}_{L^2(\mu)} \leq 2^{-\Omega_{\alpha,\kappa, r M}(D)},$$
  as required.
\end{proof}

\section{The Arithmetic Regularity Lemma}
In this section we state and prove our arithmetic regularity lemmas.
\subsection{The Basic Version}
Below is the basic version of our arithmetic regularity lemma. At first reading, the lemma should be interpreted as saying that $f$ can be approximated by $\mathrm{T}_{\mathcal{P},1-\eps}f$  in the $\|\cdot\|_{\nu}$ semi norm, for an appropriate choice of $\mathcal{P}$ and $\eps$. 
At this resolution, the result is identical to~\cite[Lemma 5.3]{BKM5}, which is insufficient for our purposes.
The key additional feature in the statement below is that $f$ is $\eta$-close to a different function  $\mathrm{T}_{\mathcal{P}',1-\eps'} f$ in the $\|\cdot\|_{\nu}$ semi-norm, where:
\begin{itemize}
    \item $\eta$ could be taken to be arbitrarily small compared to $\card{\mathcal{P}}$ and $\eps$ of the original crude estimator $\mathrm{T}_{\mathcal{P},\eps} f$;
    \item the crude estimator $\mathrm{T}_{\mathcal{P},\eps} f$ is close to the strong estimator $\mathrm{T}_{\mathcal{P}',\eps'} f$ in $\ell_2$ distance.
\end{itemize}
For technical reasons we also allow flexibility in the choice of the noise parameter, but the reader should ignore this and think of $\eps'$ as being fixed to $\eps_{i+2}$ below.
\begin{lemma}\label{lem:arithmetic_reg}
    Let $\Sigma$ be a finite alphabet of size at most $m$, let $\nu$ be a probability distribution over $\Sigma$ in which the probability of each atom is at least $\alpha>0$, and let $w\colon (0,1)\times\mathbb{N}\to(0,1)$ be a sufficiently rapid decay function. Then for all $\xi>0$, there exists  $k\in\mathbb{N}$, such that taking the sequence defined inductively as $\eps_0 = 1/2$ and $\eps_{i+1} = w(\eps_i,i)$, the following holds. For any $1$-bounded function $f\colon \Sigma^n\to\mathbb{C}$, for any $i_0\in\mathbb{N}$, there exists $i_0\leq i\leq k+i_0$ and collections of cyclic product functions $\mathcal{P}$, $\mathcal{P}'$ such that:
    \begin{enumerate}
        \item The collection $\mathcal{P}$ has size at most $1/\eps_{i}$, is $O_{m}(1)$-discrete, and for all $\eps'\in [\eps_{i+2},\eps_{i+1})$ we have
        \[
        \norm{f-\mathrm{T}_{\mathcal{P},1-\eps'}f}_{\nu}\leq \xi.
        \] 
        \item The collection $\mathcal{P}'$ has size at most $1/\eps_{i+4}$, is $O_{m}(1)$-discrete, and we have
        \[
        \norm{f-\mathrm{T}_{\mathcal{P}',1-\eps_{i+5}}f}_{\nu}\leq \eps_{i+3}.
        \]
        \item For all $\eps'\in [\eps_{i+2},\eps_{i+1})$ we have that 
        \[
        \norm{\mathrm{T}_{\mathcal{P}',1-\eps_{i+5}}f -\mathrm{T}_{\mathcal{P},1-\eps'} f}_2\leq \xi.
        \]
    \end{enumerate}
\end{lemma}
\begin{proof}
    Fix $\xi$ and $\alpha$ and take $d$ and $\delta>0$ from Theorem~\ref{thm:csp4} for $\xi$ and $\alpha$. 
    The proof is by two similar iterative processes, and we now describe them.
    Let $\eta_0=1$,   $\mathcal{P}_0 = \emptyset$, $g_0 = \mathrm{T}_{\mathcal{P}_0,1-\eta_0}f = \E[f]$,
    $f_0 = f-g_0$ and set $i=0$ and $u=i_0$. If there is $\eps'\in (\eps_{u+2},\eps_{u+1})$ such that $f-\mathrm{T}_{\mathcal{P}_{i},1-\eps'}f$ is $\delta$-correlated with a function of the form $L_i\cdot P_i$ where $L_i\colon \Sigma^{n}\to\mathbb{C}$ has degree at most $d$ and $2$-norm $1$ and $P_i$ is an $O_{m}(1)$-discrete product function, do the following:
    \begin{enumerate}
        \item Note that ${\sf ord}(P_i)\leq O_{m}(1)$, thus by Fact~\ref{fact:break_discrete_to_cyclic} we may write it as $Q_1\cdots Q_T$ where each $Q_i$ is a cyclic product function with order at most $O_{m}(1)$ and $T\leq O_{m}(1)$.   Define $\mathcal{P}_{i+1}=\mathcal{P}_i\cup \{Q_1,\ldots,Q_T\}$,
        $\eta_{i+1} = \eps'$,
        $g_{i+1} = \mathrm{T}_{\mathcal{P}_{i+1}, 1-\eta_{i+1}} f$ and $f_{i+1} = f-g_{i+1}$.
        \item Increase $i$ by $1$, increase $u$ by $2$ and repeat.
    \end{enumerate}
    We claim that this process terminates within $O(1/\delta^2)$ steps. This is a direct consequence of the following claim. 
    \begin{claim}\label{claim:it_process_key}
        If the process has not terminated by step $R$, 
        then $\norm{g_{R+1}}_2^2\geq R\frac{\delta^2}{10}$.
    \end{claim}
    \begin{proof}
        The proof is identical to the proof of~\cite[Claim 5.2]{BKM5}, and we give it for completeness. For each $i\leq R$ write $\mathrm{T}_i = \mathrm{T}_{\mathcal{P}_i,1-\eta_i}$ for simplicity, and choose complex numbers $\theta_i$ of absolute value $1$ such that 
        $\inner{f_i}{\theta_i L_i P_i}$ is real-valued and is at least $\delta$. Recalling that $f_i = (\mathrm{I}-\mathrm{T}_i)f$ and noting that $(\mathrm{I}-\mathrm{T}_i)$ is self adjoint, we conclude that
        \[
        \inner{f}{\theta_i (\mathrm{I}-\mathrm{T}_i)L_i P_i}\geq \delta.
        \]
        We now study the quantity $(\rom{1}) = \sum\limits_{i=1}^{R}\inner{\mathrm{T}_{R+1}f}{\theta_i (\mathrm{I}-\mathrm{T}_i)L_i P_i}$, and prove an upper bound and a lower bound on its absolute value.
        \paragraph{The lower bound:} since $\mathrm{T}_{R+1}$ is self adjoint, we can write
        \[
        (\rom{1}) = \inner{f}{\sum\limits_{i=1}^{R}\theta_i \mathrm{T}_{R+1}(\mathrm{I}-\mathrm{T}_i)L_i P_i}
        =
        \inner{f}{\sum\limits_{i=1}^{R}\theta_i (\mathrm{T}_{R+1}-\mathrm{T}_{R+1}\mathrm{T}_i)L_i P_i}.
        \]
        Fixing $1\leq i\leq R$, we analyze the inner product on the right hand side. 
        First, we have the pointwise equality $\mathrm{T}_{R+1} L_i P_i(x) = P_i(x)\mathrm{T}_{R+1} L_i$ (as the function $P_i$ is constant over $y\sim \mathrm{T}_{R+1} x$). Next, by Lemma~\ref{lem:act_on_ld} we have that 
        $\norm{(\mathrm{I}-\mathrm{T}_{R+1})L_i}\lll_{d,m,\card{\mathcal{P}_{R+1}},\alpha} \eta_{R+1}^{1/3}$. Altogether, it follows that $\norm{\mathrm{T}_{R+1} L_i P_i - L_i P_i}_2\leq \eta_{R+1}^{1/6}$. 
        We now claim in a similar fashion that 
        $\norm{\mathrm{T}_{R+1}\mathrm{T}_iL_i P_i - \mathrm{T}_iL_i P_i}_2\leq \eta_{R}^{1/6}$. First, note that by Lemma~\ref{lem:approx_formula} we may find a function $h_i$ of the form $\sum\limits_{P'\in \spn(\mathcal{P}_i)}L_{P'}P'$ such that $\norm{\mathrm{T}_iL_i P_i-h_i}\leq \eta_R$, and each $L_{P'}$ has degree and $2$-norm at most $O_{\eta_R}(1)$. Noting that by Lemma~\ref{lem:act_on_ld} we have $\norm{\mathrm{T}_{R+1} L_{P'} P' - L_{P'} P'}_2\lll_{\card{\mathcal{P}_{R+1}},\eta_R} \eta_{R+1}^{1/3}$, we conclude that
        \[
    \norm{\mathrm{T}_{R+1}\mathrm{T}_iL_i P_i - \mathrm{T}_iL_i P_i}_2
    \leq
    \norm{\mathrm{T}_{R+1} h_i - h_i}
    +2\norm{L_iP_i - h_i}_2
    \leq 
    \sum\limits_{P'}\norm{\mathrm{T}_{R+1} L_{P'}P' - L_{P'}P'}_2
    +2\eta_R,
        \]
        which is at most $O_{\card{\mathcal{P}_{R+1}},\eta_R}(\eta_{R+1}^{1/3}) + 2\eta_R\leq \eta_R^{1/6}$. 
        We conclude that 
        \begin{align*}
        \card{(\rom{1})}
        &\geq 
        \inner{f}{\sum\limits_{i=1}^{R}\theta_i (\mathrm{I}-\mathrm{T}_i)L_i P_i}
        -\sum\limits_{i=1}^{R}
        \norm{(\mathrm{T}_{R+1}-\mathrm{T}_{R+1}\mathrm{T}_i)L_i P_i-(\mathrm{I}-\mathrm{T}_i)L_i P_i}_2
        \\
        &\geq 
        R\delta - 
        \sum\limits_{i=1}^{R} \eta_{R+1}^{1/6} + \eta_{R}^{1/6}\\
        &\geq 0.99 R\delta.
        \end{align*}
        \paragraph{The upper bound:} by Cauchy-Schwarz
        \[
        \card{(\rom{1})}
        \leq \norm{\mathrm{T}_{R+1}f}_2
        \norm{\sum\limits_{i=1}^{R}\theta_i (\mathrm{I}-\mathrm{T}_i)L_i P_i}_2.
        \]
        Looking at the second term, taking square and expanding, we get that it is at most
        \begin{equation}\label{eq:iterative_upper_bd}
        \sum\limits_{i=1}^{R}
        \norm{(\mathrm{I}-\mathrm{T}_i)L_i P_i}_2^2
        +2\sum\limits_{i'<i}
        \card{
        \inner{
        (\mathrm{I}-\mathrm{T}_i)L_i P_i}{(\mathrm{I}-\mathrm{T}_{i'})L_{i'} P_{i'}}}.
        \end{equation}
        The first sum is at most $4R$ as 
        $\norm{(\mathrm{I}-\mathrm{T}_i)L_i P_i}_2\leq 2$ for each $i$. For the second term, fix $i'<i$. Then by self adjointness of $(\mathrm{I}-\mathrm{T}_i)$ and Cauchy-Schwarz
        \[
        \card{\inner{
        (\mathrm{I}-\mathrm{T}_i)L_i P_i}{(\mathrm{I}-\mathrm{T}_{i'})L_{i'} P_{i'}}}
        \leq \norm{L_i P_i}_2\norm{(\mathrm{I}-\mathrm{T}_i)(\mathrm{I}-\mathrm{T}_{i'})L_{i'} P_{i'}}_2\leq 
        \norm{(\mathrm{I}-\mathrm{T}_i)(\mathrm{I}-\mathrm{T}_{i'})L_{i'} P_{i'}}_2.
        \]
        An argument identical to the one in the ``upper bound'' shows that 
        $\norm{\mathrm{T}_i(\mathrm{I}-\mathrm{T}_{i'})L_{i'} P_{i'}-(\mathrm{I}-\mathrm{T}_{i'})L_{i'} P_{i'}}_2\leq \eta_{i-1}^{1/6}$. Plugging this into~\eqref{eq:iterative_upper_bd} and then back up gives that
        \[
        \card{(\rom{1})}\leq \norm{\mathrm{T}_{R+1}f}_2
        \sqrt{4R+\sum\limits_{i}i\eta_{i-1}^{1/6}}
        \leq 
        \norm{\mathrm{T}_{R+1}f}_2\sqrt{5R}.
        \]

        \paragraph{Combining the bounds:} putting the upper bound and the lower bound on $\card{(\rom{1})}$ together gives that 
        $0.99 R \delta\leq \card{(\rom{1})}\leq \norm{\mathrm{T}_{R+1}f}_2\sqrt{5R}$, and simplifying yields the statement of the claim.
    \end{proof}
    As $\norm{g_i}_2\leq \norm{f}_2\leq 1$ always, the process terminates after $i^{\star}\leq O(1/\delta^2)$ steps and let $u$ be the index of $\eps$ in the end. At the end of the process we have found $\mathcal{P}_{i^{\star}}$  which satisfies the first item (with $\eps_{u+2},\eps_{u+1}$). It is unlikely to be a stronger approximation as required in the second item, though.
    To fix that, the idea is to run more rounds of the iterative process, but with different parameters so as to find a stronger approximation. We call this the secondary process.
    
    Each run of the secondary process has a parameter $s$, which in the beginning is set to $s=u$. We then take $\xi'=\eps_{s+3}$ and choose $d'$ and $\delta'$ from Theorem~\ref{thm:csp4} for $\xi'$ and $\delta'$. Now run the process above with $\xi'$ and $\delta'$ from the point we stopped it. Then the process would terminate after at most $ O(1/\delta'^2)$ steps, and we would reach $j$ such that have that $\norm{f_j}_{\nu}\leq \xi'$. We thus get the candidate decomposition as in the statement with $\mathcal{P} = \mathcal{P}_{i^{\star}}$ with $\eta_{i^{\star}}$, and  
    $\mathcal{P}' = \mathcal{P}_{j}$ with $\eta_j$, and it is clear that the first two items hold. If the third item holds we terminate, and otherwise we take $s = j+2$ and run the secondary process again.

    If this iterative process terminates within $100/\xi^2$ steps, then it is clear we get a decomposition as needed. The key claim now is that if the secondary process runs too many times, we can still find a decomposition as per the following claim.
    \begin{claim}\label{claim:arithmetic_reg_2nd}
      Suppose that the secondary process has not terminated within $S$ runs, and that $S\geq \frac{100}{\xi^4}$. Then we can find $j\leq S-2$ such that taking $\mathcal{P} = \mathcal{P}_j$ and $\mathcal{P}'=\mathcal{P}_{j+2}$ satisfies the claim.
    \end{claim}
    \begin{proof}
        Fix $j_1,\ldots,j_S$, where $j_t$ is the index $j$ at the end of the execution of the secondary process when it is run for the $t$th time. 
        Consider the functions $ g_{j_1},\ldots,g_{j_S}$, and for simplicity we denote them by $G_1,\ldots,G_{S}$ and the corresponding noise operators by $\mathrm{T}_{1},\ldots,\mathrm{T}_{S}$. 
        
        Inspect the quantity 
        $(\rom{1}) = \sum\limits_{\ell<S/2}\norm{(\mathrm{T}_{2\ell}-\mathrm{T}_{2(\ell-1)})f}_2^2$. Then by self adjointness we may write it as
        \begin{align*}
        (\rom{1})
        =
        \sum\limits_{\ell<S/2}\inner{(\mathrm{T}_{2\ell}-\mathrm{T}_{2(\ell-1)})^2f}{f}
        &=\inner{\sum\limits_{\ell<S/2}(\mathrm{T}_{2\ell}-\mathrm{T}_{2(\ell-1)})^2f}{f}\\
        &\leq \norm{\sum\limits_{\ell<S/2}(\mathrm{T}_{2\ell}-\mathrm{T}_{2(\ell-1)})^2f},
        \end{align*}
        where the last inequality is by Cauchy-Schwarz. Squaring the last quantity and expanding, we get it is at most
        \begin{equation}\label{eq:sec_proc_2}
        \sum\limits_{\ell}\norm{(\mathrm{T}_{2\ell}-\mathrm{T}_{2(\ell-1)})^2f}_2^2+2\sum\limits_{\ell'<\ell}\card{\inner{(\mathrm{T}_{2\ell}-\mathrm{T}_{2(\ell-1)})^2f}{(\mathrm{T}_{2\ell'}-\mathrm{T}_{2(\ell'-1)})^2f}}.
        \end{equation}
        For the first sum, each summand is at most $16$. As per the second sum, fix $\ell'<\ell$ and note that 
        \[
        \card{\inner{(\mathrm{T}_{2\ell}-\mathrm{T}_{2(\ell-1)})^2f}{(\mathrm{T}_{2\ell'}-\mathrm{T}_{2(\ell'-1)})^2f}}
        =\card{\inner{(\mathrm{T}_{2\ell}-\mathrm{T}_{2(\ell-1)})f}{(\mathrm{T}_{2\ell}-\mathrm{T}_{2(\ell-1)})(\mathrm{T}_{2\ell'}-\mathrm{T}_{2(\ell'-1)})^2f}},
        \]
        which by Cauchy-Schwarz and the crude bound $\norm{(\mathrm{T}_{2\ell}-\mathrm{T}_{2(\ell-1)})f}_2\leq 2$, is at most 
        \begin{equation}\label{eq:sec_proc_1}
        2\norm{(\mathrm{T}_{2\ell}-\mathrm{T}_{2(\ell-1)})(\mathrm{T}_{2\ell'}-\mathrm{T}_{2(\ell'-1)})^2f}_2.
        \end{equation}
        Repeating the argument as in the ``lower bound'' of Claim~\ref{claim:it_process_key}, we get that 
        \begin{align*}
        &\norm{\mathrm{T}_{2\ell}(\mathrm{T}_{2\ell'}-\mathrm{T}_{2(\ell'-1)})^2f-(\mathrm{T}_{2\ell'}-\mathrm{T}_{2(\ell'-1)})^2f}_2\leq \eta_{2\ell-1}^{1/6},\\
        &\norm{\mathrm{T}_{2(\ell-1)}(\mathrm{T}_{2\ell'}-\mathrm{T}_{2(\ell'-1)})^2f-(\mathrm{T}_{2\ell'}-\mathrm{T}_{2(\ell'-1)})^2f}_2\leq \eta_{2\ell-2}^{1/6},
        \end{align*}
        which by the triangle inequality gives that 
        $\eqref{eq:sec_proc_1}\leq 2\eta_{2\ell-2}^{1/6}$. Plugging this into~\eqref{eq:sec_proc_2} gives that 
        \[
        \eqref{eq:sec_proc_2}\leq 16\left(\frac{S}{2}-1\right)+
        \sum\limits_{\ell}2\ell \eta_{2\ell-2}^{1/12}
        \leq 8S,
        \]
        and plugging this above gives that $(\rom{1})\leq \sqrt{8S}$. 
        By definition of $(\rom{1})$ 
        we conclude that there is $\ell$ such that
        \[
        \norm{\mathrm{T}_{2\ell} f - \mathrm{T}_{2\ell-2}f}_2^2\leq \frac{2\sqrt{8}}{\sqrt{S}}.
        \] 
    Pick $\mathcal{P} = \mathcal{P}_{2\ell-2}$ and 
    $\mathcal{P}'=\mathcal{P}_{2\ell}$. The first two items of the claim hold by the fact that the secondary process terminated at times $j_{2\ell-2}$ and $j_{2\ell}$. The third property follows as
  \[
        \norm{\mathrm{T}_{2\ell} f - \mathrm{T}_{2\ell-2}f}_2^2\leq \frac{2\sqrt{8}}{\sqrt{S}}.
        \leq \frac{2\sqrt{8}}{10/\xi^2}\leq \xi^2.
        \qedhere
        \] 
    \end{proof}
    Thus, if the secondary process was run for at least $100/\xi^4$ steps we are done by Claim~\ref{claim:arithmetic_reg_2nd}, and the proof is concluded.
\end{proof}

\subsection{The Version with High Rank}
We now state and prove a version of Lemma~\ref{lem:arithmetic_reg}, in which the collection $\mathcal{P}$ is guaranteed to have high rank and high softrank. 
\begin{lemma}\label{lem:arithmetic_reg_high_rank}
    Let $\Sigma$ be a finite alphabet of size at most $m$, let $\nu$ be a probability distribution over $\Sigma$ in which the probability of each atom is at least $\alpha>0$, 
    let $\mu$ be a distribution over $\Sigma^3$ with marginals equal to $\nu$, and let $w\colon (0,1)\to(0,1)$ be a sufficiently rapid decay function. Then for every $\xi>0$, and for every $1$-bounded function $f:\Sigma^n\to\mathbb{C}$, there exist $w$-separated scales
    \[
    \rho_5\ll_{w}\rho_4\ll_w\rho_3\ll_w\rho_2\ll_w\rho_1\ll_w\rho_0
\]
satisfying $\rho_0\leq \xi$ and $\rho_5\geq \Omega_{w,m,\alpha,\xi}(1)$, and collections of cyclic product functions $\mathcal{P}$, $\mathcal{P}'$ such that

    %Then for all $\xi>0$, there exists $s\in\mathbb{N}$, such that taking the sequence defined inductively as $\eps_0=1/2$ and $\eps_{i+1} = w(\eps_i)$, the following holds. For any $1$-bounded function $f\colon \Sigma^n\to\mathbb{C}$, for any $i_0\in\mathbb{N}$, there exists $i_0\leq i\leq s+i_0$ and collections of cyclic product functions $\mathcal{P}$, $\mathcal{P}'$ such that:
    \begin{enumerate}
        \item The collection $\mathcal{P}$ satisfies:
        \begin{enumerate}
            \item $\card{\mathcal{P}}\leq 1/\rho_0$.
            \item $\mathcal{P}$ is $O_{m}(1)$-discrete.
            \item $\norm{f-\mathrm{T}_{\mathcal{P},1-\rho_{1}}f}_{\nu}\leq \xi$.
            \item ${\sf rk}(\mathcal{P})\geq \frac{1}{\rho_{2}}$.
            \item ${\sf softrk}_{\mu}(\mathcal{P})\geq \frac{1}{\rho_{2}}$.
        \end{enumerate}
        \item The collection $\mathcal{P}'$ has size at most $1/\rho_{4}$, is $O_{m}(1)$-discrete and we have $\norm{f-\mathrm{T}_{\mathcal{P}',1-\rho_{5}}f}_{\nu}\leq \rho_{3}$.
        \item We have that 
        $\norm{\mathrm{T}_{\mathcal{P}',1-\rho_{5}}f -\mathrm{T}_{\mathcal{P},1-\rho_{1}} f}_2\leq \xi$.
    \end{enumerate}
\end{lemma}
\begin{proof}
     We first define a very rapidly decaying function $w'\colon (0,1)\times\mathbb{N}\to (0,1)$ with which we will want to apply Lemma~\ref{lem:arithmetic_reg}.
    Set $\tilde{w}(\eps,i) = w(\eps)$. On input $(\eps,i)$, first compute $\eps_i$ and $M = O_{m}(1)$ from Lemma~\ref{lem:arithmetic_reg} for $\tilde{w}$. 
    Then, take $k_i$ from Lemma~\ref{lem:bothrank_gain} for $M$-discrete collections of size at most $1/\eps_i$. Finally, take $w'(\eps,i) = 100^{2(k_1+\ldots+k_i)}\cdot w^{2(k_1+\ldots+k_i)}(\eps)$. For ease of notation,  throughout we denote $\eta_0$ and $\eta_{i+1} = w'(\eta_i,i)$ as used in Lemma~\ref{lem:arithmetic_reg} for the decay function $w'$ and $\xi/4$.
    
    Applying Lemma~\ref{lem:arithmetic_reg} we find $i$, $\mathcal{P}$ and $\mathcal{P}'$ satisfying the properties therein with the decay function $w'$. Set $k=k_1+\ldots+k_i$ and consider the sequence 
    $D_0 \leq D_1 \leq \ldots\leq D_k$ given as 
    \[
    D_0 = \frac{1}{\eta_{i+1}},
    \qquad
    D_1 = 2w^{(2)}(D_0^{-1})^{-1},
    \qquad
    \ldots,
    \qquad
    D_{k} = 2w^{(2)}(D_{k-1}^{-1})^{-1}.
    \]
    We note that $D_{k}\leq \eta_{i+2}^{-1}$.
    By Lemma~\ref{lem:bothrank_gain} we may find $1\leq j\leq k$ and a collection $\widetilde{\mathcal{P}}$ of $O_{m}(1)$-discrete cyclic product functions such that
    ${\sf rk}(\widetilde{\mathcal{P}}),{\sf softrk}(\widetilde{\mathcal{P}})\geq D_{j}-(D_1+\ldots+D_{j-1})$, each function in $\mathcal{P}$ is $3D_{j-1}$-close to a function in $\spn(\widetilde{\mathcal{P}})$, and each function in $\widetilde{\mathcal{P}}$ 
    is $3D_{j-1}$-close to a function in $\spn(\mathcal{P})$. We prove that the collections $\widetilde{\mathcal{P}}$ and $\mathcal{P}'$ satisfy the properties of the lemma with the scales 
    \[
\rho_0=\eta_i,\qquad
\rho_1=w(D_{j-1}^{-1}),\qquad
\rho_2=w(\rho_1),\qquad
\rho_3=\eta_{i+3},\qquad
\rho_4=\eta_{i+4},\qquad
\rho_5=\eta_{i+5}.
\]

    Item 1(a) follows because 
    $|\widetilde{\mathcal{P}}|\leq \card{\mathcal{P}}\leq \eta_{i}^{-1}$ by Lemma~\ref{lem:arithmetic_reg}. Item 1(b) follows by Lemma~\ref{lem:bothrank_gain}.  For item 1(c), by Fact~\ref{fact:noise_op_basic_prop2} we get that
    \begin{equation}\label{eq:reg_with_rank}
    \norm{\mathrm{T}_{\widetilde{\mathcal{P}},1-\rho_1}f-\mathrm{T}_{\mathcal{P},1-\rho_1}f}_2
    \leq 
    \sqrt{2(\card{\mathcal{P}}+|\widetilde{\mathcal{P}}|)\cdot 3D_{j-1}\cdot \rho_1}
    \leq \rho_1^{1/4},
    \end{equation}
    and by the first item of Lemma~\ref{lem:arithmetic_reg}, the triangle inequality and the trivial $\|g\|_{\nu}\leq \norm{g}_2$, we get that
    \[
    \norm{f-\mathrm{T}_{\mathcal{P},1-\rho_1}f}_{\nu}\leq \frac{\xi}{4}+\rho_1^{1/4}
    \leq \xi.
    \]
    
    Item 3 follows by combining the third item in Lemma~\ref{lem:arithmetic_reg} and~\eqref{eq:reg_with_rank} with the triangle inequality. Item 2 follows from the second item in Lemma~\ref{lem:arithmetic_reg}.
    Lastly, for items 1(d) and 1(e), since $w$ is sufficiently rapidly decaying we get that
    ${\sf rk}(\widetilde{\mathcal{P}}),{\sf softrk}(\widetilde{\mathcal{P}})\geq D_{j}/2\geq \rho_2^{-1}$.
\end{proof}

\section{Proof of the Counting Lemma}\label{sec:proof_of_main}

Apart from the arithmetic regularity lemma that we proved in the previous section, we also need the invariance principle from~\cite{Mossel} and a multidimensional Gaussian reverse hypercontractivity theorem due to~\cite{ChenDafnisPaouris2015HolderGaussian}. To state these results, we first give some preliminaries.

\subsection{Preliminaries}
\subsubsection{Fourier Analysis}
It will be convenient to work with the Fourier basis for the space of functions $L^2(\Sigma^n, \nu^{\otimes n})$. The definitions below and the proofs of the propositions can be found in~\cite{O14}.

\begin{definition}
	An $n$-dimensional \textit{multi-index} is a tuple $\alpha \in \mathbb{N}^n$. We write
	\[
	{\sf supp}(\alpha) = \{i : \alpha_i \neq 0\}, \quad \#\alpha = |\operatorname{\sf supp}(\alpha)|, \quad |\alpha| = \sum_{i=1}^{n} \alpha_i.
	\]
	We write $\alpha \in \mathbb{N}^n_{<m}$ when we want to emphasize that each $\alpha_i \in \{0,1,\dots,m-1\}$.
\end{definition}

\begin{definition}\label{def:fourier_basis}
	Given functions $\phi_0, \dots, \phi_{m-1} \in L^2(\Sigma, \nu)$ and a multi-index $\alpha \in \mathbb{N}^n_{<m}$, we define $\phi_{\alpha} \in L^2(\Sigma^n, \nu^{\otimes n})$ by
	\[
	\phi_{\alpha}(\V x) = \prod_{i=1}^{n} \phi_{\alpha_i}(x_i).
	\]
\end{definition}

\begin{proposition}\label{prop:fourier_n}
	Let $\phi_0, \dots, \phi_{m-1}$ be a orthonormal basis for $L^2(\Sigma, \nu)$ where $\phi_0$ is the constant $1$ function. Then the collection $(\phi_{\alpha})_{\alpha \in \mathbb{N}^n_{<m}}$ is an orthornormal basis for $L^2(\Sigma^n, \nu^{\otimes n})$.
\end{proposition}
\subsubsection{Influences}

Next, we define a few basic notions from the analysis of Boolean functions.
\begin{definition}\label{def:influence}
	For a function $f\colon (\Sigma^n,\nu^{\otimes n})\to\mathbb{C}$ and a coordinate $i\in [n]$, the influence of
	$i$ is defined as
	\[
	I_i[f, \nu^{\otimes n}] = \Expect{\substack{y\sim \nu^{\otimes (n-1)}\\a,b\sim \nu}}{\card{f(x_{-i} = y, x_i = a)-f(x_{-i} = y, x_i = b)}^2}.
	\]
    Here, $(x_{-i} = y, x_i = a)$ is the point in $\Sigma^n$ whose $i$th coordinate is $a$, and the rest of its coordinates are filled according to $y$.
	The total influence of a function is defined as $I[f, \nu^{\otimes n}] = \sum_{i=1}^n I_i[f, \nu^{\otimes n}]$.
\end{definition}
We have the following proposition.
\begin{proposition}\label{prop:total_inf}
	For a function $f\colon (\Sigma^n,\nu^{\otimes n})\to\mathbb{C}$, we have
	\[
    I[f, \nu^{\otimes n}] = \sum_{\alpha \in \mathbb{N}^n_{<m}} \card{\alpha}\card{\hat{f}(\alpha)}^2,\quad\quad I_i[f, \nu^{\otimes n}]  = \sum_{\substack{\alpha \in \mathbb{N}^{n}_{<m}\\ \alpha_i \neq 0}} \card{\hat{f}(\alpha)}^2,
    \]
	where $\hat{f}(\alpha) := \langle f, \phi_\alpha\rangle$ and $\card{\alpha}$ is the number of non-zero coordinates in $\alpha$.
\end{proposition}
For a function $f$ that is a sum of small functions with small influences, the following proposition upper bounds the influences of the function $f$.
\begin{proposition}
    \label{prop:sum_inf}
    Let $f\colon (\Sigma^n,\nu^{\otimes n})\to\mathbb{C}$ be such that 
    $$f(x) = \sum_{j=1}^{R} \theta_j \cdot 
    f_j(x)$$
    where $f_j\colon (\Sigma^n,\nu^{\otimes n})\to\mathbb{C}$ and $|\theta_j|\leq 1$ for all $1\leq j\leq R$, then for all $i\in [n]$, we have $I_i[f, \nu^{\otimes n}] \lesssim_R \sum_{j}I_i[f_j, \nu^{\otimes n}]$.
\end{proposition}
\begin{proof}
We expand $I_i(f)$ as follows.
 \begin{align*}
       I_i(f)&= \Expect{\substack{\V y\sim \nu^{\otimes (n-1)}\\ a, b \sim \nu} }{\card{ f( \V x_{-i} = \V y, x_i = a) - f(\V x_{-i} = \V y, x_i = b) }^2}\\
       & = \Expect{\substack{\V y\sim \nu^{\otimes (n-1)}\\ a, b \sim \nu} }{\card{\sum_{j=1}^R \theta_j (f_j( \V x_{-i} = \V y, x_i = a) - f_j(\V x_{-i} = \V y, x_i = b))}^2}\\
       & \lesssim_R \Expect{\substack{\V y\sim \nu^{\otimes (n-1)}\\ a, b \sim \nu} }{\sum_{j=1}^R \card{\theta_j (f_j( \V x_{-i} = \V y, x_i = a) - f_j(\V x_{-i} = \V y, x_i = b))}^2},
       \end{align*}
       where the last inequality is by the Cauchy-Schwarz inequality. Continuing, we get that
       \begin{align*}
       I_i(f)\lesssim_R 
       %\Expect{\substack{\V y\sim \nu^{\otimes (n-1)}\\ a, b \sim \nu} }{\sum_{j=1}^R \card{\theta_j f_j( \V x_{-i} = \V y, x_i = a) - f_j(\V x_{-i} = \V y, x_i = b)}^2}\\& \leq  
       \Expect{\substack{\V y\sim \nu^{\otimes (n-1)}\\ a, b \sim \nu} }{\sum_{j=1}^R \card{f_j( \V x_{-i} = \V y, x_i = a) - f_j(\V x_{-i} = \V y, x_i = b)}^2}
          = \sum_{j}I_i[f_j, \nu^{\otimes n}],
   \end{align*}
   as required.
\end{proof}
\subsubsection{The Invariance Principle}\label{sec:invariance_principle}
In this section, we present the invariance principle of~\cite{MOO}, and we begin with some set-up.
Suppose that $\Sigma,\Gamma,\Phi$ are finite alphabets of sizes $m_1,m_2,m_3$ respectively, and
$\mu$ is a probability measure over $\Sigma\times \Gamma\times \Phi$ in which the probability of
each atom is at least $\alpha>0$. We set up orthonormal bases for $(\Sigma,\mu_x)$,
$(\Gamma,\mu_y)$ and $(\Phi,\mu_z)$ given by $v_0,\ldots,v_{m_1-1}$,
$u_0,\ldots,u_{m_2-1}$ and $w_0,\ldots,w_{m_3-1}$. Consider the ensemble of random
variables
\[
\mathcal{X} = \{v_1(x),\ldots,v_{m_1-1}(x),u_1(y),\ldots,u_{m_2-1}(y),w_1(z),\ldots,w_{m_3-1}(z)\}
\]
where $(x,y,z)\sim \mu$. We define the covariance matrix
$P\in\mathbb{C}^{(m_1+m_2+m_3-3)\times(m_1+m_2+m_3-3)}$ whose rows and columns correspond
to the functions in $\mathcal{X}$, and the entry corresponding to tow random variables
in $\mathcal{X}$.
For example, for $v_i$ and $u_j$,
the corresponding entry is
\[
P(v_i,u_j) = \Expect{(x,y,z)\sim \mu}{v_i(x)\overline{u_j(y)}}.
\]

Let $L^2(\mathbb{C}^n, \gamma^n)$ be the inner product space of functions with standard complex $\mathcal{C}\mathcal{N}(0, 1)$ Gaussian measure.\footnote{A complex random variable $Z$ whose real and imaginary parts are independent normally distributed random variables with mean zero and variance $\frac{1}{2}$.} We define an ensemble of standard {\em proper}\footnote{An ensemble $\mathcal{G}$ of standard complex Gaussian variables is proper if it satisfies $\Expect{}{\mathcal{G}\mathcal{G}^T}=0$.} complex Gaussian variables having the same covariance matrix $P$ as follows.  Let $Z\sim \mathcal{C}\mathcal{N}(0, I_{m_1+m_2+m_3-3})$ be $(m_1+m_2+m_3-3)$ independent standard complex Gaussian variables. We define 
\[
\mathbf{\mathcal{G}} = \{G_{1,x},\ldots,G_{m_1-1, x},\ldots, G_{1,z},\ldots,G_{m_3-1, z}\}
\]
as $\mathcal{G}:= P^{1/2}Z$ to be an ensemble of centered proper Gaussian random variables with covariance matrix $P$. Note that $\mathcal{G}$ satisfies (i) $\Expect{}{\mathcal{G}\mathcal{G}^T}=0$ and (ii) $\Expect{}{\mathcal{G}\mathcal{G}^*}=P$. The invariance principle relates
the behavior of low-influence, multi-linear polynomials over $\mathcal{X}$ and over $\mathbf{\mathcal{G}}$. Below, we state
the version that we need from~\cite{Mossel} specialized to our case of interest, but before that, we need a few definitions.

Denote $q = m_1+m_2+m_3-3$, and let $M\colon \mathbb{C}^{q n}\to\mathbb{C}$ be a multi-linear polynomial given as
\[
M(a_{1,1},\ldots,a_{1,q},\ldots, a_{n,1},\ldots, a_{n,q})
=\sum\limits_{T\subseteq [n]\times [q]}m_T \prod\limits_{(i,j)\in T} a_{i,j}.
\]

\begin{definition}
	The degree of $M$ as above is defined as
	\[
	{\sf deg}(M) = \max_{T}\{ \card{T} \mid M_T \neq 0\}.
	\]
\end{definition}

\begin{definition}
	The influence of variable $(i,j)$ on $M$ as above is defined as
	\[
	I_{i,j}[M] = \sum\limits_{T\ni (i,j)}\card{m_T}^2.
	\]
\end{definition}
We will also consider vector-valued multi-linear functions, which are functions
$M\colon \mathbb{C}^{qn}\to\mathbb{C}^{k}$ wherein each $M_s$ is a multi-linear function.
The influence of $(i,j)$ on $M$ is defined as $I_{i,j}[M] = \max_{s}I_{i,j}[M_s]$.

Define ${\sf trunc}_{[0,1]}\colon\mathbb{C}\to [0,1]$ as follows:
\begin{align*}
	\mathrm{trunc}_{[0,1]}(a) =  \begin{cases}
		0 \qquad  \qquad \mbox{if $\operatorname{Re}$($a$)} < 0,\\
		\mbox{$\operatorname{Re}$($a$)} \qquad  \mbox{ if } 0 \leq \mbox{$\operatorname{Re}$($a$)} \leq 1,\\
		1  \qquad \qquad  \mbox{ if } \mbox{$\operatorname{Re}$($a$)} > 1.
	\end{cases}
\end{align*}

The invariance principle theorem from Mossel~\cite{Mossel} is as follows.
\begin{thm}\label{thm:invariance_principle}
	For all $\alpha>0$, $k,m\in\mathbb{N}$, $d\in\mathbb{N}$, $C>0$ and $\eps>0$, there exists $\tau>0$
	such that the following holds. Suppose that $\card{\Sigma} = m_1, \card{\Gamma} = m_2,\card{\Phi}= m_3$,
	that $\mu$ is a distribution over $\Sigma\times\Gamma\times\Phi$ in which the probability
	of each atom is at least $\alpha$, and let $\mathcal{X}$ and $\mathcal{G}$ be the ensembles
	of random variables as above with the same covariance matrix. Let $M\colon\mathbb{C}^{qn}\to\mathbb{C}^k$
	is a multi-linear polynomial with $\max_{i,j}I_{i,j}[M]\leq \tau$, ${\sf deg}(M) \leq d$ and $\|M\|_2 \leq C$.
	\begin{enumerate}
		\item 
         If $\Psi\colon \mathbb{C}^{k}\to\mathbb{C}$ is differentiable three times and its third order derivatives
		are at most $C$ in absolute value, then
		\[
		\card{\Expect{}{\Psi(M(\mathcal{X}^n))}-\Expect{}{\Psi(M(\mathbf{\mathcal{G}}^n))}}\leq \eps.
		\]
		 \item Define $\zeta\colon \mathbb{C}^k\to \mathbb{R}$ by $\zeta(a_1,\ldots,a_k) = \sqrt{\sum\limits_{i=1}^{k}\card{{\sf trunc}_{[0,1]}(a_i) - a_i}^2}$.
		Then
		\[
		\card{\Expect{}{\zeta(M(\mathcal{X}^n))}-\Expect{}{\zeta(M(\mathcal{G}^n))}}\leq \eps.
		\]
	\end{enumerate}
\end{thm}

We will also need the following fact.
\begin{fact}\label{fact:trivial_lipshitz_pf}
	Consider the function $\zeta\colon \mathbb{C}\to [0,\infty)$ defined as
	$\zeta(a) = \card{{\sf trunc}_{[0,1]}(a) - a}$. Then $\zeta$ is $2$-Lipshitz function.
\end{fact}
\begin{proof}
	Let $a,b\in\mathbb{C}$; we show that $\card{\zeta(a) - \zeta(b)} \leq 2\card{a-b}$. It is enough to show this for the real numbers $a$ and $b$.  
	If $a,b$ are both on $[0,1]$, then the left-hand side is
	$0$ and the claim is trivial. If exactly one of $a$ and $b$ is in $[0,1]$, say $a$ then
	\[
	\card{\zeta(a) - \zeta(b)} =
	\zeta(b)
	\]
	If $b>1$, then $\zeta(b) = b-1 \leq \card{b-a}$. Otherwise, $b<0$, and $\zeta(b) = |b|\leq \card{b-a}$. It remains to consider the case that both $a$ and $b$ are not in $[0,1]$. When $a, b<0$, then $\card{\zeta(a) - \zeta(b)} = \card{a}-\card{b} =  \card{a-b}$. If $a, b>1$, then $\card{\zeta(a) - \zeta(b)} = \card{a-b}$. Finally, if $a<0$ and $b>1$, we have $\zeta(a) - \zeta(b) = -a+(b-1)\leq 2\card{a-b}$,
	and the proof is concluded.
\end{proof}

\subsubsection{Gaussian Reverse Hypercontractivity}
Finally, we will need the multidimensional version of the Gaussian reverse hypercontractivity theorem from~\cite{ChenDafnisPaouris2015HolderGaussian}. We use the following special case of the theorem from~\cite{Mos}. 

\begin{restatable}{thm}{MGRH}{}\label{thm:multi_reverse_hyp}
	Let $\eta>0$ and let $\mathbf{\mathcal{G}} = (\mathcal{G}_1,\ldots,\mathcal{G}_{\ell})$ be a jointly proper Gaussian collection of $\ell$ random vectors such that:
	\begin{enumerate}
		\item For each $i \in [\ell]$, $\mathcal{G}_i = \left(G_{i,1}, \ldots, G_{i,n}\right)$ is a random vector distributed as $n$ independent $\mathcal{C}\mathcal{N}(0,1)$ Gaussians.
		\item  The $n\ell\times n\ell$ covariance matrix $P = \Expect{}{\mathcal{G}\mathcal{G}^*}$ of the Gaussian ensemble satisfies  $P-\eta I\succeq 0$, for some $\eta>0$. %For every collection of complex numbers
  %       $\{\alpha_{i, j}\} \subset \mathbb{C}$:
		% \[
		% \mathrm{Var} \left[ \sum_{i,j} \alpha_{i,j} \cdot G_{i,j} \right]
		% \ \ge\ \eta \cdot \sum_{i,j} \card{\alpha_{i,j}}^2.
		% \]
	\end{enumerate}
	Then, for all functions $f_1,\ldots,f_{\ell} \in L^2(\mathbb{C}^n,\gamma^n)$ such that $f_{1},\ldots,f_{\ell} : \mathbb{C}^n \to [0,1]$ and
	$\Expect{}{f_{j}(\mathcal{G}_{j})} = \mu_j$,
	we have:
	\[
	\Expect{}{\prod_{j=1}^\ell f_{j}(\mathcal{G}_{j})} \geq
	\left( \prod_{j=1}^\ell \mu_j \right)^{1/\eta}.
	\]
\end{restatable}

\begin{remark}
    The statement of the theorem from~\cite{ChenDafnisPaouris2015HolderGaussian} (and~\cite{Mos}) is for real Gaussian random variables (See Theorem~\ref{thm:multi_reverse_hyp_real}). In Appendix~\ref {sec:GRH_complex}, we show how to derive the above theorem from their theorem statement.
\end{remark}

% \begin{remark}
% 	The property $2$ in the above theorem is the same as saying $T-\eta I\succeq 0$, where $T$ is the covariance matrix of the Gaussian ensemble. 
% \end{remark}

\subsection{Proof of the Main Theorem}
Let $A\subseteq \Sigma^n$ such that $\frac{|A|}{|\Sigma|^n}\geq \delta$ and let $f=1_A$ be the indicator function of the set $A$. The goal is to estimate the following quantity
\begin{align*}
    \Delta_1(f) := \Expect{(\V x, \V y, \V z)\sim \mu^{\otimes n}}{f(\V x) f(\V y) f(\V z)}
\end{align*}
where the minimum non-zero probability of the distribution $\mu$ is $\alpha>0$. 
Note that $\mu|_{x} = \mu|_{y} = \mu|_{z} = \nu$ are the marginals of $\mu$. Hence, $\mu\in M_{\nu,\alpha}$ as per Definition~\ref{def:set_of_dists}. Therefore, the conditions in the arithmetic regularity lemma, Lemma~\ref{lem:arithmetic_reg_high_rank}, are satisfied for this distribution and we will use it henceforth.

\subsubsection{Making the Function Resilient}

As the first step of this process, we show that it suffices to bound $\Delta_1$ for resilient functions, defined as follows:

\begin{definition}
   For $B\in \mathbb{N}$ and $0< \gamma\leq 1$, we say a function $g\colon \Sigma^n\to [0,1]$ is $(B,\gamma)$-resilient if
    for every $I\subseteq [n]$ of size at most $B$ and 
    every  $v\in \Sigma^{I}$ we have that $\nu^{\bar{I}}(g_{I\rightarrow v})\geq \gamma\cdot \nu^{\otimes n}(g)$.
\end{definition}

We have the following claim.
\begin{claim}\label{claim:making_f_resilient}
    Suppose $\Delta_1(g) \geq \delta'(B, \eta)$ for every $(B,1/2)$-resilient function $g: \Sigma^n \rightarrow \{0,1\}$ with $\eta =\nu^{\otimes n}(g)$. Then, for every $f: \Sigma^n \rightarrow \{0,1\}$ with $\nu^{\otimes n}(f) = \delta$, we have
    $$\Delta_1(f)\geq  \alpha^{\frac{2B}{\delta\alpha^B}}\delta'(B, \delta).$$ 
\end{claim}
\begin{proof}
If $f$ is $(B, 1/2)$-resilient, then $\Delta_1(f)\geq \delta'(B, \delta)$ and we are done. Otherwise, we have that $f$ is not $(B, 1/2)$-resilient. This means there exists $I_1\subseteq [n]$ and $w\in \Sigma^{I_1}$ such that $|I_1|\leq B$ and $\nu^{\bar{I_1}}(f_{I_1\rightarrow w}) \leq \delta/2$. As $\nu^{\otimes n}(f) = \delta$, there exists $v_1$ such that $\nu^{\bar{I_1}}(f_{I_1\rightarrow v_1}) \geq \delta + \frac{\delta}{2}\alpha^{B}$. Now, if $f_{I_1\rightarrow v_1}$ is $(B, 1/2)$-resilient then we get that 
\begin{align*}
    \Delta_1(f) &= \Expect{(\V x, \V y, \V z)\sim \mu^{\otimes n}}{f(\V x) f(\V y) f(\V z)}\\
    & \geq \alpha^{|I_1|}\cdot  \Expect{(\V x, \V y, \V z)\sim \mu^{\overline{I_1}}}{ f_{I_1\rightarrow v_1}(\V x) f_{I_1\rightarrow v_1}(\V y) f_{I_1\rightarrow v_1}(\V z)}\\
    &\geq \alpha^{B}\cdot \delta'(B, \delta),
\end{align*}
and hence the conclusion holds. If $f_{I_1\rightarrow v_1}$ is not $(B, 1/2)$-resilient, then we can find $I_2$ and $v_2$ and continue the process. Each time, the measure of the restricted function increases by an additive factor of at least $\frac{\delta}{2}\alpha^B$. Thus, the process stops after at most $t\leq \frac{2}{\delta\alpha^B}$ iterations, as otherwise the measure of the restricted function becomes greater than $1$, which cannot happen as $f(\V x)\in \{0,1\}$ for all $\V x$. Let $I = I_1\cup I_2\cup \ldots \cup I_t$ and $v=v_1\circ v_2\circ \ldots \circ v_t$ be the union of the subsets and the concatenation of the strings found during the process. We have $|I|\leq t\cdot B \leq  \frac{2B}{\delta\alpha^B}$. The function $f_{I\rightarrow v}$ is now $(B, 1/2)$-resilient with measure at least $\frac{\delta}{2}$, and hence
\begin{align*}
    \Delta_1(f) &= \Expect{(\V x, \V y, \V z)\sim \mu^{\otimes n}}{f(\V x) f(\V x+\V a) f(\V x+2\V a)}\\
    & \geq \alpha^{|I|}\cdot \E_{(\V x, \V y, \V z)\sim \mu^{\overline{I}}}\left[ f_{I\rightarrow v}(\V x) f_{I\rightarrow v}(\V y) f_{I\rightarrow v}(\V z)\right]\\
    &\geq \alpha^{\frac{2B}{\delta\alpha^B}}\cdot \delta'(B, \delta).
    \qedhere
\end{align*}
\end{proof}
Thus, it suffices to establish a lower bound $\Delta_1(f)$ under the assumption that $f$ is $(B, 1/2)$-resilient, and focus on this case henceforth.

\subsubsection{Applying the Arithmetic Regularity Lemma}
\label{sec:apply_arl}
% To estimate the quantity
% %
% $$	\Delta_1(f) =\Expect{(\V x, \V y, \V z)\sim \mu^{\otimes n}}{ f(\V x) f(\V y) f(\V z)},$$
% %
% we apply the arithmetic regularity lemma (Lemma~\ref{lem:arithmetic_reg_high_rank}) to the function $f$ for sufficiently small $\xi$ and $\rho_i$'s compared to the density of the function $f$. 
We now apply Lemma~\ref{lem:arithmetic_reg_high_rank}. For better readability, it will be convenient to denote various functions of the lemma using simpler notation. We will use the following notation.

\newcommand{\noisyfcoarse}{\mathsf{f}}
\newcommand{\noisyfrefined}{\mathsf{f}'}
\newcommand{\fcoarsedecomposed}{{\mathsf{F}}}
\newcommand{\fcoarsedecoupled}{{\mathsf{F}_d}{}}
\newcommand{\truncfcoarsedecomposed}{\tilde{\mathsf{F}}}
\newcommand{\truncfcoarsedecoupled}{{\tilde{\mathsf{F}}_d}{}}
\newcommand{\truncmultilinfcoarsedecoupled}{{\widetilde{\mathsf{M}}_d}{}}
\newcommand{\multilinfcoarsedecoupled}{{{\mathsf{M}}_d}{}}

\begin{equation}
\label{eq:regularity_functions}
\begin{array}{ll}
   \noisyfcoarse = \mathrm{T}_{\mathcal{P},1-\rho_1}f,  &\quad\quad \noisyfrefined = \mathrm{T}_{\mathcal{P}',1-\rho_5}f\\ \\
   \fcoarsedecomposed(\V x)  = \sum\limits_{P\in \spn(\mathcal{P})}P(x)\cdot L_{P}(\V x), &\quad\quad\truncfcoarsedecomposed(\V x) = \mathrm{trunc}_{[0,1]}(\fcoarsedecomposed(\V x))\\ \\
   \fcoarsedecoupled(\V x, \V y)  = \sum\limits_{P\in \spn(\mathcal{P})}P(\V x)\cdot L_{P}(\V y), &\quad\quad\truncfcoarsedecoupled(\V x', \V x'') := \mathrm{trunc}_{[0,1]}(\fcoarsedecoupled(\V x', \V x'')).
\end{array}
\end{equation}
Here, $\noisyfcoarse$ 
and $\noisyfrefined$ are from Lemma~\ref{lem:arithmetic_reg_high_rank}, and $\fcoarsedecomposed$ is from Lemma~\ref{lem:approx_formula} applied to the function $\noisyfcoarse$ 
%with the same $\xi$ from the arithmetic regularity lemma, Lemma~\ref{lem:arithmetic_reg_high_rank}. In other words, 
with
$\norm{\fcoarsedecomposed - \noisyfcoarse}_2 \leq \xi$. The function $\fcoarsedecoupled$ is the same as the function $\fcoarsedecomposed$ except that it uses two separate inputs for the product functions and the low-degree functions in the decomposition of $\fcoarsedecomposed$. We will also use the following notation throughout.\\

$m := |\Sigma|$

$M := \max_{P\in \mathcal{P}} \{{\sf ord}(P) \}$ where $M= O_m(1)$.

$r:= |\mathcal{P}|  \leq \frac{1}{\rho_0}$

$\|L_P\|_2\leq C, \quad {\sf deg}(L_P) \leq D$ where $C$ and $D$ depend on $\xi$ and $\rho_1$.\\
\\
We set the parameters as follows (the parameters $\tau, \eta$ and $\kappa$ would appear later in the proof).  Let $w\colon (0,1)\to(0,1)$ be a sufficiently rapid decay function and set\footnote{Note that this setting corresponds to applying Lemma~\ref{lem:arithmetic_reg_high_rank} with the decay function $w(.)^2$ as $\rho_3\ll_{w^2} \rho_2$ and $\rho_2\ll_{w^2}\rho_1$.}

\begin{equation}
    \label{eq:parameters}
    \rho_5\ll_w \rho_4\ll_w \rho_3\ll_w \tau\ll_w \rho_2\ll_w \kappa\ll_w \rho_1 \ll_w \rho_0 \leq \eta \leq \xi  \ll \delta \ll m^{-1},\alpha.
\end{equation}
%We use the notation $O_{\rho_i}(1)$ to denote a quantity that is upper bounded by some fixed function of $\rho_i$ (independent of $\rho_j$ for $j>i$). The above setting would make sure that $\rho_{i+1} \ll O_{\rho_i}(1)$. Similarly, we use $\Omega_{\rho_i}(1)$ to denote a quantity that is lower bounded by some fixed function of $\rho_i$. With these notations, for instance, we can simply write $\|L_P\|_2 = O_{\rho_1}(1)$ and ${\sf deg}(L_P) = O_{\rho_1}(1)$ for the low-degree functions $L_P$ from the decomposed function $\fcoarsedecomposed$ above. \\
%\vspace{20pt}

%\abnote{Used the following inequalities}

%$e^{O_m(1/\rho_0)}\cdot e^{-\frac{1}{8\rho_2}} \leq 0.01$ \hyperlink{setting1}{here}

%$\tau\ll C^{-1}, D^{-1}$

We now replace $f$ with $\noisyfrefined$ in the expression $\Delta_1(f)$ by paying only a factor of $O(\rho_3)$ as follows.
\begin{align}
    \Delta_1(f) &:=\Expect{(\V x, \V y, \V z)\sim \mu^{\otimes n}}{ \noisyfrefined(\V x) f(\V y) f(\V z)} + \Expect{(\V x, \V y, \V z)\sim \mu^{\otimes n}}{ (f- \noisyfrefined)(\V x) f(\V y) f(\V z)} \nonumber\\
    & \geq  \Expect{(\V x, \V y, \V z)\sim \mu^{\otimes n}}{\noisyfrefined(\V x) f(\V y) f(\V z)} - \norm{f-\noisyfrefined}_{\mu}\nonumber\\
        & \geq \Expect{(\V x, \V y, \V z)\sim \mu^{\otimes n}}{\noisyfrefined(\V x) f(\V y) f(\V z)}- \rho_3.\nonumber
 \end{align}
Repeating a similar process for the second and third occurrences of $f$, we get
\begin{align}
    \Delta_1(f) & \geq\Expect{(\V x, \V y, \V z)\sim \mu^{\otimes n}}{ \noisyfrefined(\V x) \noisyfrefined(\V y) \noisyfrefined(\V z)}- 3\rho_3. \label{eq:delta1_to_delta2}
\end{align}

The following claim shows that the resilience of $f$ implies resilience of $\noisyfcoarse$.
\begin{claim}
\label{claim:f_to_noisyf_resilience}
The function $\noisyfcoarse$ is $(B, 1/4)$-resilient.
\end{claim}
\begin{proof}
To show that $\noisyfcoarse$ is $(B, 1/4)$-resilient, fix a subset $J \subseteq [n]$ of size at most $B$ and a string $v\in \Sigma^{J}$. Let $n' = n-|J|$.
We start by looking at the function $\noisyfcoarse_{J\rightarrow v}$.
\[
\noisyfcoarse_{J\rightarrow v} (x) = \Expect{y\sim {\rm T}_{\mathcal{P}, 1-\rho_1} ((v, x))}{f(y)}.
\]
Our goal is to show that the following quantity is at least $ \frac{\delta}{4}$.
\[
\nu^{\otimes n'}(\noisyfcoarse_{J\rightarrow v} ) = \Expect{x\sim \nu^{\otimes n'}}{\Expect{y\sim {\rm T}_{\mathcal{P}, 1-\rho_1} ((v, x))}{f(y)}}.
\]
Recall the operator $\mathrm{T}_{\nu, \mathcal{P}, I} \colon L_2(\Sigma^n,\nu^{\otimes n})\to L_2(\Sigma^n,\nu^{\otimes n})$ from Definition~\ref{def:noise_op_P}. 
%Now, it is easy to see that the operator ${\rm T}_{\mathcal{P}, 1-\rho_1}$ as the average of the operators ${\rm T}_{\mathcal{P},I}$ where $I\subseteq [n]$ with each $i\in I$ is included with probability $\rho_1$. Let $I\subseteq_{\rho_1}[n]$ denote the distribution on $I$ above. 
We have
\[
\nu^{\otimes n'}(\noisyfcoarse_{J\rightarrow v} ) = \Expect{x\sim \nu^{\otimes n'}}{\Expect{I\subseteq_{\rho_1}[n]}{\Expect{y\sim {\rm T}_{\mathcal{P},I} ((v, x))}{f(y)}}} .
\]
We will show that with probability at least $0.99$ over the choice of $I$, the statistical distance between the distribution on $y$ in the above expectation and the distribution on $z\sim \nu^{\otimes n}$ conditioned on $z_{J'} =  v'$ for $J' = J\cap I$ and $v' = v|_{J'}$  is upper bounded by $\frac{\delta}{10}$. Once we show that, as $f$ takes value in $\{0, 1\}$, we will get that
\[
\nu^{\otimes n'}(\noisyfcoarse_{J\rightarrow v} )  \geq 0.99\Expect{I\subseteq_{1-\rho_1} [n]}{\nu^{\otimes n-|J'|}(f_{J'\rightarrow v'})} - \frac{\delta}{10}\geq 0.99\cdot \frac{\delta}{2} - \frac{\delta}{10} \geq \frac{\delta}{4},
\]
where the second inequality used the fact that $f$ is $(B, 1/2)$-resilient and $\nu^{\otimes n}(f) \geq \delta$.

Towards showing the closeness in the statistical distance, define the following two distributions:
\begin{itemize}
    \item $\mathcal{D}_{I,v}$: Sample $x\sim \nu^{\otimes n'}$, and output $y\sim {\rm T}_{\mathcal{P},I} ((v, x))$. 
    \item $\mathcal{D}'_{J', v'}$: Sample $z\sim \nu^{\otimes n}$ conditioned on $z|_{J'} = v'$. 
\end{itemize}
Let $\mathcal{P} = \{P_1, P_2, \ldots, P_r\}$. 
%For a given $I\subseteq [n]$, let $\mathcal{P}|_{{I}}$ be the following collection of product functions $\{ P_i|_{{I}} \mid P_i\in \mathcal{P}\}$, where $P_i|_{{I}}(x) = \prod_{j\in {I}} P_{i,j}(x)$ if $P_i(x) = \prod_{j\in [n]} P_{i,j}(x)$. 
Thus, the proof will be complete if we show the following two statements. 

\begin{enumerate}
    \item[(a)] With probability at least $0.99$ over $I\subseteq_{\rho_1} [n]$, ${\sf rk}(\mathcal{P}|_{I})\geq \frac{1}{4\rho_2}$.
    \item[(b)] If ${\sf rk}(\mathcal{P}|_{{I}})\geq \frac{\rho_1}{2\rho_2}$, then the distribution $\mathcal{D}_{I, v}$ is $\frac{\delta}{10}$-close to the distribution $\mathcal{D}'_{J', v'}$.
\end{enumerate}

\paragraph{Proof of $(a)$.} Fix a vector $\vec{\alpha}\in (\mathbb{N} \cup \{0\})^{|\mathcal{P}|}$ where $0\leq \alpha_P <{\sf ord}(P)$ for all $P\in \mathcal{P}$ that are not all zero. Let $S_{\vec{\alpha}}(x) = \prod_{P\in \mathcal{P}} P(x)^{\alpha_P}$. As ${\sf rk}(\mathcal{P}) \geq \frac{1}{\rho_2}$, we have $\Delta_{{\sf symbolic}}(S_{\vec{\alpha}},1)\geq \frac{1}{\rho_2}$. Let $A({\vec{\alpha}}) \subseteq [n]$ be the set of coordinates where $(S_{\vec{\alpha}})_i \not\equiv 1$. Thus, we have $|A({\vec{\alpha}})|\geq \frac{1}{\rho_2}$. For $I\subseteq_{\rho_1} [n]$, we have that $	\Expect{I}{|I\cap A({\vec{\alpha}})|} \geq \frac{\rho_1}{\rho_2}$. By Chernoff bound, the probability that $|I\cap A({\vec{\alpha}})| \leq \frac{\rho_1}{2\rho_2}$ is at most $e^{-\frac{\rho_1}{8\rho_2}}$. Finally, taking the union bound over all the admissible $\vec{\alpha}$ (the number of admissible $\vec{\alpha}$s is at most $e^{O_m(1/\rho_0)}$ as ${\sf ord}(P) = O_m(1)$ for all $P\in \mathcal{P}$ and $|\mathcal{P}|\leq \frac{1}{\rho_0}$), we have that there exists an $\vec{\alpha}$ such that $|I\cap A({\vec{\alpha}})| \leq \frac{\rho_1}{2\rho_2}$ is at most \hypertarget{setting1}{$e^{O_m(1/\rho_0)}\cdot e^{-\frac{\rho_1}{8\rho_2}}\leq 0.01$}. Therefore, with probability at least $0.99$ over $I\subseteq_{\rho_1} [n]$, we have ${\sf rk}(\mathcal{P}|_{I})\geq \frac{\rho_1}{2\rho_2}$ as required.
\paragraph{Proof of $(b)$.}
Consider the following two sets for $v'\in \Sigma^{J'}$ and $w\in \Sigma^n$:
\begin{align*}
    S_{v'} &= \{ z \mid z|_{J'} = v'\},\\
    T_w &= \{z\mid P(x) = P(w) ~~\forall P\in \mathcal{P} \mbox{ and } z_{\overline{I}} = w_{\overline{I}}\}.
\end{align*}
For any fixed $w\in \Sigma^n$, we will now compute $\Pr_{y\sim \mathcal{D}'_{J', v'}}[y=w]$ and $\Pr_{y\sim \mathcal{D}_{I,v}}[y=w]$. The first one is easy:
$$\Pr_{y\sim \mathcal{D'}_{J', v'}}[y=w] = \frac{\nu^{\otimes n}(w)}{\nu^{\otimes n}(S_{v'})}.$$
We now estimate $\Pr_{y\sim \mathcal{D}_{I,v}}[y=w]$. Recall the distribution $\mathcal{D}_{I,v}$: Sample $x\sim \nu^{\otimes n'}$, and output $y\sim {\rm T}_{\mathcal{P},I} ((v, x))$. Thus, the following conditions must be met for $y=w$
\begin{enumerate}
    \item $x_{\overline{I}\cap \overline{J}} = w_{\overline{I}\cap \overline{J}}$, as the noise operator ${\rm T}_{\mathcal{P},I}$ does not change $x|_{\overline{I}\cap \overline{J}}$. 
    \item $P((v,x)) = P(w) ~\forall P\in \mathcal{P}$. Let this event be $E$.
    \item Finally, $y_{I} = w_{I}$.
\end{enumerate}
Thus,
\begin{align*}
    &\Pr_{\substack{x\sim \nu^{\otimes n'},\\ y\sim {\rm T}_{\mathcal{P},I} ((v, x))}}[y=w] \\
    &\quad= \Pr_{x\sim \nu^{\otimes n'}}[x_{\overline{I}\cap \overline{J}} = w_{\overline{I}\cap \overline{J}}]\cdot \Pr_{x\sim \nu^{\otimes n'}}[ E \mid x_{\overline{I}\cap \overline{J}} = w_{\overline{I}\cap \overline{J}}] \cdot  \Pr_{\substack{x\sim \nu^{\otimes n'},\\ y\sim {\rm T}_{\mathcal{P},I} ((v, x))}}[y=w \mid E \wedge x_{\overline{I}\cap \overline{J}} = w_{\overline{I}\cap \overline{J}}].
\end{align*}
We now estimate each of these probabilities. Note that
\begin{align*}
    \Pr_{x\sim \nu^{\otimes n'}}[x_{\overline{I}\cap \overline{J}} = w_{\overline{I}\cap \overline{J}}] = \prod_{\ell\in \overline{I}\cap \overline{J} } \nu(w_\ell) = \prod_{\ell\in \overline{I}\cap \overline{J} } \nu(w_\ell)  \cdot \frac{\prod_{\ell\in J'} \nu(w_\ell) }{\prod_{\ell\in J'} \nu(w_\ell) } = \frac{\nu^{\bar{I}}(w|_{\overline{I}})}{\nu^{\otimes n}(S_{v'})}.
\end{align*}
Also, 
\begin{align*}
    \Pr_{\substack{x\sim \nu^{\otimes n'},\\ y\sim {\rm T}_{\mathcal{P},I} ((v, x))}}[y=w \mid E \wedge x_{\overline{I}\cap \overline{J}} = w_{\overline{I}\cap \overline{J}}] = \frac{\nu^{\otimes n}(w)}{\nu^{\otimes n}(T_w)}.
\end{align*}
Therefore,
\begin{align*}
    \card{\Pr_{\substack{x\sim \nu^{\otimes n'},\\ y\sim {\rm T}_{\mathcal{P},I} ((v, x))}}[y=w]  - \frac{\nu^{\otimes n}(w)}{\nu^{\otimes n}(S_{v'})}} &\leq \card{ \frac{\nu^{\bar{I}}(w|_{\overline{I}})}{\nu^{\otimes n}(S_{v'})}\Pr_{x\sim \nu^{\otimes n'}}[ E \mid x_{\overline{I}\cap \overline{J}} = w_{\overline{I}\cap \overline{J}}] \frac{\nu^{\otimes n}(w)}{\nu^{\otimes n}(T_w)}   -\frac{\nu^{\otimes n}(w)}{\nu^{\otimes n}(S_{v'})}}\\
    &\leq \frac{\nu^{\otimes n}(w)}{\nu^{\otimes n}(S_{v'})}\card{ \frac{\nu^{\bar{I}}(w|_{\overline{I}})\Pr_{x\sim \nu^{\otimes n'}}[ E \mid x_{\overline{I}\cap \overline{J}} = w_{\overline{I}\cap \overline{J}}] - \nu^{\otimes n}(T_w)}{\nu^{\otimes n}(T_w)}}.
\end{align*}
Let $s= \prod_{P\in \mathcal{P}} {\sf ord}(P)^{-1}$. Then, using the fact that  ${\sf rk}(\mathcal{P}|_{{I}})\geq \frac{\rho_1}{2\rho_2}$ and Lemma~\ref{lemma:quasirandom_onefunc}, we have
$$ \card{\Pr_{x\sim \nu^{\otimes n'}}[ E \mid x_{\overline{I}\cap \overline{J}} = w_{\overline{I}\cap \overline{J}}] - s} \leq 2^{-\Omega_{M,r,\alpha}\left(\frac{\rho_1}{\rho_2}\right)},$$
and
$$\card{\nu^{\otimes n}(T_w) - s\cdot \nu^{\bar{I}}(w|_{\overline{I}})} \leq 2^{-\Omega_{M,r,\alpha}\left(\frac{\rho_1}{\rho_2}\right)}\nu^{\bar{I}}(w|_{\overline{I}}).$$
Although the latter statement does not follow directly from Lemma~\ref{lemma:quasirandom_onefunc}, observe that the same proof works if we replace the distribution $\nu^{\otimes n}$ with the conditional distribution $\nu^{\otimes n}|x|_{\overline{I}}= w_{\overline{I}}$ and using the fact that  ${\sf rk}(\mathcal{P}|_{{I}})\geq \frac{\rho_1}{2\rho_2}$. Adding the conditioning back gives the required bound as above.

If we let $\eta = 2^{-\Omega_{M,r,\alpha}\left(\frac{\rho_1}{\rho_2}\right)}$, then using the fact that $\rho_2\ll \rho_1\ll s$, we have
\begin{align*}
    \card{\Pr_{\substack{x\sim \nu^{\otimes n'},\\ y\sim {\rm T}_{\mathcal{P},I} ((v, x))}}[y=w]  - \frac{\nu^{\otimes n}(w)}{\nu^{\otimes n}(S_{v'})}} &\leq \frac{\nu^{\otimes n}(w)}{\nu^{\otimes n}(S_{v'})}\card{ \frac{\nu^{\bar{I}}(w|_{\overline{I}})\Pr_{x\sim \nu^{\otimes n'}}[ E \mid x_{\overline{I}\cap \overline{J}} = w_{\overline{I}\cap \overline{J}}] - \nu^{\otimes n}(T_w)}{\nu^{\otimes n}(T_w)}}\\
    & \leq \frac{\nu^{\otimes n}(w)}{\nu^{\otimes n}(S_{v'})}\card{ \frac{\nu^{\bar{I}}(w|_{\overline{I}})\cdot (s+\eta) - (s\cdot \nu^{\bar{I}}(w|_{\overline{I}})-\eta\cdot \nu^{\bar{I}}(w|_{\overline{I}}))}{(s/2)\cdot \nu^{\bar{I}}(w|_{\overline{I}})}}\\
    &\leq \frac{4\eta}{s} \frac{\nu^{\otimes n}(w)}{\nu^{\otimes n}(S_{v'})}.
\end{align*}
Summing over all of $w$, we get that the the distribution $\mathcal{D}_{I,v}$ is $\frac{4\eta}{s}$-close to the distribution $\mathcal{D}'_{J', v'}$. As $\rho_2\ll \rho_1\ll s$, we have $\frac{4\eta}{s}\leq \frac{\delta}{10}$ as required.
\end{proof}

Define the following quantity, 
$$\Delta_2(\noisyfrefined) :=\Expect{(\V x, \V y, \V z)\sim \mu^{\otimes n}}{\noisyfrefined(\V x) \noisyfrefined(\V y) \noisyfrefined(\V z)}.$$
 Using (\ref{eq:delta1_to_delta2}), as we have

\begin{equation}\label{eq:Delta_1_to_2}
    \Delta_1(f)\geq \Delta_2(\noisyfrefined) -3\rho_3,
\end{equation}
and our next goal is to establish a lower bound on  $\Delta_2(\noisyfrefined)$.

\subsubsection{Removing influential variables}
\label{sec:remove_inf}

In the next step, we apply a random restriction to a small subset of coordinates. The primary goal of this restriction is to ensure that in the approximate decomposition of $\noisyfcoarse$, given by $\fcoarsedecomposed$,

  \[
	\fcoarsedecomposed(\V x) = \sum\limits_{P\in\spn(\mathcal{P})} P(\V x) \cdot L_P(\V x),
	\]
the restricted low-degree functions $L_P$ for every $P\in \spn(\mathcal{P})$ have small influences.

We formally define the influence of a variable in the decomposed function. 

\begin{definition}
    We say that a function $\fcoarsedecomposed: (\Sigma^n, \nu^{\otimes n})\rightarrow \mathbb{C}$ given by 
     \[
	\fcoarsedecomposed(\V x) = \sum\limits_{P\in\spn(\mathcal{P})} P(\V x) \cdot L_P(\V x).
	\]
    has $\tau$-small shifted low-degree influences if for every $i\in [n]$  and every $P\in\spn(\mathcal{P})$ it holds that $I_i[L_P, \nu^{\otimes n}]\leq \tau$.
\end{definition}

The following claim shows the existence of a small set such that most restrictions to the variables in the set entail low influences for the restricted function $\fcoarsedecomposed$.   The proof follows the strategy of the regularity lemma from~\cite{Jones}. We prove the claim in the appendix, focusing on the bounds of the parameters.

\begin{restatable}{claim}{jones}\label{claim:influence_removal}
    Consider the function  $\fcoarsedecomposed$ from above. Then for any $\tau>0$, there exists $J\subseteq [n]$ of size $O_{|\mathcal{P}|, |\Sigma|, D, C,\eta, \tau}(1)$ such that with probability at least $1-\eta$ over  $\V w\sim \nu^{\otimes |J|}$, we have
    $$ \max_{P\in \spn(\mathcal{P})} \{ \max_i I_i[(L_P)_{J\rightarrow \V w}]\}\leq \tau.$$
\end{restatable}

We take the set $J$ from the previous claim. We next restrict the function $\noisyfrefined$ to inputs $\V x$ such that $\V x|_{J} = \V w$ for some $\V w\in \Sigma^J$. Denote the restricted function by $\noisyfrefined_{J\rightarrow \V w}$. We consider the following four events regarding $w$:
\begin{enumerate}
    \item $E_1$: all the influences of $L_P|_{J\rightarrow \V w}$ are at most $\tau$ for every $P\in \spn(\mathcal{P})$.
    \item $E_2$: the $\ell_2$ norms of the restricted low-degree functions $L_P|_{J\rightarrow \V w}$ are bounded by $C'$ where $C' = 100C2^{O_m(r)}$.
    \item $E_3$: $\|\noisyfrefined|_{J\rightarrow \V w} - \noisyfcoarse|_{J\rightarrow \V w}\|_2 \leq \sqrt{\xi}$.
    \item $E_4$:  $\|\fcoarsedecomposed_{J\rightarrow \V w}  - \noisyfcoarse|_{J\rightarrow \V w}\|_2 \leq \sqrt{\xi}$.
\end{enumerate}
The following claim shows that with high probability over $\V w\sim \Sigma^J$, all the events $E_1, E_2, E_3$, and $E_4$ hold.
\begin{claim}\label{claim:goodevents}
    $\Pr_{\V w\sim \Sigma^J}\left[E_1 \wedge E_2 \wedge E_3\wedge E_4\right]\geq 0.98$.
\end{claim}
\begin{proof}
    Consider a function $L_P$ where $P\in \spn(\mathcal{P})$. As $\|L_P\|\leq C$, by Markov's inequality the probability over the choice of $w$ that $\|L_P|_{J\rightarrow \V w}\|_2\geq C'$ is at most $\frac{C^2}{C'^2}$. As the number of the functions $L_P$ in the expansion of $\fcoarsedecomposed$ is upper bounded by $|\spn(\mathcal{P})|\leq 2^{O_m(r
    )}$, 
    by union bound, the event $E_2$ holds with probability at least $1-\frac{3C^22^{O_m(r)}}{C'^2}\geq 0.01$. Using Claim~\ref{claim:influence_removal}, $E_1$ holds with probability at least $1-
    \eta$.
    As $\|\noisyfrefined-\noisyfcoarse\|_2^2\leq \xi^2$, Markov's inequality gives that the probability that $\|\noisyfrefined|_{J\rightarrow \V w} - \noisyfcoarse|_{J\rightarrow \V w}\|_2^2\geq \xi$ is at most $\xi$. Similarly, $\Pr[\neg E_4]\leq \xi$ as  $\|\fcoarsedecomposed  - \noisyfcoarse\|_2 \leq \xi$. Therefore, by the union bound $E_1, E_2, E_3$ and $E_4$ hold together with probability at least $1-\eta-2\xi -0.01\geq 0.98$. 
    \end{proof}

    Fix $J$ from above and let $E_{\V w}$ be the indicator that the events $E_1 , E_2, E_3 $ and $E_4$ happen for the given $\V w$. Recall, $\Delta_2(\noisyfrefined)$,
    \begin{align*}
    \Delta_2(\noisyfrefined) &=\Expect{(\V x, \V y, \V z)\sim \mu^{\otimes n}}{\noisyfrefined(\V x) \noisyfrefined(\V y) \noisyfrefined(\V z)}.
    \end{align*}
We can rewrite the above expectations as follows.
    \begin{align*}
     \Delta_2(\noisyfrefined) &=  \Expect{(\V w, \V w', \V w'') \sim \mu^{J}}{\Expect{(\V x, \V y, \V z)\sim \mu^{\bar{J}}}{\noisyfrefined|_{J\rightarrow \V w}(\V x) \noisyfrefined|_{J\rightarrow \V w'}(\V y) \noisyfrefined|_{J\rightarrow \V w''}(\V z)}}\\
     & \geq 0.98\cdot \alpha^{|J|} \cExpect{w\sim\nu^{J}}{E_w}{\Expect{(\V x, \V y, \V z)\sim \mu^{\bar{J}}}{\noisyfrefined|_{J\rightarrow \V w}(\V x) \noisyfrefined|_{J\rightarrow \V w}(\V y) \noisyfrefined|_{J\rightarrow \V w}(\V z)}}\\
       & \geq \alpha^{|J|+1} \cExpect{w\sim\nu^{J}}{E_w}{\underbrace{\Expect{(\V x, \V y, \V z)\sim \mu^{\bar{J}}}{\noisyfrefined|_{J\rightarrow \V w}(\V x) \noisyfrefined|_{J\rightarrow \V w}(\V y) \noisyfrefined|_{J\rightarrow \V w}(\V z)}}_{\Delta_2(\noisyfrefined, E_{\V w}) }}\\
    \end{align*}
   where in the second step, we use Claim~\ref{claim:goodevents} and the fact that $\V w = \V w' = \V w''$ happens with probability at least $\alpha^{|J|}$.  Next, we estimate $\Delta_2(\noisyfrefined, E_{\V w})$ defined above.
   
\subsubsection{Moving to a Decoupled Function}
\label{sec:move_to_decoulped}
We will fix the set $J$ and $\V w$ for the remainder of this section, and assume that $E_w$ holds.  Let us denote $n' = n-|J|$. For succinctness, we denote $\noisyfrefined_{J\rightarrow\V w}, \noisyfcoarse_{J\rightarrow\V w}$ and $\fcoarsedecomposed|_{J\rightarrow \V w}$ simply by $\noisyfrefined_{\V w}, \noisyfcoarse_{\V w}$ and $\fcoarsedecomposed_{\V w}$, respectively. The next step is to decouple the input from $\fcoarsedecomposed_{\V w}$ to the product functions and the low-degree functions. Towards this, for a given function $\fcoarsedecomposed_{\V w} : (\Sigma^{n'}, \nu^{\otimes n'}) \rightarrow \mathbb{C}$, define a function $\fcoarsedecoupled_{,\V w}: (\Sigma^{n'}\times \Sigma^{n'},  \nu^{\otimes n'}\times  \nu^{\otimes n'}) \rightarrow \mathbb{C}$
$$\fcoarsedecoupled_{, \V w}(\V x', \V x'') = \sum\limits_{P\in\spn(\mathcal{P})} P(\V w, \V x') \cdot L_P(\V w, \V x'').$$

We also truncate the functions so that they remain bounded. Towards this, consider the following functions,
    $$\truncfcoarsedecoupled_{,\V w}(\V x', \V x'') := \mathrm{trunc}_{[0,1]}(\fcoarsedecoupled_{, \V w}(\V x', \V x'')),$$
and
$$\truncfcoarsedecomposed_{\V w}(\V x) = \mathrm{trunc}_{[0,1]}(\fcoarsedecomposed_{\V w}(\V x)).$$

The following lemma from~\cite{BKM5} allows one to replace the function $\fcoarsedecomposed_{\V w} $ with $\fcoarsedecoupled_{, \V w}$ by changing the distribution $X\sim \nu^{\otimes n'}$ to a coupled distribution $(X', X'') \sim \nu^{\otimes n'}\times  \nu^{\otimes n'}$ as long as $\mathcal{P}$ has high rank. The decoupled distribution $\mathcal{D}$ is as follows (recall from (\ref{eq:parameters}) that we choose $\kappa$ such that $\rho_2\ll \kappa\ll \rho_1$):
   \begin{enumerate}
            \item Sample $(X,Y,Z)\sim \mu^{\otimes n'}$.
            \item Sample 
            $(X',Y',Z')\sim \mathrm{T}_{\mathcal{\mathcal{P}},0} (X,Y,Z)$.
            \item Sample $(X'',Y'',Z'')\sim \mathrm{T}_{1-\kappa}(X,Y,Z)$. By that, we mean that for each coordinate $i\in \bar{J}$ we have
$(X''_i,Y''_i,Z''_i) = (X_1, Y_i, Z_i)$ with probability $1-\kappa$ and we take $(X''_i,Y''_i,Z''_i)\sim \mu$ with probability $\kappa$ independently.
\end{enumerate}

We have the following lemma from~\cite{BKM5} adapted to our setting above.
\begin{lemma}(\cite[Lemma 6.7]{BKM5})\label{lem:decoupled_move}
		Let $D>0$ and $\mathcal{P} = \{P_1,\ldots,P_r\colon \Sigma^{n'}\to\mathbb{C}\}$ be a collection of
		product functions from the decomposed function $\fcoarsedecomposed_{\V w}$ such that ${\sf rk}(\mathcal{P})\geq D$. Then for every $\kappa>0$, the coupling $\mathcal{D}$ is a coupling between $(X',X'')$ and $X$ such that
		\[
		\Expect{(X,X',X'')\sim \mathcal{D}}{\card{\fcoarsedecoupled_{,\V w}(X',X'') - \fcoarsedecomposed_{\V w}(X)}^2}
		\leq O_{\rho_0, \rho_1, |\Sigma|}(\kappa) +  2^{-\Omega_{\rho_1, \rho_1,|\Sigma|,\kappa}(D)}.
		\]
	\end{lemma}

Towards estimating $\Delta_2(\noisyfrefined, E_{\V w})$, we change the function $\noisyfrefined_{\V w}$ to $\fcoarsedecoupled_{,\V w}$. However, for what comes next, after fixing the input to the product functions in $\fcoarsedecoupled_{,\V w}$, we require that the expectation involves the same function on the second input and that the function has non-negligible measure. Therefore, we do the transformation more carefully by conditioning on certain events defined below.

For a fixed $\vec{b}\in\prod\limits_{P\in\mathcal{P}}{\sf Image}(P)$, define the following sets

$$S_{\V w, \vec{b}} = \{ x\in \Sigma^{n'} \mid P(\V w,\V x) = b_P \mbox{ for every } P\in \mathcal{P}\}.$$
Also, for a given $\V x\in \Sigma^{\bar{J}}$, define the vector $\vec{b}(\V x)$ where $b_P(\V x) = P(\V w, \V x)$ for all $P\in \mathcal{P}$.\\

Abusing the notation, for a vector $\vec{b}\in\prod\limits_{P\in\mathcal{P}}{\sf Image}(P)$, we use the following function in the subsequent analysis.

$$\fcoarsedecoupled_{, \V w}(\vec{b}, \V x'') := \sum\limits_{P\in\spn(\mathcal{P})} b_P \cdot L_P(\V w, \V x'').$$
where for $P'\in\spn(\mathcal{P})$ with $P'(x) = \prod\limits_{P\in\mathcal{P}}P(x)^{\alpha_{P}}$, we have $b_{P'} = \prod\limits_{P\in\mathcal{P}}b_P^{\alpha_{P}}$. Also, let $\truncfcoarsedecoupled_{,\V w}(\vec{b},\V x'') = \mathrm{trunc}_{[0,1]}(\fcoarsedecoupled_{,\V w}(\vec{b},\V x''))$.

Suppose we sample $(X, Y, Z), (X', Y', Z')$ and $(X'', Y'', Z'')$ from the decoupled distribution $\mathcal{D}$ defined above. Consider the following event $E' = E'_1\wedge E'_2$ on the string $X$ where events $E'_i$ are defined as follows for $\vec{b} = \vec{b}(X)$.

\begin{enumerate}
    \item\label{event:E'1} $E'_1$: $\Expect{\V x\sim \nu^{\otimes n'}}{\noisyfrefined(\V w, \V x) \mid \V x\in S_{\V w, \vec{b}}}\geq \frac{\delta}{32}$.
    \item\label{event:E'2} $E'_2$: $\Expect{(\V x, \V x', \V x'')\sim \mathcal{D}}{|\truncfcoarsedecoupled_{,\V w}(\vec{b},\V x'')  - \noisyfrefined_{\V w}(\V x)|^2 \mid \V x\in S_{\V w, \vec{b}}} \leq \xi^{1/4}$.
\end{enumerate}

The event $E'_1$ says that the measure of $\noisyfrefined_{\V w}$ remains large when restricted to inputs $x$ such that $x\in S_{\V w, \vec{b}}$. The event $E'_2$ shows the $\ell_2$ closeness of the decoupled function $\truncfcoarsedecoupled_{,\V w}(\vec{b},\V x'')$ and $\noisyfrefined_{\V w}(\V x)$, again conditioned on $x\in S_{\V w, \vec{b}}$.

The following claim shows that $E'$ occurs with non-negligible probability.
\begin{claim}
    If we sample a $(X, Y, Z), (X', Y', Z')$ and $(X'', Y'', Z'')$ from the decoupled distribution $\mathcal{D}$, then the probability that $E'$ happens is at least 
    $\frac{\delta}{32} - \xi^{1/4}\geq \frac{\delta}{64}$.
\end{claim}
\begin{proof}
%We first show that if we sample $X \sim \nu^{\otimes n'}$, then the events $E'_1, E'_1$ and $E'_2$ happen simultaneously for $\vec{b}(X)$ with probability at least $\left(\frac{\delta}{100} - \xi^{1/4} - O_{\rho_0, m}(\rho_2'')\right)$. Conditioned on these events, by Lemma~\ref{lemma:quasirandom_prop_2}, event $E'_1$ occurs with probability at least $\prod_{P\in \mathcal{P}}\frac{1}{{\sf ord}(P)}$.

As $\|f-\noisyfcoarse\|_\nu\leq \xi$, we have $\Expect{\V x\sim \nu^{\otimes n}}{\noisyfcoarse(\V x)} \geq \frac{\delta}{2}$. Since $\noisyfcoarse$ is $(B, 1/4)$-resilient (Claim~\ref{claim:f_to_noisyf_resilience}), we have $\Expect{\V x\sim \nu^{\otimes n'}}{\noisyfcoarse(\V w, \V x)}\geq \frac{\delta}{8}$. Furthermore, since $\|\noisyfrefined_{\V w} - \noisyfcoarse_{\V w}\|_2\leq \sqrt{\xi}$ (because of ${\V w}$ satisfying event $E_3$), we get $\Expect{\V x\sim \nu^{\otimes n'}}{\noisyfrefined(\V w, \V x)}\geq \frac{\delta}{8}-\xi \geq \frac{\delta}{16}$. Note that $\noisyfrefined$ is a $1$-bounded function. Therefore, if we sample $(X, Y, Z)\sim \mu^{\otimes n'}$, and letting $\vec{b} = \vec{b}(\V X)$, then with probability at least $\frac{\delta}{32}$, we have $\Expect{\V x\sim \nu^{\otimes n'}}{\noisyfrefined(\V w, \V x) \mid \V x\in S_{\V w, \vec{b}}}\geq \frac{\delta}{32}$. 

In the rest of the proof, we estimate $\Pr[E'_2]$. Before doing that, let us first upper bound the quantity $\Expect{(X,X',X'')\sim \mathcal{D}}{\card{\truncfcoarsedecoupled_{,\V w}(X',X'') - \noisyfrefined_{\V w}(X)}^2}$. We have,

\begin{align*}
    &\Expect{(X,X',X'')\sim \mathcal{D}}{\card{\truncfcoarsedecoupled_{,\V w}(X',X'') - \noisyfrefined_{\V w}(X)}^2}\\
    &\quad = \Expect{(X,X',X'')\sim \mathcal{D}}{\card{\truncfcoarsedecoupled_{,\V w}(X',X'') - \noisyfcoarse_{\V w}(X) + \noisyfcoarse_{\V w}(X) - \noisyfrefined_{\V w}(X)}^2}\\
    &\quad \lesssim \Expect{(X,X',X'')\sim \mathcal{D}}{\card{\truncfcoarsedecoupled_{,\V w}(X',X'') - \noisyfcoarse_{\V w}(X)}^2} + \Expect{(X,X',X'')\sim \mathcal{D}}{\card{ \noisyfcoarse_{\V w}(X) - \noisyfrefined_{\V w}(X)}^2}.
\end{align*}
Since $\V w$ satisfies the event $E_3$, the second expectation is upper bounded by $\xi$. We now bound the first expectation,
\begin{align*}
&\Expect{(X,X',X'')\sim \mathcal{D}}{\card{\truncfcoarsedecoupled_{,\V w}(X',X'') - \noisyfcoarse_{\V w}(X)}^2} \\
&\quad = \Expect{(X,X',X'')\sim \mathcal{D}}{\card{\truncfcoarsedecoupled_{,\V w}(X',X'') - \fcoarsedecomposed_{\V w}(X) + \fcoarsedecomposed_{\V w}(X)- \noisyfcoarse_{\V w}(X)}^2} \\
&\quad \lesssim \Expect{(X,X',X'')\sim \mathcal{D}}{\card{\truncfcoarsedecoupled_{,\V w}(X',X'') - \fcoarsedecomposed_{\V w}(X)}^2}  + \Expect{(X,X',X'')\sim \mathcal{D}}{\card{\fcoarsedecomposed_{\V w}(X)- \noisyfcoarse_{\V w}(X)}^2}.
\end{align*}
Again, since $\V w$ satisfies the event $E_4$, the second expectation is upper bounded by $\xi$. As for the first expectation,
 \begin{align}
        &\Expect{X, X', X''}{\card{\truncfcoarsedecoupled_{, \V w} (X',X'') - \fcoarsedecomposed_{\V w}(X)}^2} \nonumber\\
        & = \Expect{X, X', X''}{\card{\truncfcoarsedecoupled_{, \V w}(X',X'') - \fcoarsedecoupled_{, \V w}(X',X'')  + \fcoarsedecoupled_{, \V w}(X',X'') - \fcoarsedecomposed_{\V w}(X)}^2}\nonumber\\
        & \lesssim \Expect{X, X', X''}{\card{\truncfcoarsedecoupled_{, \V w}(X',X'') - \fcoarsedecoupled_{, \V w}(X',X'')}^2} + \Expect{X, X', X''}{\card{\fcoarsedecoupled_{, \V w}(X',X'') - \fcoarsedecomposed_{\V w}(X)}^2}.\label{eq:two_exp_lemma_3wise_decouple}
    \end{align}
As for the first term, using the property that $\eta(a) = \card{a-\mathrm{trunc}_{[0,1]}(a)}$ is $2$-Lipshitz from Fact~\ref{fact:trivial_lipshitz_pf},
\begingroup
\allowdisplaybreaks
\begin{align*}
    &\Expect{X, X', X''}{\card{\truncfcoarsedecoupled_{, \V w}(X',X'') - \fcoarsedecoupled_{, \V w}(X',X'')}^2} \\
    & \quad\quad=  \Expect{X, X', X''}{\eta(\fcoarsedecoupled_{, \V w}(X',X''))^2}\\
    & \quad\quad\lesssim  \Expect{X, X', X''}{\eta(\noisyfcoarse_{\V w}(X))^2 + \card{\fcoarsedecoupled_{, \V w}(X',X'') - \noisyfcoarse_{\V w}(X)}^2} \\
    & \quad\quad =  \Expect{X, X', X''}{\card{\fcoarsedecoupled_{, \V w}(X',X'') - \noisyfcoarse_{\V w}(X)}^2} \tag*{(Using $\eta(\noisyfcoarse_{\V w}(X))=0$)}\\
   & \quad\quad\leq \Expect{X, X', X''}{\card{\fcoarsedecoupled_{, \V w}(X',X'') - \fcoarsedecomposed_{\V w}(X) + \fcoarsedecomposed_{\V w}(X) - \noisyfcoarse_{\V w}(X)}^2} \\
      & \quad\quad\lesssim  \Expect{X, X', X''}{\card{\fcoarsedecoupled_{, \V w}(X',X'') - \fcoarsedecomposed_{\V w}(X)}^2} + \Expect{X, X', X''}{\card{\fcoarsedecomposed_{\V w}(X) - \noisyfcoarse_{\V w}(X)}^2} \\
  & \quad\quad\leq \rho_2' + \xi,
\end{align*}
\endgroup
where $\rho'_2\ll \xi$ if $\rho_2 \ll \rho_1, \rho_0$. Here, we used Lemma~\ref{lem:decoupled_move} to bound the first expectation by  $O_{\rho_0, \rho_1,|\Sigma|}(\kappa) +  2^{-\Omega_{\rho_0, \rho_1,|\Sigma|, \kappa}(1/\rho_2)}$ which is at most $\rho'_2$ (by choosing small enough $\kappa$ and $\rho_2 \ll \kappa$ and noting that ${\sf rk}(\mathcal{P}|_{\overline{J}})\geq {\sf rk}(\mathcal{P})/2$), and the second expectation is bounded by $\xi$ because of the event $E_4$.
The second expectation from (\ref{eq:two_exp_lemma_3wise_decouple}) is at most $\rho_2'$ as we already bounded the same expectation above. Therefore, we get,
\begin{equation}\label{eq:close_truncfcoarsedecoupled_noisyfrefined}
\Expect{(X,X',X'')\sim \mathcal{D}}{\card{\truncfcoarsedecoupled_{,\V w}(X',X'') - \noisyfrefined_{\V w}(X)}^2}\lesssim \xi + \rho'_2.
\end{equation}
Now suppose that for with probability at least $\xi^{1/4}$ over $X\sim \nu^{\otimes n'}$ we have 
$$\Expect{(\V x, \V x', \V x'')\sim \mathcal{D}}{|\truncfcoarsedecoupled_{,\V w}(\vec{b},\V x'')  - \noisyfrefined_{\V w}(\V x)|^2 \mid \V x\in S_{\V w, \vec{b}(X)}} \geq \xi^{1/4}.$$
This would imply
$$\Expect{(X,X',X'')\sim \mathcal{D}}{\card{\truncfcoarsedecoupled_{,\V w}(X',X'') - \noisyfrefined_{\V w}(X)}^2}\geq \xi^{1/2} \gg  \xi + \rho'_2,$$
contradicting (\ref{eq:close_truncfcoarsedecoupled_noisyfrefined}). Thus, $\Pr[E'_2] \geq 1-\xi^{1/4}$.

Therefore, by the union bound, we have $\Pr[E'_1\wedge E'_2] \geq \frac{\delta}{32} - \xi^{1/4}$, as required.
\end{proof}

Let $W_{\vec{b}}$ be the event that $(X, Y, Z)\sim \mu^{\otimes n'}$ are such that $X, Y, Z\in S_{\V w, \vec{b}}$. We have the following claim that shows the closeness of the distribution $(X'', Y'', Z'')|W_{\vec{b}}$ to the distribution $\mu^{\otimes n'}$, provided $\vec{b} \in \prod_{P\in \mathcal{P}} {\sf Image}(P)$.
\begin{claim}\label{claim:closeness}
Fix any $\vec{b}\in  \prod_{P\in \mathcal{P}} {\sf Image}(P)$. In the decoupled distribution, the conditional distribution on $(X'', Y'', Z'')|W_{\vec{b}}$ is $\tilde{\rho}_2$-close to the distribution $\mu^{\otimes n'}$ where $\tilde{\rho}_2 =  2^{-\Omega_{\alpha, \kappa, r, M}(1/\rho_2)}$.
\end{claim}
\begin{proof}
    Fix any subset $A\subseteq (\Sigma^{n'})^3$. The measure of $A$ under $\mu^{\otimes n'}$ is simply $\mu^{\otimes n'}(A)$. We now compute the probability that  $(X'', Y'', Z'')\in A$ conditioned on the event $W_{\vec{b}}$. Let $S = S_{\V w, \vec{b}} \times S_{\V w, \vec{b}}\times S_{\V w, \vec{b}}$.
    \begin{align*}
        \Pr_{\substack{(X, Y, Z)\sim \mu^{\otimes n'}\\ (X'', Y'', Z'')\sim T_{1-\kappa}(X, Y, Z)}}\left[(X'', Y'', Z'')\in A \mid (X, Y, Z) \in S\right] = \frac{1}{\mu^{\otimes n'}(S)} \langle T_{1-\kappa}1_S, 1_A \rangle_{L^2(\mu)}.
    \end{align*}
    Subtracting $\mu^{\otimes n'}(A) = \frac{\mu^{\otimes n'}(S)}{\mu^{\otimes n'}(S)} \langle1, 1_A \rangle_{L^2(\mu)}$ and taking absolute value, we get
    \begin{align*}
        \card{\frac{1}{\mu^{\otimes n'}(S)} \langle T_{1-\kappa}1_S, 1_A \rangle_{L^2(\mu)} - \mu^{\otimes n'}(A)} &= \frac{1}{\mu^{\otimes n'}(S)}\langle T_{1-\kappa}1_S -\mu^{\otimes n'}(S) , 1_A\rangle_{L^2(\mu)}\\
        &\leq \frac{1}{\mu^{\otimes n'}(S)}\norm{T_{1-\kappa}1_S -\mu^{\otimes n'}(S)}_{L^2(\mu)}.
    \end{align*}
    Using Lemma~\ref{lemma:noise_condition_uniform}, we have $\norm{T_{1-\kappa}1_S -\mu^{\otimes n'}(S)}_{L^2(\mu)}\leq 2^{-\Omega_{\alpha, \kappa, r, M}(1/\rho_2)}$. For the denominator, since $\vec{b}\in  \prod_{P\in \mathcal{P}} {\sf Image}(P)$, using Lemma~\ref{lemma:quasirandom_prop_2}, we have $\mu^{\otimes n'}(S) \geq \prod\limits_{P\in\mathcal{P}}\frac{1}{{\sf ord}(P)^3}-2^{-\Omega_{M,r,\alpha}(1/\rho_2)} = \Omega_{r, M}(1)$. Thus,
     \begin{align*}
        \card{\Pr_{\substack{(X, Y, Z)\sim \mu^{\otimes n'}\\ (X'', Y'', Z'')\sim T_{1-\kappa}(X, Y, Z)}}\left[(X'', Y'', Z'')\in A \mid (X, Y, Z) \in S\right] - \mu^{\otimes n'}(A)} \leq  2^{-\Omega_{\alpha, \kappa, r, M}(1/\rho_2)}.
    \end{align*}
\end{proof}

 Consider the distribution $\mathcal{B}$ on the vectors $\vec{b}$ where we first sample $X\sim \nu^{\otimes n'}$ and then output $\vec{b}(X)$. We have 
\begin{align*}
\Delta_2(\noisyfrefined, E_{\V w}) &= \Expect{\substack{(\V x, \V y, \V z)\sim \mu^{n'}}}{\noisyfrefined_{\V w}(\V x) \noisyfrefined_{\V w}(\V y) \noisyfrefined_{\V w}(\V z)}\\
& \geq\Expect{\vec{b}\sim \mathcal{B}}{\Pr[W_{\vec{b}}]\cdot 1_{\vec{b} ~\mbox{satisfies}~E'}\cdot \Expect{\substack{(\V x, \V y, \V z)\sim\mu^{n'}}}{\noisyfrefined_{\V w}(\V x) \noisyfrefined_{\V w}(\V y) \noisyfrefined_{\V w}(\V z) \mid W_{\vec{b}}} }\tag*{($\noisyfrefined$ in non-negative)}\\
& \geq\Expect{\vec{b}\sim \mathcal{B}}{\Pr[W_{\vec{b}}] \cdot 1_{\vec{b} ~\mbox{satisfies}~E'}\cdot\left( \Expect{\substack{(\V x, \V y, \V z)\\(\V x', \V y', \V z')\\ (\V x'', \V y'', \V z'')}\sim \mathcal{D}}{\truncfcoarsedecoupled_{,\V w}(\vec{b}, \V x'') \truncfcoarsedecoupled_{,\V w}(\vec{b}, \V y'') \truncfcoarsedecoupled_{,\V w}(\vec{b}, \V z'')  \mid W_{\vec{b}}} - 3\xi^{1/8}\right) }.
\end{align*}
In the last inequality, for $\vec{b}$ satisfying the event $E'$, we used
\begin{align*}
    &\Expect{(\V x, \V y, \V z)\sim \mu^{n'}}{\noisyfrefined_{\V w}(\V x) \noisyfrefined_{\V w}(\V y) \noisyfrefined_{\V w}(\V z) \mid W_{\vec{b}}} - \Expect{\substack{(\V x, \V y, \V z)\\(\V x', \V y', \V z')\\ (\V x'', \V y'', \V z'')}\sim \mathcal{D}}{{\truncfcoarsedecoupled_{,\V w}(\vec{b}, \V x'') \noisyfrefined_{\V w}(\V y) \noisyfrefined_{\V w}(\V z) \mid  W_{\vec{b}}} }\\
    &\quad\quad\quad\quad = \Expect{\substack{(\V x, \V y, \V z)\\(\V x', \V y', \V z')\\ (\V x'', \V y'', \V z'')}\sim \mathcal{D}}{(\noisyfrefined_{\V w}(\V x) - \truncfcoarsedecoupled_{,\V w}(\vec{b}, \V x''))\noisyfrefined_{\V w}(\V y) \noisyfrefined_{\V w}(\V z) \mid  W_{\vec{b}}}\\
    &\quad\quad\quad\quad \leq \Expect{\substack{(\V x, \V y, \V z)\\(\V x', \V y', \V z')\\ (\V x'', \V y'', \V z'')}\sim \mathcal{D}}{\card{\noisyfrefined_{\V w}(\V x) - \truncfcoarsedecoupled_{,\V w}(\vec{b}, \V x'')} \Bigg|  W_{\vec{b}}} \tag*{($\noisyfrefined$ is $[0,1]$ bounded)}\\
    &\quad\quad\quad\quad = \Expect{\substack{(\V x, \V y, \V z)\\(\V x', \V y', \V z')\\ (\V x'', \V y'', \V z'')}\sim \mathcal{D}}{\card{\noisyfrefined_{\V w}(\V x) - \truncfcoarsedecoupled_{,\V w}(\vec{b}, \V x'')} \Bigg|   x\in S_{\V w, \vec{b}}} \\
    &\quad\quad\quad\quad \leq \Expect{(\V x, \V x', \V x'')}{\card{\noisyfrefined_{\V w}(\V x) - \truncfcoarsedecoupled_{,\V w}(\vec{b}, \V x'')}^2 \Big| \V x\in S_{\V w, \vec{b}}}^{1/2}.\tag*{(Cauchy-Schwarz)}
\end{align*}
As $\vec{b}$ under consideration satisfies the event $E'_2$, the last expectation is upper bounded by $\xi^{1/4}$. Repeating the same thing for the $\V y$ and the $\V z$ variables, and using $[0,1]$-boundedness of $\noisyfrefined_{\V w}$ and $\truncfcoarsedecoupled_{,\V w}$, we get the required bound.

Using Lemma~\ref{lemma:quasirandom_prop_2}, for every $\vec{b}\in\prod\limits_{P\in\mathcal{P}}{\sf Image}(P)$, $\Pr[W_{\vec{b}}]$ is lower bounded by
$$\Pr[W_{\vec{b}}] \geq \prod\limits_{P\in\mathcal{P}}\frac{1}{{\sf ord}(P)^3}-2^{-\Omega_{\rho_0, m, \alpha}(1/\rho_2)}\geq \frac{1}{2}\prod\limits_{P\in\mathcal{P}}\frac{1}{{\sf ord}(P)^3}.$$

Thus, we fix an arbitrary $\vec{b}_\star$ for which the event $E'$ occurs. Note that, such a $\vec{b}_\star$ is in $\prod\limits_{P\in\mathcal{P}}{\sf Image}(P)$ by definition. Using Lemma~\ref{lemma:quasirandom_onefunc}, we have
$$\Pr_{\vec{b}\sim \mathcal{B}}[\vec{b} = \vec{b}_\star] \geq \prod\limits_{P\in\mathcal{P}}\frac{1}{{\sf ord}(P)}-2^{-\Omega_{\rho_0, m, \alpha}(1/\rho_2)}\geq \frac{1}{2}\prod\limits_{P\in\mathcal{P}}\frac{1}{{\sf ord}(P)}.$$
We get
\begingroup
\allowdisplaybreaks
\begin{align*}
\Delta_2(\noisyfrefined, E_{\V w}) &\geq\Expect{\vec{b}\sim \mathcal{B}}{\Pr[W_{\vec{b}}] \cdot 1_{\vec{b} ~\mbox{satisfies}~E'}\cdot\left( \Expect{\substack{(\V x, \V y, \V z)\\(\V x', \V y', \V z')\\ (\V x'', \V y'', \V z'')}\sim \mathcal{D}}{\truncfcoarsedecoupled_{,\V w}(\vec{b}, \V x'') \truncfcoarsedecoupled_{,\V w}(\vec{b}, \V y'') \truncfcoarsedecoupled_{,\V w}(\vec{b}, \V z'')  \mid W_{\vec{b}}} - 3\xi^{1/8}\right) }\\
& \geq \frac{1}{4}\left(\prod\limits_{P\in\mathcal{P}}\frac{1}{{\sf ord}(P)}\right)^4 \cdot \left(\Expect{\substack{(\V x, \V y, \V z)\\(\V x', \V y', \V z')\\ (\V x'', \V y'', \V z'')}\sim \mathcal{D}}{\truncfcoarsedecoupled_{,\V w}(\vec{b}_\star, \V x'') \truncfcoarsedecoupled_{,\V w}(\vec{b}_\star, \V y'') \truncfcoarsedecoupled_{,\V w}(\vec{b}_\star, \V z'')  \mid W_{\vec{b}_\star}} - 3\xi^{1/8}\right).
\end{align*}
\endgroup
We now give a lower bound on the above expectation in the next section.

\subsubsection{Applying the Invariance Principle}
\label{sec:apply_invariance_stuff}

From the last section, our goal is to estimate the following quantity:

$$\Delta_3(\truncfcoarsedecoupled_{, \V w}, \vec{b}_\star) := \Expect{\substack{(\V x, \V y, \V z)\\(\V x', \V y', \V z')\\ (\V x'', \V y'', \V z'')}\sim \mathcal{D}}{\truncfcoarsedecoupled_{,\V w}(\vec{b}_\star, \V x'') \truncfcoarsedecoupled_{,\V w}(\vec{b}_\star, \V y'') \truncfcoarsedecoupled_{,\V w}(\vec{b}_\star, \V z'')  \mid W_{\vec{b}_\star}},$$
where $\vec{b}_\star$ is such that it satisfies the event $E'$. Using Claim~\ref{claim:closeness} and $1$-boundedness of the function $\truncfcoarsedecoupled_{,\V w}$, we have

\begin{equation}
    \label{eq:decoupled_move_to_uniform}
\Delta_3(\truncfcoarsedecoupled_{, \V w}, \vec{b}_\star) \geq  \Expect{(\V x'', \V y'', \V z'')\sim \mu^{\otimes n'}}{\truncfcoarsedecoupled_{,\V w}(\vec{b}_\star, \V x'') \truncfcoarsedecoupled_{,\V w}(\vec{b}_\star, \V y'') \truncfcoarsedecoupled_{,\V w}(\vec{b}_\star, \V z'')} - \tilde{\rho}_2.
\end{equation}
Furthermore, since $\vec{b}_\star$ satisfies the events $E'_1$ and $E'_2$, we have
$$\Expect{(\V x, \V x', \V x'')\sim \mathcal{D}}{\truncfcoarsedecoupled_{,\V w}(\vec{b}_\star,\V x'') \mid \V x\in S_{\V w, \vec{b}_\star} } \geq \frac{\delta}{64}  - \xi^{1/4} .$$
Again, using Claim~\ref{claim:closeness} and the  $1$-boundedness of the function $\truncfcoarsedecoupled_{,\V w}$, we get
\begin{equation}\label{eq:dense_F}
\Expect{\V x''\sim \nu^{\otimes n'}}{\truncfcoarsedecoupled_{,\V w}(\vec{b}_\star,\V x'')  } \geq \frac{\delta}{64}  - \xi^{1/4} - \tilde{\rho}_2\geq \frac{\delta}{128}.
\end{equation}

In the following claim, we use the invariance principle to lower bound the expectation from (\ref{eq:decoupled_move_to_uniform}).
\begin{claim}
\label{claim:invariance_lb}
  For any $\vec{b}$ that satisfies the event $E'$, 
$$\Expect{(\V x'', \V y'', \V z'')\sim\mu^{\otimes n'}}{\truncfcoarsedecoupled_{,\V w}(\vec{b}, \V x'')\truncfcoarsedecoupled_{,\V w'}(\vec{b}, \V y'') \truncfcoarsedecoupled_{,\V w''}(\vec{b}, \V z'')}\geq \delta^{\Omega_{m,\alpha}(1)}.$$
\end{claim}
\begin{proof}
   Consider the function $\fcoarsedecoupled_{,\V w}(\vec{b}, \cdot) : (\Sigma^{n'}, \nu^{\otimes n'}) \rightarrow \mathbb{C}$ and its decomposition from (\ref{eq:regularity_functions}). Every function $L_{P}$ in the decomposition of the function $\fcoarsedecomposed_{w}$, we have $I_i(L_{P}) \leq \tau$ for every $i\in J$.  As the function $\fcoarsedecoupled_{,\V w}(\vec{b}, \cdot)$ is the weighted sum of these functions $L_{P}$ for $P\in \mathcal{P}$, we can express $\fcoarsedecoupled_{,\V w}(\vec{b}, \cdot) = \sum_{P\in \spn(\mathcal{P})} \theta_P \cdot L_P$ where $|\theta_P| = 1$. Applying Proposition~\ref{prop:sum_inf}, as $\card{\spn(\mathcal{P})} = O_r(1)$, we get that $I_i(\fcoarsedecoupled_{,\V w}(\vec{b},\cdot))\leq  O_{r}(\tau) = O_{\rho_1}(\tau)$ for all $i\in J$.
   
   % Furthermore, using the fact that ${\sf trunc}_{[0,1]}$ is $1$-Lipschitz, we have 
   % \begin{align*}
   %     I_i(\truncfcoarsedecoupled_{,\V w}(\vec{b},\cdot)) &= \Expect{\substack{\V y\sim \nu^{\otimes (n'-1)}\\ y', y'' \sim \nu} }{\card{ \truncfcoarsedecoupled_{,\V w}(\vec{b},\V x_{-i} = \V y, x_i = y') - \truncfcoarsedecoupled_{,\V w}(\vec{b}, \V x_{-i} = \V y, x_i = y'') }^2}\\
   %     &\lesssim \Expect{\substack{\V y\sim \nu^{\otimes(n'-1)}\\ y', y'' \sim \nu} }{\card{ \fcoarsedecoupled_{,\V w}(\vec{b}, \V x_{-i} = \V y, x_i = y') - \fcoarsedecoupled_{,\V w}(\vec{b}, \V x_{-i} = \V y, x_i = y'') }^2}\\
   %     &= I_i(\fcoarsedecoupled_{,\V w}(\vec{b}, \cdot)) \leq O_{r}(\tau).
   % \end{align*}

   	Let $\Psi\colon \mathbb{C}^3\to\mathbb{C}$ be any smooth extension of the function $\Psi'(a,b,c) = abc$ for $a,b,c$
	that have absolute value at most $1$, by which we mean it has third order derivatives bounded by $O(1)$ and  additionally it satisfies
    $$\card{\Psi(a, b, c) - \Psi(a',b',c')} \lesssim \sqrt{\card{a-a'}^2+\card{b-b'}^2+\card{c-c'}^2}$$
    for all complex numbers $a,b,c,a',b',c'\in \mathbb{C}$. Let $\mathbf{\mathcal{X}} = \{X_{1},\ldots,X_{m-1}, Y_{1}, \ldots, Y_{m-1}, Z_{1},\ldots, Z_{m-1}\}$ and $\mathbf{\mathcal{G}} = \{G_{1,x},\ldots,G_{m-1, x},\ldots, G_{1,z},\ldots,G_{m-1, z}\}$ be the ensemble of random variables and the Gaussian random variables for the distribution $\mu$, as defined in Section~\ref{sec:invariance_principle}. Let $X = (\{X_{1},\ldots,X_{m-1}), Y = (Y_{1}, \ldots, Y_{m-1}), Z=  (Z_{1},\ldots, Z_{m-1})$ and ${G}_x = (G_{1,x},\ldots,G_{m-1, x}),  {G}_y = (G_{1,y},\ldots,G_{m-1, y}), {G}_z=(G_{1,z},\ldots,G_{m-1, z})$. We now apply the invariance principle, Theorem~\ref{thm:invariance_principle}, to the expectation under consideration to move to the Gaussian space. Towards this, we work with the multi-linear representation $\multilinfcoarsedecoupled_{, \V w}: \mathbb{C}^{(m-1)n'}\rightarrow \mathbb{C}$ of the function $\fcoarsedecoupled_{,\V w}$. Also, let $\truncmultilinfcoarsedecoupled_{, \V w} = \mathrm{trunc}_{[0,1]}(\multilinfcoarsedecoupled_{, \V w})$ be its truncated version.

Consider the following quantities
\begin{align*}
    \mathbf{\Lambda}_1 &= \Expect{\substack{(\V x, \V y, \V z)\sim \mu^{\otimes n'}}}{\Psi\left(\truncfcoarsedecoupled_{,\V w}(\vec{b},\V x), \truncfcoarsedecoupled_{,\V w}(\vec{b}, \V y),  \truncfcoarsedecoupled_{,\V w}(\vec{b}, \V z)\right)}\\
    \mathbf{\Lambda}_1' &= \Expect{\substack{(\V x, \V y, \V z)\sim \mu^{\otimes n'}}}{\Psi\left(\fcoarsedecoupled_{,\V w}(\vec{b},\V x), \fcoarsedecoupled_{,\V w}(\vec{b}, \V y),  \fcoarsedecoupled_{,\V w}(\vec{b}, \V z)\right)}\\
    \mathbf{\Lambda}_2 &= \Expect{\mathbf{\mathcal{X}}}{\Psi\left(\multilinfcoarsedecoupled_{,\V w}(\vec{b}, X), \multilinfcoarsedecoupled_{,\V w}(\vec{b}, Y),  \multilinfcoarsedecoupled_{,\V w}(\vec{b}, Z)\right)}\\
     \mathbf{\Lambda}_3 &= \Expect{\mathbf{\mathcal{G}}}{\Psi\left(\multilinfcoarsedecoupled_{,\V w}(\vec{b}, G_x), \multilinfcoarsedecoupled_{,\V w}(\vec{b}, G_y),  \multilinfcoarsedecoupled_{,\V w}(\vec{b}, G_z)\right)}\\
   \mathbf{\Lambda}_3' &= \Expect{\mathbf{\mathcal{G}}}{\Psi\left(\truncmultilinfcoarsedecoupled_{,\V w}(\vec{b}, G_x), \truncmultilinfcoarsedecoupled_{,\V w}(\vec{b}, G_y),  \truncmultilinfcoarsedecoupled_{,\V w}(\vec{b}, G_z)\right)}
\end{align*}
    Note that  $\mathbf{\Lambda}'_1 =  \mathbf{\Lambda}_2$, as in $\mathbf{\Lambda}_2$, we just moved to the multi-linear representation of $\fcoarsedecoupled_{,\V w}$. We have the following claims.
   \begin{claim}\label{claim:inv_truncation}
        $\card{\mathbf{\Lambda}_1 -  \mathbf{\Lambda}'_1 }\leq O_{\rho_1}(\tilde{\rho}_2) + O(\xi^{1/4}).$
   \end{claim}
   \begin{proof}
           Using the smoothness property of $\Psi$, we have 
       \begin{align*}
          & \card{\mathbf{\Lambda}_1 -  \mathbf{\Lambda}'_1 } \\
          &~~
          \lesssim \Expect{\substack{(\V x, \V y, \V z)\sim \mu^{\otimes n'}}}{\sqrt{\card{\truncfcoarsedecoupled_{,\V w}(\vec{b},\V x) - \fcoarsedecoupled_{,\V w}(\vec{b},\V x) }^2 + \card{\truncfcoarsedecoupled_{,\V w}(\vec{b},\V y) - \fcoarsedecoupled_{,\V w}(\vec{b},\V y) }^2+ \card{\truncfcoarsedecoupled_{,\V w}(\vec{b},\V z) - \fcoarsedecoupled_{,\V w}(\vec{b},\V z) }^2}}.
           \end{align*}
       Now, using the fact that the marginal distribution of $\mu$ on each coordinate is $\mu$, by the Cauchy-Schwarz inequality,
      \begin{equation}\label{eq:inv_trunc_bound}
        \card{\mathbf{\Lambda}_1 -  \mathbf{\Lambda}'_1 } \lesssim \sqrt{\Expect{\substack{\V x\sim \nu^{\otimes n'}}}{\card{\truncfcoarsedecoupled_{,\V w}(\vec{b},\V x) - \fcoarsedecoupled_{,\V w}(\vec{b},\V x) }^2}}.
      \end{equation}
      Recall the decoupled distribution $\mathcal{D}$ from Lemma~\ref{lem:decoupled_move}. We move from the distribution $\nu^{\otimes n'}$ to the marginal of the decoupled distribution as follows. Using Claim~\ref{claim:closeness}, the statistical distance between the distribution on $X''$ in $(X, X', X'')\sim \mathcal{D}|W_{\vec{b}}$ and the distribution $x\sim \nu^{\otimes n'}$ is at most $\tilde{\rho}_2$. Thus, we may couple $X''$ in $(X, X', X'')$ with $x$ distributed according to $\nu^{\otimes n'}$ such that $\Pr[X''\neq x]\leq \tilde{\rho}_2$. Let $E$ be the event that $x\neq X''$ in the coupling. We have,
      \begin{align*}
          &\card{\Expect{\substack{\V x\sim \nu^{\otimes n'}}}{\card{\truncfcoarsedecoupled_{,\V w}(\vec{b},\V x) - \fcoarsedecoupled_{,\V w}(\vec{b},\V x) }^2} -\Expect{\substack{(X, X', X'')\sim \mathcal{D}|W_{\vec{b}}}}{\card{\truncfcoarsedecoupled_{,\V w}(\vec{b},\V X'') - \fcoarsedecoupled_{,\V w}(\vec{b},\V X'') }^2} }^2\\
          & = \Expect{\substack{(X, X', X'')\\\V x }}{\left(\card{\truncfcoarsedecoupled_{,\V w}(\vec{b},\V x) - \fcoarsedecoupled_{,\V w}(\vec{b},\V x) }^2 -\card{\truncfcoarsedecoupled_{,\V w}(\vec{b},\V X'') - \fcoarsedecoupled_{,\V w}(\vec{b},\V X'') }^2\right) 1_E} ^2\\
          &\leq \Pr[E]\cdot \Expect{\substack{(X, X', X'')\\\V x }}{\left(\card{\truncfcoarsedecoupled_{,\V w}(\vec{b},\V x) - \fcoarsedecoupled_{,\V w}(\vec{b},\V x) }^2 -\card{\truncfcoarsedecoupled_{,\V w}(\vec{b},\V X'') - \fcoarsedecoupled_{,\V w}(\vec{b},\V X'') }^2\right)} ^2\\
           &\leq \Pr[E]\cdot \Expect{\substack{(X, X', X'')\\\V x }}{\left(\card{\truncfcoarsedecoupled_{,\V w}(\vec{b},\V x) - \fcoarsedecoupled_{,\V w}(\vec{b},\V x) }^2 -\card{\truncfcoarsedecoupled_{,\V w}(\vec{b},\V X'') - \fcoarsedecoupled_{,\V w}(\vec{b},\V X'') }^2\right)^2}\\
          &\lesssim \Pr[E]\cdot \left(\Expect{}{\card{\truncfcoarsedecoupled_{,\V w}(\vec{b},\V x)}^4} + \Expect{}{ \card{\fcoarsedecoupled_{,\V w}(\vec{b},\V x) }^4} +\Expect{}{\card{\truncfcoarsedecoupled_{,\V w}(\vec{b},\V X'')}^4} + \Expect{}{\card{\fcoarsedecoupled_{,\V w}(\vec{b},\V X'') }^4}\right)
      \end{align*}
      The first and the third expectations are bounded by $1$. We bound the second expectation, and the last expectation can be bounded analogously. Note that $\card{\fcoarsedecoupled_{,\V w}(\vec{b},\V x) } \leq \sum_{P} \card{L_P(w,x)}$. We now use the fact that the ${\sf deg}(L_P(w,\cdot))$ and $\|L_P(w,\cdot)\|_2$ are bounded by $O_{\rho_1}(1)$. By H\"older's inequality and hypercontractivity, we get
      $$\Expect{}{ \card{\fcoarsedecoupled_{,\V w}(\vec{b},\V x) }^4} \lesssim_r \sum_P \|L_P(w,\cdot)\|_4^4 \lesssim_{r, \rho_1} \|L_P(w,\cdot)\|_2^4 = O_{\rho_1}(1).$$
      Thus, we get
      $$\card{\Expect{\substack{\V x\sim \nu^{\otimes n'}}}{\card{\truncfcoarsedecoupled_{,\V w}(\vec{b},\V x) - \fcoarsedecoupled_{,\V w}(\vec{b},\V x) }^2} -\Expect{\substack{(X, X', X'')\sim \mathcal{D}|W_{\vec{b}}}}{\card{\truncfcoarsedecoupled_{,\V w}(\vec{b},\V X'') - \fcoarsedecoupled_{,\V w}(\vec{b},\V X'') }^2} } = O_{\rho_1}(\tilde{\rho}_2).$$
      Hence,
    $$\card{\mathbf{\Lambda}_1 -  \mathbf{\Lambda}'_1 } \lesssim \sqrt{\Expect{\substack{(X, X', X'')\sim \mathcal{D}|W_{\vec{b}}}}{\card{\truncfcoarsedecoupled_{,\V w}(\vec{b}, X'') - \fcoarsedecoupled_{,\V w}(\vec{b}, X'') }^2}}+O_{\rho_1}(\tilde{\rho}_2).$$
    Now, using the property that $\eta(a) = \card{a-\mathrm{trunc}_{[0,1]}(a)}$ is $2$-Lipshitz from Fact~\ref{fact:trivial_lipshitz_pf},
    \begin{align*}
            &\Expect{\substack{(X, X', X'')\sim \mathcal{D}|W_{\vec{b}}}}{\card{\truncfcoarsedecoupled_{,\V w}(\vec{b}, X'') - \fcoarsedecoupled_{,\V w}(\vec{b}, X'') }^2} \\
            &= \Expect{\substack{(X, X', X'')\sim \mathcal{D}|W_{\vec{b}}}}{\eta(\fcoarsedecoupled_{,\V w}(\vec{b}, X''))^2}\\
            &= \Expect{\substack{(X, X', X'')\sim \mathcal{D}}}{\eta(\fcoarsedecoupled_{,\V w}(\vec{b}, X''))^2\mid X\in S_{w,\vec{b}}}\\
            &= \Expect{\substack{(X, X', X'')\sim \mathcal{D}}}{\eta(\noisyfrefined_{\V w}(X))^2 + \card{\noisyfrefined_{\V w}(X) - \fcoarsedecoupled_{,\V w}(\vec{b}, X'')}^2\mid X\in S_{w,\vec{b}}}\\
             &= \Expect{\substack{(X, X', X'')\sim \mathcal{D}}}{\card{\noisyfrefined_{\V w}(X) - \fcoarsedecoupled_{,\V w}(\vec{b}, X'')}^2\mid X\in S_{w,\vec{b}}} \tag*{($\eta(\noisyfrefined_{\V w}(X)) = 0$)}
    \end{align*}
    As $\vec{b}$ satisfies the vent $E'$ (more specifically, the event $E'_2$), the above expectation is at most $\xi^{1/4}$.
   \end{proof}
   \begin{claim}\label{claim:inv_gaussian} $\card {\mathbf{\Lambda}_2 - \mathbf{\Lambda}_3} = O_{\rho_1}(\tau')$ where $\tau'\rightarrow 0 $ as $\tau \rightarrow 0$. 
   \end{claim}
   \begin{proof}
       This follows from a direct application of Theorem~\ref{thm:invariance_principle} by noting that the degree and the $\ell_2$ norm of the multilinear function $\multilinfcoarsedecoupled_{,\V w}(\vec{b}, \cdot)$ is upper bounded by $O_{\rho_1}(1)$.
   \end{proof}

   \begin{claim}\label{claim:inv_gaussian_lb} $\mathbf{\Lambda}'_3 \geq \delta^{\Omega_{m,\alpha}(1)}.$
   \end{claim}
   \begin{proof}
   Recall that from (\ref{eq:inv_trunc_bound}) we derived the following in the proof of Claim~\ref{claim:inv_truncation}.
       $$\Expect{\substack{\V x\sim \nu^{\otimes n'}}}{\card{\truncfcoarsedecoupled_{,\V w}(\vec{b},\V x) - \fcoarsedecoupled_{,\V w}(\vec{b},\V x) }^2} = O_{\rho_1}(\tilde{\rho}_2) + O(\xi^{1/4}).$$
       This along with the lower bound from (\ref{eq:dense_F}), we get
       $$\Expect{\V x\sim \nu^{\otimes n'}}{\fcoarsedecoupled_{,\V w}(\vec{b}_\star,\V x)  } \geq \frac{\delta}{128} -  O_{\rho_1}(\tilde{\rho}_2) + O(\xi^{1/4})\geq \frac{\delta}{256}.$$
 By applying Theorem~\ref{thm:invariance_principle} again, we have 
\begin{align*}
\card{\Expect{\mathcal{X}}{\multilinfcoarsedecoupled_{,\V w}(\vec{b}, X)} - \Expect{\mathcal{G}}{\multilinfcoarsedecoupled_{,\V w}(\vec{b},\V G_x)}} &\leq O_{\rho_1}(\tau'),
\end{align*}
and thus,
\begin{align*}
\Expect{\mathcal{G}}{\multilinfcoarsedecoupled_{,\V w}(\vec{b},\V G_x)}\geq \frac{\delta}{256} - O_{\rho_1}(\tau')\geq \frac{\delta}{512}.
\end{align*}
The same holds for $\Expect{G_y}{\multilinfcoarsedecoupled_{,\V w}(\vec{b},  \V G_y)}$ and $\Expect{G_z}{\multilinfcoarsedecoupled_{,\V w}(\vec{b}, \V G_z)}$. Applying Theorem~\ref{thm:invariance_principle} (item $2$.), we get
\begin{align*}
\card{\Expect{\mathcal{X}}{\eta(\multilinfcoarsedecoupled_{,\V w}(\vec{b}, X))} - \Expect{\mathcal{G}}{\eta(\multilinfcoarsedecoupled_{,\V w}(\vec{b},\V G_x))}} &\leq O_{\rho_1}(\tau').
\end{align*}
As by the Cauchy-Schwarz inequality,
$$\Expect{\mathcal{X}}{\eta(\multilinfcoarsedecoupled_{,\V w}(\vec{b}, X))}^2 \leq \Expect{\substack{\V x\sim \nu^{\otimes n'}}}{\card{\truncfcoarsedecoupled_{,\V w}(\vec{b},\V x) - \fcoarsedecoupled_{,\V w}(\vec{b},\V x) }^2} = O_{\rho_1}(\tilde{\rho}_2) + O(\xi^{1/4}),$$
we get,
\begin{equation}\label{eq:trunc_gauss_small}\Expect{\mathcal{G}}{\eta(\multilinfcoarsedecoupled_{,\V w}(\vec{b},\V G_x))}^2 = O_{\rho_1}(\tilde{\rho}_2) + O(\xi^{1/4}) + O_{\rho_1}(\tau').
\end{equation}
Thus, using the fact that $\Expect{\mathcal{G}}{\multilinfcoarsedecoupled_{,\V w}(\vec{b},\V G_x)}\geq  \frac{\delta}{512}$ and the above bound, we get
\begin{align*}
\Expect{\mathcal{G}}{\truncmultilinfcoarsedecoupled_{,\V w}(\vec{b},\V G_x)}\geq \frac{\delta}{1024},
\end{align*}

Let $T$ be the covariance matrix of the ensemble $\mathcal{G}$. In Proposition~\ref{prop:positive_definite} below, we show that the matrix $T$ satisfies $T-\eta_{\alpha, |\Sigma|} I\succeq 0$, for some constant $\eta_{\alpha, |\Sigma|}>0$ that depends on $\alpha$ and $|\Sigma|$. Hence, using Theorem~\ref{thm:multi_reverse_hyp}, we have
    \begin{align*}
        \Expect{\mathcal{G}^{n'}}{ \truncmultilinfcoarsedecoupled_{,\V w}(\vec{b}, \V G_x),\truncmultilinfcoarsedecoupled_{,\V w}(\vec{b}, \V G_y), \truncmultilinfcoarsedecoupled_{,\V w}(\vec{b},  \V G_z)}\geq \left(\frac{\delta^3}{1024^3}\right)^{1/\eta_{m,\alpha}}.
    \end{align*}
   \end{proof}

 \begin{claim}\label{claim:inv_trunc_gauss}
       $\card{{\bf \Lambda}_3 - {\bf \Lambda}'_3} = O_\xi(1)$.
   \end{claim}
   \begin{proof}
       Using the smoothness property of $\Psi$, we have 
       \begin{align*}
          & \card{\mathbf{\Lambda}_3 -  \mathbf{\Lambda}'_3 } \\
          &~~
          \lesssim \Expect{\substack{\mathcal{G}}}{\sqrt{\card{\multilinfcoarsedecoupled_{,\V w}(\vec{b},G_x) - \truncmultilinfcoarsedecoupled_{,\V w}(\vec{b},G_x) }^2 + \card{\multilinfcoarsedecoupled_{,\V w}(\vec{b},G_y) - \truncmultilinfcoarsedecoupled_{,\V w}(\vec{b},G_y) }^2+ \card{\multilinfcoarsedecoupled_{,\V w}(\vec{b},G_z) - \truncmultilinfcoarsedecoupled_{,\V w}(\vec{b},G_z) }^2}}.
           \end{align*}
           Using the fact that $\sqrt{\card{a}^2+\card{b}^2+\card{c}^2} \leq (\card{a}+\card{b}+\card{c})$ for all $a, b, c\in \mathbb{C}$, we get
           \begin{align*}
          & \card{\mathbf{\Lambda}_3 -  \mathbf{\Lambda}'_3 } \\
          &~~
          \lesssim \Expect{\substack{\mathcal{G}}}{\card{\multilinfcoarsedecoupled_{,\V w}(\vec{b},G_x) - \truncmultilinfcoarsedecoupled_{,\V w}(\vec{b},G_x) } + \card{\multilinfcoarsedecoupled_{,\V w}(\vec{b},G_y) - \truncmultilinfcoarsedecoupled_{,\V w}(\vec{b},G_y) } + \card{\multilinfcoarsedecoupled_{,\V w}(\vec{b},G_z) - \truncmultilinfcoarsedecoupled_{,\V w}(\vec{b},G_z) }}.
           \end{align*}
     From (\ref{eq:trunc_gauss_small}), we have
     $$\Expect{\substack{\mathcal{G}}}{\card{\multilinfcoarsedecoupled_{,\V w}(\vec{b},G_x) - \truncmultilinfcoarsedecoupled_{,\V w}(\vec{b},G_x) }}^2  = \Expect{\substack{\mathcal{G}}}{\eta(\multilinfcoarsedecoupled_{,\V w}(\vec{b},G_x))}^2 = O_{\rho_1}(\tilde{\rho}_2) + O(\xi^{1/4}) + O_{\rho_1}(\tau') = O_\xi(1). $$
    We have the same bound for the other terms in the upper bound of $\card{\mathbf{\Lambda}_3 -  \mathbf{\Lambda}'_3 }$ above, and hence 
    \[
    \card{\mathbf{\Lambda}_3 -  \mathbf{\Lambda}'_3 } = O_{\xi}(1).
    \qedhere
    \]
   \end{proof}

   Using Claims~\ref{claim:inv_truncation},~\ref{claim:inv_gaussian},~\ref{claim:inv_trunc_gauss}, and ~\ref{claim:inv_gaussian_lb}, we have the following.
   \begin{align*}
   \Expect{(\V x'', \V y'', \V z'')\sim\mu^{\otimes n'}}{\truncfcoarsedecoupled_{,\V w}(\vec{b}, \V x'')\truncfcoarsedecoupled_{,\V w'}(\vec{b}, \V y'') \truncfcoarsedecoupled_{,\V w''}(\vec{b}, \V z'')}&\geq \delta^{\Omega_{m,\alpha}(1)}- O_{\rho_1}(\tilde{\rho}_2) - O_{\rho_1}(\tau') -O_{\xi}(1)\\
   &\geq \delta^{\Omega_{m,\alpha}(1)},
   \end{align*}
   as required, where the last inequality follows from the relations (\ref{eq:parameters}).
\end{proof}

We are now ready to prove the main theorem for resilient functions.
\begin{lemma}\label{lemma:mainthm_resilient_f}
    Suppose $f:\Sigma^n \rightarrow \{0,1\}$ is a $(B, 1/2)$ resilient function for $B \gg \Omega_{\rho_1}(1)$, then $\Delta_1(f)\geq  \alpha^{B+1} \Omega_{\rho_0}(\delta^{\Omega_{m,\alpha}(1)})$.
\end{lemma}
\begin{proof}
    Apply the arithmetic regularity lemma (Lemma~\ref{lem:arithmetic_reg_high_rank}) to the function $f$ and the distribution $\mu$ for $\xi\ll \delta,\alpha, m^{-1}$. We use the notations from (\ref{eq:regularity_functions}) and the parameters satisfying (\ref{eq:parameters}). From (\ref{eq:Delta_1_to_2}), we have
    
    $$\Delta_1(f) \geq \Delta_2(\noisyfrefined)  - 3\rho_3.$$

    Next, as in Section~\ref{sec:remove_inf}, we randomly restrict at most $B$ coordinates to a string $\V w$ so that in the restricted function $\fcoarsedecomposed_{\V w}$, all the influences of the restricted functions $L_P$ in its decomposition are $\tau  \ll O_{\rho_0}(1)$ (along with $\V w$ satisfying the event $E$). For this restriction, we get,
    
    $$\Delta_2(\noisyfrefined) \geq \alpha^{B+1}\cdot \Delta_2(\noisyfrefined, E_{\V w}).$$
   From Section~\ref{sec:move_to_decoulped}, for some vector $\vec{b}_\star$ satisfying the event $E'$, we get

    $$ \Delta_2(\noisyfrefined, E_{\V w}) \geq \Omega_{\rho_0}(1)\left(\Delta_3(\truncfcoarsedecoupled_{, \V w}, \vec{b}_\star) - 3\xi^{1/8}\right)$$
    Finally, using (\ref{eq:decoupled_move_to_uniform}) and Claim~\ref{claim:invariance_lb}, we have

$$\Delta_3(\truncfcoarsedecoupled_{, \V w}, \vec{b}_\star)\geq \delta^{\Omega_{m,\alpha}(1)} -\tilde{\rho}_2,$$
where $\tilde{\rho}_2$ is from Claim~\ref{claim:closeness}. Combining all these, we get
\begin{align*}
\Delta_1(f) &\geq  \alpha^{B+1} \Omega_{\rho_0}(1)\left( \delta^{\Omega_{m,\alpha}(1)} -\tilde{\rho}_2 - 3\xi^{1/8}\right) - 3\rho_3\\
& = \alpha^{B+1} \Omega_{\rho_0}(\delta^{\Omega_{m,\alpha}(1)}),
\end{align*}
where in the last equality, we use the relations (\ref{eq:parameters}) among different parameters. 
\end{proof}

The proof of the main theorem follows from Lemma~\ref{claim:making_f_resilient} and Lemma~\ref{lemma:mainthm_resilient_f}.\\

\noindent {\bf Proof of Theorem~\ref{thm:main_2}:} We start with $f:\Sigma^n \rightarrow \{0,1\}$ with density $\delta>0$. Using Claim~\ref{claim:influence_removal}, in order to move to a restricted function where all influences of the decomposed function are at most $\tau$, and still have the average of the restricted function being $\Omega(\delta)$, we need $f$ to be $(B, 1/2)$-resilient where $B = O_{|\mathcal{P}|, |\Sigma|, D, C,\eta, \tau}(1) = O_{\rho_1}(1)$. Using Claim~\ref{claim:making_f_resilient}, if $f$ is not $(B, 1/2)$-resilient to begin with, we first make $f$ resilient by losing a multiplicative factor of $2^{-\Omega_{B, \delta, \alpha}(1)} =  2^{-\Omega_{\rho_1}(1)}$ towards estimating $\Delta_1(f)$. Using the estimate for ($B, 1/2)$ resilient functions from Lemma~\ref{lemma:mainthm_resilient_f}, we have

\begin{align*}
\Delta_1(f) &\geq 2^{-\Omega_{\rho_1}(1)}\Omega_{\rho_0}(\delta^{\Omega_{m,\alpha}(1)}).
\end{align*}
thereby finishing the proof of the main theorem. $\qedsymbol$\\

The following proposition shows that for $\mu$ such that ${\sf supp}(\mu)$ has no $\mathbb{Z}$-embedding, the corresponding covariance matrix from the proof of Claim~\ref{claim:invariance_lb} is positive definite. 

\begin{proposition}\label{prop:positive_definite}
Let $\mu$ be a distribution on $\Sigma^3$ such that ${\sf supp}(\mu)$ has no $\mathbb{Z}$-embedding and let $\alpha>0$ be the minimum non-zero probability of an atom in $\mu$. Let $\mu_1,\mu_2,\mu_3$ denote the marginals of $\mu$ on the first, second, and third coordinates, respectively. Assume that each marginal $\mu_i$ has full support on $\Sigma$. For
$i\in\{1,2,3\}$, let
\[
f^{(i)}_0=1,f^{(i)}_1,\ldots,f^{(i)}_{q-1}
\]
be an orthonormal basis of $L^2(\Sigma,\mu_i)$ where $q=|\Sigma|$. Define the random vector
\[
W=
\bigl(
f^{(1)}_1(X),\ldots,f^{(1)}_{q-1}(X),
f^{(2)}_1(Y),\ldots,f^{(2)}_{q-1}(Y),
f^{(3)}_1(Z),\ldots,f^{(3)}_{q-1}(Z)
\bigr)
\in\mathbb{C}^{3(q-1)},
\]
and let $M=\operatorname{Cov}(W)$ be the covariance matrix of $W$, i.e.,  $M= \Expect{\mu}{ WW^* }$ where $^*$ denotes conjugate transpose. Then $M$ is Hermitian positive definite, i.e., $M-\eta_{\alpha, |\Sigma|}I\succeq 0$ for some $\eta_{\alpha, |\Sigma|}>0$.
\end{proposition}

\begin{proof}
Let
\[ \theta= (\alpha_1,\ldots,\alpha_{q-1}, \beta_1,\ldots,\beta_{q-1}, \gamma_1,\ldots,\gamma_{q-1})^{\mathsf T} \in\mathbb{C}^{3(q-1)}. \] 
Define complex-valued functions $g,h,k:\Sigma\to\mathbb{C}$ by 
\[ g(x) = \sum_{j=1}^{q-1}\overline{\alpha_j}f^{(1)}_j(x), \quad\quad h(y) = \sum_{j=1}^{q-1}\overline{\beta_j}f^{(2)}_j(y),  \mbox{ and } \quad\quad k(z) = \sum_{j=1}^{q-1}\overline{\gamma_j}f^{(3)}_j(z). \] 
Then 
\[ \theta^*W=g(X)+h(Y)+k(Z). \] 
Consequently, 
\[ \begin{aligned} \theta^*M\theta &= \Expect{\mu}{ \left| \theta^*W \right|^2 }\\ &=  \Expect{\mu}{\left| g(X)+h(Y)+k(Z)  \right|^2}. \end{aligned} \] 

It follows immediately that $M$ is Hermitian positive semidefinite. We show that it is strictly positive definite. Suppose, toward a contradiction, that there exists $\theta\neq 0$ such that 
\[ \theta^*M\theta=0. \] 
Then $g(X)+h(Y)+k(Z)$ is $0$. Hence $g(x)+h(y)+k(z)=0$ for every $(x,y,z)$ in the support of $\mu$. 
Write 
\[ g=g_{\mathrm R}+{\bf i}g_{\mathrm I},\qquad h=h_{\mathrm R}+{\bf i}h_{\mathrm I},\qquad k=k_{\mathrm R}+{\bf i}k_{\mathrm I}, \]
where all real and imaginary parts are real-valued functions. Taking real and imaginary parts gives 
\[ g_{\mathrm R}(x)+h_{\mathrm R}(y)+k_{\mathrm R}(z) = 0 \]
and 
\[ g_{\mathrm I}(x)+h_{\mathrm I}(y)+k_{\mathrm I}(z) =0 \]
for every $(x,y,z)\in {\sf supp}(\mu)$. For every $j\geq 1$, the function $f^{(i)}_j$ is orthogonal to the constant function $1$ in $L^2(\Sigma,\mu_i)$. Therefore, 
\[ \mathbb{E}_{\mu_1}[g]=0,\qquad \mathbb{E}_{\mu_2}[h]=0,\qquad \mathbb{E}_{\mu_3}[k]=0. \]
In particular, if $g,h,k$ were all constant, then they would all be identically zero. Since the basis functions are linearly independent and $\theta\neq 0$, this is impossible. Hence at least one of $g,h,k$ is nonconstant. It follows that at least one of the two real triples 
\[ (g_{\mathrm R},h_{\mathrm R},k_{\mathrm R}) \qquad\text{or}\qquad (g_{\mathrm I},h_{\mathrm I},k_{\mathrm I}) \]
contains a nonconstant function. Choose such a triple and denote it by $(G,H,K)$. Then 
\[ G(x)+H(y)+K(z)=0 \qquad \text{for every }(x,y,z)\in {\sf supp}(\mu), \]
and at least one of $G,H,K$ is nonconstant. Consider the homogeneous linear system 
\[ A(x)+B(y)+C(z)=0 \qquad \text{for every }(x,y,z)\in {\sf supp}(\mu). \]
Its coefficient matrix has integer entries. The tuple $(G,H,K)$ is a real solution lying outside the linear subspace in which $A,B,C$ are all constant. Since the kernel of a matrix with rational entries has a basis over $\mathbb{Q}$, the real solution space is the real span of its rational solutions. Therefore, if every rational solution had $A,B,C$ all constant, then every real solution would also have $A,B,C$ all constant. This contradicts the existence of $(G,H,K)$. Hence there exists a rational solution \[ (A,B,C) \] such that at least one of $A,B,C$ is nonconstant and 
\[ A(x)+B(y)+C(z)=0 \qquad \text{for every }(x,y,z)\in {\sf supp}(\mu). \]
Multiplying by a common denominator, we may assume that $A,B,C$ are integer-valued. Thus $(A, B, C)$ is a nonconstant embedding of ${\sf supp}(\mu)$ into $\mathbb{Z}$, contradicting the hypothesis. Therefore, \[ \theta^*M\theta>0 \qquad \text{for every nonzero } \theta\in\mathbb{C}^{3(q-1)}. \] Hence $M$ is Hermitian positive definite.

From this, standard compactness argument shows that $\theta^*M\theta$ is indeed bounded away from $0$ as a function of $|\Sigma|$ and $\alpha$ for every nonzero $\theta\in\mathbb{C}^{3(q-1)}$ with $\norm{\theta}_2 = 1$.
\end{proof}

\section{Discussion and Open Problems}\label{sec:discussion}
%\subsection{Generalized Gaussian Problems}
A key conceptual challenge in extending Theorem~\ref{thm:main} to the setting of restricted $3$-APs is the corresponding problem in Gaussian space, which we now describe, and for simplicity we fix $p=3$. 
Let $\omega_3$ be the primitive root of unity of order $3$, and let $\mu$ be the uniform distribution over $\{(x,x+a,x+2a)~|~x\in\mathbb{F}_3, a\in\{0,1\}\}$. Consider the $6$ by $6$ correlation matrix, $M$, whose rows and columns are indexed by $\{1,2\}\times\{1,2,3\}$ and are given by
\[
M((a,i),(b,j))
=\Expect{(x_1,x_2,x_3)\sim\mu}{\omega^{a x_i}\omega^{-b x_j}}.
\]
Explicitly, using the lexicographic ordering on $\{1,2\}\times\{1,2,3\}$, $M$ can be written as
\begin{equation}\label{matrix:M}
M = \begin{pmatrix}
1 & \frac{1+\omega}{2} & \frac{1+\omega^2}{2} & 0 & 0 & 0 \\
\frac{1+\omega^2}{2} & 1 & \frac{1+\omega}{2} & 0 & 0 & 0 \\
\frac{1+\omega}{2} & \frac{1+\omega^2}{2} & 1 & 0 & 0 & 0 \\
0 & 0 & 0 & 1 & \frac{1+\omega^2}{2} & \frac{1+\omega}{2} \\
0 & 0 & 0 & \frac{1+\omega}{2} & 1 & \frac{1+\omega^2}{2} \\
0 & 0 & 0 & \frac{1+\omega^2}{2} & \frac{1+\omega}{2} & 1
\end{pmatrix}.
\end{equation}
It turns out that this matrix is singular, and this presents difficulties. The result of~\cite{ChenDafnisPaouris2015HolderGaussian} can be used to bound expectations of the form
\[
\Expect{(G_1,\ldots,G_k)}{\prod\limits_{i=1}^{k}f_i(G_i)},
\]
where $G_1,\ldots,G_k$ jointly distributed Gaussian random variables with non-singular covariance matrix, and each $f_i$ is non-negative bounded function. Strictly speaking, their result is phrased for real-valued Gaussian distributions, but it also holds for complex Gaussians. 

For this multi-function version, the non-singularity assumption is indeed necessary, and we give an example. Multiplying $M$ by $(1,\omega,\omega^2,1,\omega^2,\omega)$ shows that $M$ is singular, and this implies that complex Gaussians $(G_1,\ldots,G_6)$ with covariance matrix $M$ satisfy the linear relation $G_1+\omega^2 G_2 + \omega G_3 = 0$. Take a constant $\eps>0$ and define $f_1(G)=f_2(G) = 1_{\card{G}\leq \eps}$ and 
$f_3(G) = 1_{\card{G}\geq 3\eps}$. Then for $G_1,G_2,G_3$ as above, if $f_1(G_1)=f_2(G_2)=1$, then we get that
\[
\card{G_3}=\card{G_1+\omega^2 G_2}\leq 2\eps< 3\eps,
\]
so $f_3(G_3) = 0$. This shows that $\Expect{G_1,G_2,G_3}{f_1(G_1)f_2(G_2)f_3(G_3)} = 0$.

We are not aware of any similar example where the functions $f_1,f_2,f_3$ are all the same. In fact, we believe the following statement should be true:
\begin{conj}\label{conj:equil_triangle}
    For all $\alpha>0$, there exists $\beta>0$, such that for all $n\in\mathbb{N}$, if $f\colon \mathbb{C}^{2n}\to\{0,1\}$ has measure at least $\alpha$, then
    \[
    \Expect{(G_1,\ldots,G_6)}{f(G_1,G_4)f(G_2,G_5)f(G_3,G_6)}\geq \beta.
    \]
    Here, $G_1,\ldots,G_6$ is the $n$-dimensional Gaussian distribution, where each coordinate is sampled independently as jointly distributed centered complex Gaussian random variables with the covariance matrix $M$ from (\ref{matrix:M}).
\end{conj}
We remark that we know how to prove dimension-dependent bounds for Conjecture~\ref{conj:equil_triangle}, and the real challenge here is to establish bounds that hold for all dimensions. 

\begin{remark}
    There may be a more general phenomenon underlying statements along the lines of Conjecture~\ref{conj:equil_triangle} for matrices $M$ whose main diagonal are all $1$'s. If $M$ is singular and there is $v\in {\sf ker}(M)$ not orthogonal to the all $1$ vector, then even the single function can produce an expectation equal to $0$, so the analogous statement to Conjecture~\ref{conj:equil_triangle} fails for $M$. It is feasible that this is the only obstruction, and that if a covariance matrix $M$ has that ${\sf ker}(M)\subseteq{\sf span}(\vec{1})^{\bot}$, then $M$ admits a counting result in the spirit of Conjecture~\ref{conj:equil_triangle}. If true, such a result would be very interesting in our opinion, as it gives a clear distinction between the single-function version and the multi-function version.
\end{remark}

\bibliography{ref}
\bibliographystyle{alpha}

\appendix

\section{Proof of Theorem~\ref{thm:multi_reverse_hyp}}
\label{sec:GRH_complex}

We begin by stating the multidimensional version of the Gaussian reverse hypercontractivity theorem from~\cite{ChenDafnisPaouris2015HolderGaussian, Mos}.

\begin{thm}\label{thm:multi_reverse_hyp_real}(\cite{ChenDafnisPaouris2015HolderGaussian, Mos})
	Let $\eta>0$ and let $\mathbf{\mathcal{G}} = (\mathcal{G}_1,\ldots,\mathcal{G}_{\ell})$ be a jointly Gaussian collection of $\ell$ random vectors such that:
	\begin{enumerate}
		\item For each $i \in [\ell]$, $\mathcal{G}_i = \left(G_{i,1}, \ldots, G_{i,n}\right)$ is a random vector distributed as $n$ independent $\mathcal{N}(0,1)$ real Gaussians.
		\item  The $n\ell\times n\ell$ covariance matrix $P = \Expect{}{\mathcal{G}\mathcal{G}^T}$ of the Gaussian ensemble satisfies  $P-\eta I\succeq 0$, for some $\eta>0$. 
	\end{enumerate}
	Then, for all functions $f_1,\ldots,f_{\ell} \in L^2(\mathbb{R}^n,\gamma^n)$, where $\gamma$ is the standard $\mathcal{N}(0,1)$ Gaussian measure, such that $f_{1},\ldots,f_{\ell} : \mathbb{R}^n \to [0,1]$ and
	$\Expect{}{f_{j}(\mathcal{G}_{j})} = \mu_j$,
	we have:
	\[
	\Expect{}{\prod_{j=1}^\ell f_{j}(\mathcal{G}_{j})} \geq
	\left( \prod_{j=1}^\ell \mu_j \right)^{1/\eta}.
	\]
\end{thm}

We show how to derive Theorem~\ref{thm:multi_reverse_hyp} from the above theorem by reducing the complex case to the real case.  
%by treating the complex input $z=x+{\bf i}y\in\mathbb{C}^n$, as $(x,y)\in\mathbb{R}^{2n}$.  
Towards this end, we use the following definition.

\begin{definition}\label{def:realification}
    For \(H\in\mathbb{C}^{m\times m}\), define its {\em realification} by
\[
\mathcal{R}(H)
:=
\begin{pmatrix}
\operatorname{Re}H & -\operatorname{Im}H\\
\operatorname{Im}H &  \operatorname{Re}H
\end{pmatrix}
\in\mathbb{R}^{2m\times 2m}.
\]
\end{definition}
The key feature of $\mathcal{R}(H)$ is that it correctly maps the inner product over $\mathbb{C}^{m}$ to the inner product over $\mathbb{R}^{2m}$. In particular, we have the following lemma asserting that if a Hermitian matrix $H$ is positive definite, then $\mathcal{R}(H)$ is also positive definite.

\begin{lemma}\label{lem:complex_real_psd}
If $H\in\mathbb{C}^{m\times m}$ is Hermitian, then 
$\mathcal{R}(H)\in\mathbb{R}^{2m\times 2m}$ is real symmetric and
for every $x,y\in\mathbb{R}^m$
\[
(x+{\bf i}y)^*H(x+{\bf i}y)
=
\begin{pmatrix}
x\\y
\end{pmatrix}^{\mathsf T}
\mathcal{R}(H)
\begin{pmatrix}
x\\y
\end{pmatrix}.
\]
Consequently, $
H\succeq 0
\Longleftrightarrow
\mathcal{R}(H)\succeq 0,
$
and
$
H\succ 0
\Longleftrightarrow
\mathcal{R}(H)\succ 0.
$
\end{lemma}
\begin{proof}
Write
\[
H=M+{\bf i}N,
\qquad
M=\operatorname{Re}H,
\qquad
N=\operatorname{Im}H.
\]
Since \(H\) is Hermitian,
\[
M^{\mathsf T}=M
\qquad\qquad
N^{\mathsf T}=-N,
\qquad\qquad
\mathcal{R}(H)
=
\begin{pmatrix}
M&-N\\
N&M
\end{pmatrix}.
\]
% It follows that
% \[
% \mathcal{R}(H)
% =
% \begin{pmatrix}
% M&-N\\
% N&M
% \end{pmatrix}
% \]
% is a real symmetric matrix. 
Let \(z=x+{\bf i}y\), with \(x,y\in\mathbb{R}^m\). A direct
calculation gives
\[
\begin{aligned}
z^*Hz
=
(x^{\mathsf T}-{\bf i}y^{\mathsf T})
(M+{\bf i}N)
(x+{\bf i}y)
&=
x^{\mathsf T}Mx-x^{\mathsf T}Ny
+y^{\mathsf T}Nx+y^{\mathsf T}My\\
&=
\begin{pmatrix}
x\\y
\end{pmatrix}^{\mathsf T}
\begin{pmatrix}
M&-N\\
N&M
\end{pmatrix}
\begin{pmatrix}
x\\y
\end{pmatrix},
\end{aligned}
\]
giving the first assertion of the lemma. The second assertion is immediate by definitions.
\end{proof}

We are now ready to prove Theorem~\ref{thm:multi_reverse_hyp}, which we restate here.

{\MGRH*}

\begin{proof}
Write
\[
\mathcal{G}_j=X_j+{\bf i}Y_j,
\qquad X_j,Y_j\in\mathbb{R}^n,
\]
and define the normalized realification
\[
\widetilde{\mathcal{G}}_j
:=
\sqrt{2}
\begin{pmatrix}
X_j\\
Y_j
\end{pmatrix}
\in\mathbb{R}^{2n},
\qquad
\widetilde{\mathcal{G}}
=
(\widetilde{\mathcal{G}}_1,\ldots,
 \widetilde{\mathcal{G}}_\ell).
\]
Since the coordinates of $\mathcal{G}_j$ are independent
$\mathcal{CN}(0,1)$ random variables, the coordinates of
$\widetilde{\mathcal{G}}_j$ are independent and mean $0$. The normalization $\sqrt{2}$ makes them $\mathcal{N}(0,1)$ random variables.

Write \(P=(P_{jk})_{j,k=1}^{\ell}\), where each \(P_{jk}\) is an
\(n\times n\) complex matrix, and let $\widetilde{P} :=\operatorname{Cov}(\widetilde{\mathcal{G}})
$
be the real covariance matrix of the realified collection. Then we have
%The joint properness of $\mathcal{G}$ gives
\begin{equation}\label{eq:realified_covariance}
\widetilde P_{jk}
=
\begin{pmatrix}
\operatorname{Re}P_{jk} & -\operatorname{Im}P_{jk}\\
\operatorname{Im}P_{jk} &  \operatorname{Re}P_{jk}
\end{pmatrix}.
\end{equation}
Indeed, by definition we have 
\[
P_{jk}=\Expect{}{\mathcal{G}_j\mathcal{G}_k^*} = \Expect{}{(X_j+{\bf i} Y_j)(X_k^T-{\bf i}Y_k^T)} = \Expect{}{X_jX_k^T} + \Expect{}{Y_jY_k^T} + i \Expect{}{Y_jX_k^T} - i \Expect{}{X_jY_k^T}.
\]
On the other hand, the properness of the collection gives
\[
0=\Expect{}{\mathcal{G}_j\mathcal{G}_k^T} = \Expect{}{(X_j+{\bf i} Y_j)(X_k+{\bf i}Y_k)^T} = \Expect{}{X_jX_k^T} - \Expect{}{Y_jY_k^T} + {\bf i} \Expect{}{Y_jX_k^T} + {\bf i} \Expect{}{X_jY_k^T},
\]
implying $\Expect{}{X_jX_k^T} = \Expect{}{Y_jY_k^T}$ and $\Expect{}{Y_jX_k^T} =- \Expect{}{X_jY_k^T}$. Plugging this above gives
\[
\operatorname{Re}P_{jk} = \Expect{}{X_jX_k^T} + \Expect{}{Y_jY_k^T} = 2\Expect{}{X_jX_k^T} = 2\Expect{}{Y_jY_k^T},
\]
and
\[
\operatorname{Im}P_{jk}  =  \Expect{}{Y_jX_k^T} -  \Expect{}{X_jY_k^T} = -2  \Expect{}{X_jY_k^T} = 2\Expect{}{Y_jX_k^T}.
\]
This implies~\eqref{eq:realified_covariance}, and so $\widetilde P=\mathcal{R}(P)$, where $\mathcal{R}$ denotes
blockwise realification. From Lemma~\ref{lem:complex_real_psd}, we conclude that hence
\[
\widetilde P-\eta I_{2n\ell}
=
\mathcal{R}\left(P-\eta I_{n\ell}\right)
\succeq 0.
\]
Therefore, $\widetilde{\mathcal{G}}$ satisfies the covariance
hypothesis of the real reverse Gaussian hypercontractivity theorem, Theorem ~\ref{thm:multi_reverse_hyp_real}, with $n$ replaced by $2n$.

For $x,y\in\mathbb{R}^n$, define
\[
\widetilde f_j(x,y)
:=
f_j\left(\frac{x+{\bf i}y}{\sqrt{2}}\right).
\]
Then
\[
\widetilde f_j(\widetilde{\mathcal{G}}_j)
=
f_j(\mathcal{G}_j)
\quad\text{and}\quad
\Expect{}{
\widetilde f_j(\widetilde{\mathcal{G}}_j)}
=
\mu_j.
\]
Applying Theorem ~\ref{thm:multi_reverse_hyp_real} to
\(\widetilde{\mathcal{G}}_1,\ldots,\widetilde{\mathcal{G}}_\ell\)
therefore yields
\[
\Expect{\mathcal{G}}{
    \prod_{j=1}^{\ell}f_j(\mathcal{G}_j)
} = \Expect{\widetilde{\mathcal{G}}}{
    \prod_{j=1}^{\ell}\widetilde{f}_j(\widetilde{\mathcal{G}}_j)
}
\geq
\left(
    \prod_{j=1}^{\ell}\mu_j
\right)^{1/\eta},
\]
as required.
    
\end{proof}

\section{Proof of Claim~\ref{claim:influence_removal}}

   The following lemma will be crucial in proving the claim.

    \begin{lemma}
    \label{lemma:split_f_total_inf}
        Fix $i\in [n]$ and a function $f: (\Sigma^n, \nu^{\otimes n}) \rightarrow \mathbb{C}$. Then,
        $$\Expect{a\sim \nu}{I(f_{x_i \rightarrow a})} = I(f) - I_i(f).$$
    \end{lemma}
\begin{proof}
Let $|\Sigma|= m$ and consider the Fourier basis $\phi_0 ={\bf 1}, \phi_1, \ldots, \phi_{m-1}$ for the space $L^2(\Sigma, \nu)$. Using Proposition~\ref{prop:fourier_n}, $f$ has the following expansion
$$f(\V x) = \sum_{\alpha \in \mathbb{N}^n_{<m}} \hat{f}(\alpha) \phi_{\alpha}(\V x).$$
Thus, for $\V x\in \Sigma^{n-1}$, the expansion of the function $f_{x_i \rightarrow a}(\V x)$ is
$$f_{x_i \rightarrow a}(\V x) = \sum_{\alpha \in \mathbb{N}^{n-1}_{<m}} \sum_{t=0}^{m-1} \phi_t(a)\hat{f}(\alpha^{i\rightarrow t}) \phi_{\alpha}(\V x),$$
where $\alpha^{i\rightarrow t} \in \mathbb{N}^n_{<m}$ with $(\alpha^{i\rightarrow t})_i = t$ and the other coordinates same as $\alpha$. Now using Proposition~\ref{prop:total_inf},
\begin{align*}
    I(f_{x_i \rightarrow a}) &= \sum_{\alpha \in \mathbb{N}^{n-1}_{<m}} \card{\alpha}\card{\sum_{t=0}^{m-1} \phi_t(a)\hat{f}(\alpha^{i\rightarrow t})}^2\\
    & = \sum_{\alpha \in \mathbb{N}^{n-1}_{<m}} \card{\alpha}\left( \sum_{t} \card{\hat{f}(\alpha^{i\rightarrow t})}^2 + \sum_{t\neq t'} \phi_t(a)\overline{\phi_t'(a)}\hat{f}(\alpha^{i\rightarrow t})\overline{\hat{f}(\alpha^{i\rightarrow t'})} \right).
\end{align*}
Using the fact that the characters are orthogonal, and hence $\Expect{a\sim \nu}{\phi_t(a)\overline{\phi_t'(a)}} = 0$ for $t\neq t'$, we get,
\begin{align*}
   \Expect{a\sim \nu}{ I(f_{x_i \rightarrow a})} &= \sum_{\alpha \in \mathbb{N}^{n-1}_{<m}} \card{\alpha}\left( \sum_{t} \card{\hat{f}(\alpha^{i\rightarrow t})}^2 \right)\\
   &= \sum_{\alpha \in \mathbb{N}^{n}_{<m}} \card{\alpha} \card{\hat{f}(\alpha)}^2  - \sum_{\substack{\alpha \in \mathbb{N}^{n}_{<m}\\ \alpha_i \neq 0}} \card{\hat{f}(\alpha)}^2 \\
   & = I(f) - I_i(f),
\end{align*}
as required.
\end{proof}

We restate the claim here.

\jones*
\begin{proof}
We build a decision tree $\mathcal{D}$ as follows. Start with the root note and associate it with functions $\{L_P\mid P\in \spn(\mathcal{P})\}$. At each node $N$ in the tree $\mathcal{D}$, we pick a variable $i\in [n]$ and add $|\Sigma|$ many child nodes to it. Suppose node $N$ is associated with functions $\{(L_P)|_{N} \mid P\in \spn(\mathcal{P})\}$, then the children nodes will be associated with the restriction $x_i \rightarrow a$ of functions $\{(L_P)|_{N} \mid P\in \spn(\mathcal{P})\}$ for $a\in \Sigma$.

    We prove the statement by starting with the tree with just the root node. Each time, if for most leaf nodes $T$, at least one of the functions $\{(L_P)|_T\}$ has an influential variable, then we split the node $T$ into $\Sigma$ nodes each with the restriction $x_i \rightarrow a$ for $a\in \Sigma$. We show that after $O(1)$ many steps, most of the leaf nodes in $\mathcal{D}$ are such that all functions $\{(L_P)|_T\}$ associated with $T$ have all influences at most $\tau$. \\

 For a given decision tree, consider the following distribution, denoted by $\nu(\mathcal{D})$ on the leaf nodes of $\mathcal{D}$. Start with a root node of $\mathcal{D}$, at each step sample the $a^{th}$ child of the current node according to the probability $\nu(a)$ and keep going down until the process hits a leaf node $T$. Consider the following energy function associated with a given decision tree:
    $$\varphi(\mathcal{D}) = \Expect{T\sim \nu(\mathcal{D})}{\sum_{P\in \spn(\mathcal{P})} I[(L_P)|_{T}]},$$
    where the expectation is over the leaf nodes $T$ of $\mathcal{D}$, weighted according to the distribution $\nu$ , and $I[f]$ is the total influence of $f$. Note that if $\mathcal{D}_0$ is just the root node, then $\varphi(\mathcal{D}_0) = O_{|\mathcal{P}|, |\Sigma|, C, D}(1)$, as $\mathrm{deg}(L_P)\leq D$ and $\|L\|_2\leq C$ and hence $I[L_P] \leq CD$.

    Consider any decision tree $\mathcal{D}$. Suppose that in $\mathcal{D}$, for at least $\eta$ fraction of leaf nodes $T$, weighed according to $\nu(\mathcal{D})$, at least one of the associated functions has some variable $i$ with influence at least $\tau$. In this case, we can split each of these leaf nodes according to the respective influential variable. Using Lemma~\ref{lemma:split_f_total_inf}, the new tree $\mathcal{D}'$ has the property
    $$ \varphi(\mathcal{D}') \leq \varphi(\mathcal{D}) - \eta\cdot \tau.$$

    From the above, we concluded that if for at least $\eta$ fraction of leaf nodes $T$, at least one of the associated functions has a variable $i$ with influence at least $\tau$, then the energy of the next tree (generated by splitting each of these leaf nodes according to the respective influential variable) decreases by at least $\eta\cdot \tau$. Hence, after $k := O_{|\mathcal{P}|, |\Sigma|, C, D}(\frac{1}{\tau \eta})$ many steps, we get a tree $\mathcal{D}^{\mathrm{final}}$ with the property that for at least $(1-\eta)$ fraction of leaf nodes $T$ in $\mathcal{D}^{\mathrm{final}}$, the associated functions $\{(L_P)|_{T}\}$ have all the influences at most $\tau$.

    In order to prove the claim,
    it will be convenient to have a homogeneous tree where all nodes at a given level $\ell$ are split based on a fixed variable $i(\ell)$. We can modify the above process of getting a decision tree to ensure that the tree is always homogeneous: At each iteration $j$, suppose in the above process we split the nodes based on the variables $x_{j_1}, x_{j_2}, \ldots, x_{j_t}$, then in the modified process we split all the nodes at last level based on all the values of $x_{j_1}, x_{j_2}, \ldots, x_{j_t}$ variables (this will increase the depth of the homogeneous tree by $t$, assuming all $x_{j_m}$s are distinct). Similar to the above analysis, the energy of the new homogeneous tree is at most the energy of the original non-homogeneous tree. Hence, this process will also stop after $k= O_{|\mathcal{P}|, |\Sigma|, C, D}(\frac{1}{\tau \eta})$ iterations. However, note that in each iteration, the depth of the tree increases by up to the number of leaf nodes in the previous iteration. If we let $d(\ell)$ denote the depth of the homogeneous tree before the $\ell^{th}$ iteration, then we have
    $$d(0) = 1, \quad d(\ell+1) \leq d(\ell)+ 2^{d(\ell)}.$$ 
    Solving this recurrence gives $d(k) \leq 2\uparrow\uparrow k-1$. 

    If we let the restricted coordinates in the above homogeneous tree by the set $J\subseteq [n]$, then $|J| = O_{|\mathcal{P}|, |\Sigma|, C, D, \eta, \tau}(1)$ and the guarantee on the decision tree implies that with probability at least $1-\eta$ over $\V w\sim \nu^{\otimes |J|}$, we have

    $$\max_{P\in \spn(\mathcal{P})} \{ \max_{i}I_i[(L_P)_{J\rightarrow \V w}] \}\leq \tau,$$
    as required
\end{proof}

\end{document}